\documentclass[11pt]{article} 

\usepackage[margin=1in]{geometry}
\usepackage[utf8]{inputenc} 
\usepackage[T1]{fontenc}    
\usepackage{hyperref}       
\usepackage{url}            
\usepackage{booktabs}       
\usepackage{amsfonts}       
\usepackage{nicefrac}       
\usepackage{microtype}      
\usepackage{color}
\usepackage{empheq}
\usepackage{mathtools}
\usepackage{amssymb}
\usepackage{mdframed}
\usepackage{hyperref}
\usepackage{amsthm}
\usepackage{cleveref}
\usepackage{stmaryrd}
\usepackage{autonum}
\usepackage{enumerate}

\usepackage{cleveref}

\crefname{equation}{}{}
\crefname{proposition}{Proposition}{Propositions}
\crefname{appendix}{Appendix}{Appendices}

\usepackage{color-edits}
\addauthor{vs}{blue}
\addauthor{ma}{red}

\newtheorem{theorem}{Theorem}
\newtheorem{cor}{Corollary}
\newtheorem{lemma}[theorem]{Lemma}

\newtheorem{example}{Example}[section]
\newtheorem{remark}{Remark}
\newtheorem{prop}{Proposition}

\newtheorem{assumption}{Assumption}

\DeclarePairedDelimiter{\dbracket}{\llbracket}{\rrbracket}

\newcommand{\z}{\mathbf{z}}

\usepackage{schemata}

\newcommand{\E}{\ensuremath{\mathbb{E}}}
\newcommand{\mcB}{\ensuremath{{\cal B}}}
\newcommand{\mcE}{\ensuremath{{\cal E}}}
\newcommand{\mcF}{\ensuremath{{\cal F}}}
\newcommand{\mcH}{\ensuremath{{\cal H}}}
\newcommand{\mcJ}{\ensuremath{{\cal J}}}

\newcommand{\mcR}{\ensuremath{{\cal R}}}
\newcommand{\bb}[1]{\ensuremath{\left[#1\right]}}
\newcommand{\bp}[1]{\ensuremath{\left(#1\right)}}
\newcommand{\bc}[1]{\ensuremath{\left\{#1\right\}}}
\newcommand{\ba}[1]{\ensuremath{\left| #1\right|}}
\newcommand{\bn}[1]{\ensuremath{\left\| #1\right\|}}
\newcommand{\RR}{\ensuremath{\mathbb{R}}}
\newcommand{\CC}{\ensuremath{\mathbb{C}}}
\newcommand{\II}{\ensuremath{\mathbb{I}}}
\newcommand{\QQ}{\ensuremath{\mathbb{Q}}}

\newcommand{\Var}{\ensuremath{\text{Var}}}

\newcommand{\dfmidi}{\ensuremath{\bar{\mcJ}_i}}
\newcommand{\dfmidik}[1][]{%
   \ifthenelse{ \equal{#1}{} }
      {\ensuremath{\bar{\mcJ}_{i, k_1}}}
      {\ensuremath{\bar{\mcJ}_{i, #1}}}
}
\newcommand{\ddfmidi}{\ensuremath{\bar{\mcH}_i}}
\newcommand{\ddfmidik}[1][]{%
   \ifthenelse{ \equal{#1}{} }
      {\ensuremath{\bar{\mcH}_{i, k_1, k_2}}}
      {\ensuremath{\bar{\mcH}_{i, #1}}}
}

\title{Asymptotics of the Empirical Bootstrap Method Beyond Asymptotic Normality}

\author{%
  Morgane Austern \&  Vasilis Syrgkanis \\
  Microsoft Research\\
}
\date{~}

\begin{document}

\maketitle
\begin{abstract}
 One of the most commonly used methods for forming confidence intervals for statistical inference is the empirical bootstrap, which is especially expedient when the limiting distribution of the estimator is unknown. However, despite its ubiquitous role, its theoretical properties are still not well understood for non-asymptotically normal estimators. In this paper, under stability conditions, we establish the limiting distribution of the empirical bootstrap estimator, derive tight conditions for it to be asymptotically consistent, and quantify the speed of convergence. Moreover, we propose three alternative ways to use the bootstrap method to build confidence intervals with coverage guarantees. Finally, we illustrate the generality and tightness of our results by a series of examples, including uniform confidence bands, two-sample kernel tests, minmax stochastic programs and the empirical risk of stacked estimators.
\end{abstract}

\section{Introduction}

 One of the most important tasks in statistical inference is to draw confidence intervals for the mean of a given estimator $\hat{\theta}_n:=g_n(X_1,\dots,X_n)$. When the limiting distribution of $\hat{\theta}_n$ is known, such knowledge can be leveraged to build confidence intervals that are asymptotically consistent. For example,
 if $\hat\theta_n$ is asymptotically normal and if $\hat{s}_n$ is a consistent estimator of its standard deviation, then  the following interval is an asymptotically consistent confidence interval at level $1-\alpha$: $$\lim_{n\rightarrow \infty}P\bp{\theta\in\Big[\hat{\theta}_n- \Phi^{-1}(1-\frac{\alpha}{2})\frac{\hat{s}_n}{\sqrt{n}},\hat{\theta}_n+ \Phi^{-1}(1-\frac{\alpha}{2})\frac{\hat{s}_n}{\sqrt{n}}\Big] }=1-\alpha.$$ 
 However, in general the limiting distribution of the estimator $\hat{\theta}_n$ is unknown and an alternative approach must be taken. One such approach is the empirical bootstrap method, which consists of: sampling new observations $Z_1,\dots,Z_n$  independently and uniformly in $\{X_1,\dots,X_n\}$, and defining $\hat{\theta}_n^{{\rm boot}}$ as the value of the estimator taken at the bootstrap sample $\hat{\theta}_n^{{\rm boot}}:=g_n(Z_1,\dots,Z_n)$. This procedure can be repeated many times in order to estimate the conditional distribution of  $\hat{\theta}_n^{{\rm boot}}-\mathbb{E}(\hat{\theta}_n^{{\rm boot}}\mid X_1,\dots,X_n)$. If this distribution is approximately the same as the distribution of $\hat{\theta}_n-\mathbb{E}(\hat\theta_n)$, as the sample size $n$ grows,  we say that the bootstrap method is consistent. We note that when this holds the bootstrap method can be used to establish approximate confidence intervals for $\mathbb{E}(\hat\theta_n)$. Notably, when $\hat{\theta}_n$ is asymptotically normal those intervals are known to be consistent under general conditions; see e.g. \cite{hall2013bootstrap,bickel1981,chernozhukov2013gaussian}. 
 
 \noindent As the bootstrap method is often used for estimators whose limiting distribution is unknown or non-Gaussian, we are interested in studying the limiting distribution of $\hat{\theta}_n^{{\rm boot}}$ for a general class of estimators with arbitrary limiting distributions and that  satisfy simple stability conditions. In particular, we assume that the functions $(g_n)$ are approximable by three-times differentiable functions whose first, second and third order partial derivatives, taken at $(X_1,\dots,X_n)$, are of respective order $o(n^{-1/3})$, $o(n^{-1/2})$ and $o(n^{-1})$. These conditions assure that the value of $g_n(X_1,\dots,X_n)$ is not oversensitive to the value of a single observation, and guarantee that the difference $g_n(X_1,\dots,X_n)-g_n(0,X_2,\dots,X_n)$ is approximable by $X_1 \times h_n(0,X_2,\dots,X_n)$ for a function $h_n$; this latter condition controls the degree of non-linearity of $g_n$ (see \Cref{sec:main} for a formal exposition). Exploiting these assumptions, we exactly characterize the limiting distribution of the bootstrap estimator $\hat \theta_n^{\rm boot}$, compare it to the distribution of the original statistic $\hat\theta_n$ and study how fast the distribution of  $\hat{\theta}_n^{{\rm boot}}$ converges. This allows us to derive tight conditions on the functions $(g_n)$ and the process $(X_i)$ guaranteeing the consistency of the bootstrap method, and to study how the shape of the confidence intervals evolve when those conditions do not hold.  Notably, we discover that when the mean of the observations $X_1$ is unknown the bootstrap method is in general not consistent. 
Moreover, we propose three alternative ways of using the bootstrap method to draw conservative confidence intervals with guaranteed minimum coverage. We illustrate our results by providing a series of simple examples, as well as examples derived from machine learning and econometrics, including:  the p-value of  kernel two-sample tests, the empirical risk of smooth stacked estimators, the value of min-max objectives, and confidence bands.  
 \subsection{Related litterature}
 \vspace{3mm}
 
 The empirical bootstrap method was first introduced in a breakthrough paper by Efron \cite{efron1992bootstrap}. Other bootstraps methods have since been proposed including the multiplier bootstrap \cite{wu1986jackknife}, the residual bootstrap  \cite{davison1997bootstrap} or the non-remplacement bootstrap method \cite{politis1994large}. A vast literature studies the theoretical properties of those techniques with some of the main results synthesized in the following books \cite{hall2013bootstrap,davison1997bootstrap,johnson2001introduction,beran1991asympotic}.  Most relevant to us are studies of the asymptotics of the bootstrap method. The consistency of the bootstrap method for linear statistics, t-statistics, Von-Mises functionals and quantiles has been established in \cite{bickel1981,mamen,singh1981asymptotic} and for U-statistics in \cite{arcones1994u,zhang2001bayesian}.
 Those results, among others, have been extended to high-dimensional regression  and M-estimation \cite{bickel1983bootstrapping,mammen1989asymptotics,chatterjee2011bootstrapping,belloni2015uniform,dezeure2017high}, misspecified models \cite{spokoiny2015bootstrap}, solutions of estimating equations \cite{chatterjee2005generalized}  and to robust estimators \cite{chen2020robust}.In contrast, other works established the poor performance of the bootstrap method for non smooth statistics \cite{eaton1991wielandt,beran1985bootstrap,bickel1981}, or for non-sparse high-dimensional regressions \cite{el2018can}.

\noindent  Several recent breakthrough papers studied the consistency of the bootstrap method, both empirical and wild, for the maximum of high-dimensional averages with  the dimension taken to be growing exponentially fast with the sample size. Notably \cite{chernozhukov2013gaussian,chernozhukov2017central} established the consistency of the  bootstrap and gaussian approximation method when respectively $\log(p_nn)^{7/8}=o(n^{1/8})$ and  $\log(p_nn)^{7/6}=o(n^{1/6}) $ hold. A series of work have strengthen those results:  \cite{chernozhukov2019improved,deng2017beyond,koike2020notes} established the consistency of the multiplier and empirical bootstrap when $\log(pn)^{5/4}=o(n^{1/4})$, \cite{lopes2020central} established a quasi $\sqrt{n}^{-1}$ rate for the wild bootstrap, \cite{deng2020slightly} built slightly conservative confidence sets with guaranteed coverage under the conditions than $\log(p)=o({n})$ and \cite{chen2018gaussian} proved that similar results hold for high-dimension U-statistics.
   Those works use a combination of the Stein method, Edgeworth expansions, Lindeberg's method \cite{chatterjee2006generalization} and the Slepian smart interpolation path. We note that the limiting distribution of those statistics are in general not Gaussian \cite{deng2017beyond}.
Other works have studied the accuracy of the bootstrap method for specific statistics whose distributions are known to be asymptotically not Gaussian such as: the operator norm in high dimensions \cite{lopes2019bootstrapping,han2018gaussian,johnstone2018pca}, sampled eigenvalues of random matrices in high and moderate dimensions \cite{el2019non} or M-estimators having cube root convergence \cite{cattaneo2020bootstrap}.
The main contrast between this series of work and ours is that, rather than studying the bootstrap method for one specific statistics or application, we seek to establish the asymptotics of the bootstrap method under universal  conditions on the estimators $(g_n)$. Our proof builds on a breakthrough method proposed by Chatterjee \cite{chatterjee2006generalization} that generalized the Lindeberg method to a general technique to compare the expectations of $f(X_{1:n})$ and $f(Y_{1:n})$ of a large class of functions $f$.

\section{Preliminaries}
\label{sec:preliminaries}

Let $\bp{X^n_i}$ be a triangular array of independent and identically distributed (i.i.d) processes with observations $X_i^n$ taking value in $\RR^{d_n}$. Moreover, let $X^n=(X_1^n, \ldots, X_n^n)$ denote its $n$-th row. Consider an estimator $\hat\theta_n:=g_n(X^n)$, where $g_n:\times_{l=1}^n \RR^{d_n}\rightarrow \RR$ is a measurable function, that we will typically refer to as a \emph{statistic}, and let $(g_n)$ denote the sequence of measurable functions as $n$ grows. To evaluate the performance of this estimator and build confidence intervals, we need to approximate its distribution. In this work, we will analyze the empirical bootstrap method.

\paragraph{Empirical bootstrap} Bootstrap samples $Z^n=(Z_1^n, \ldots, Z_n^n)$ are sampled with replacement from the observations $\bc{ X_1^n, \dots, X_n^n }$. This implies that conditionally on $X^n$ the coordinates of $Z^n$ are distributed i.i.d, with $Z_i^n \mid X^n \sim {\rm unif}\bp{ \bc{ X_1^n, \dots, X^n_n } }$, for all $i\in [n]$. 

\paragraph{Consistency metric and bootstrap consistency} Throughout the paper we denote with $Y^n=\bp{ Y_1^n, \ldots, Y_n^n }$ an independent copy of $X^n$. The bootstrap method is said to be consistent for $(g_n)$ if conditionally on $X^n$ the distribution of $g_n(Z^n)$ well-approximates the distribution of $g_n(Y^n)$, as $n\to \infty$. To make this statement rigorous we introduce a metric on the space of probability distributions. First, we define the class of three times continuously differentiable measurable functions with bounded third-order derivatives:
\begin{equation}
    \mcF:= \bc{ h\in C^3(\RR) \mid  ~\sup_{x\in \RR} \ba{h^{(i)}(x)} \le 1,~~\forall~ 1\le i\le 3 };
\end{equation}
Given this, we define the distance on the space of probability measures, as the maximum mean discrepancy, where test functions range over the class $\mcF$:
\begin{equation}
d_{\mcF}(\mu,\nu):=\sup_{h\in \mcF} { \E_{X\sim \mu, Y\sim \nu}\bb{h(X)-h(Y)} }.
\end{equation}
Moreover, we use the shorthand notation:
\begin{equation}
d_{\mcF}\bp{ \mu,\nu \mid \mcE} :=\sup_{h\in \mcF} { \E_{X\sim \mu, Y\sim \nu}\bb{h(X)-h(Y) \mid \mcE} }.
\end{equation}
This metric is related to the classical Levy-Prokhorov distance on probability spaces \cite{billingsley2013convergence}. We say that \emph{the empirical bootstrap method is consistent for $(g_n)$} if:
\begin{equation}
d_{\mcF} \bp{ g_n(Z^n), g_n(Y^n) \mid X^n  } \xrightarrow{p}{0}.
\end{equation}

\paragraph{Centering discrepancy and centered bootstrap consistency} Notably, an individual bootstrap sample $Z_1^n\mid X^n$, has a slightly different mean $\E\bb{Z_1^n\mid X^n}=\bar{X}^n:=\frac{1}{n}\sum_{i\le n}X_i^n$, than the one of $X_1^n$. As we will see this small difference plays a crucial role in determining the consistency of the bootstrap and for this reason it will be useful to define artificially centered versions of the random variables $(Z_i^n)$ and $(Y_i^n)$. A centered bootstrap sample
\begin{equation}
\tilde{Z}^n_i := Z^n_i- \bp{ \bar{X}^n-\E\bb{X_1^n} }
\end{equation}
is a bootstrap sample that has been re-centered to artificially have the same mean than $X_1^n$. Moreover, denote with $\tilde{Y}_i^n$ a corrected version of $Y_i^n$, artificially re-centered to have the same mean as $Z_1^n$, i.e.:
\begin{equation}
\tilde{Y}_i^n:=Y_i^n+\bar{X}^n-\E\bb{X_1^n}.
\end{equation}
Similarly we write $\tilde{Z}_i^n$ a corrected version of $Z_i^n$, artificially re-centered to have the same mean as $X_1^n$, i.e.:
\begin{equation}
\tilde{Z}_i^n:=Z_i^n-\bar{X}^n+\E\bb{X_1^n}.
\end{equation}
We say that \emph{the centered bootstrap is consistent for $(g_n)$} if:
\begin{equation}
d_{\mcF}\bp{ g_n(\tilde{Z}^n), g_n(Y^n) \mid X^n } \xrightarrow{p}{0}.
\end{equation}

\paragraph{From metric consistency to confidence intervals with nominal coverage} We can compare the confidence intervals of two random variables $X$ and $Y$ in terms of their mutual distance $d_{\mcF}(X,Y)$ (proof in \Cref{app:prop1}).
\begin{prop}
\label{prop1} 
Let $X$ and $Y$ be two real-valued random variables and $\mcE$ any random event. Let $\epsilon>0$ be a constant then for any Borel set $A\in \mcB(\RR)$ the following holds: 
{\begin{equation}
P(X\in A_{6\epsilon} \mid \mcE) \geq P(Y\in A \mid \mcE) - \frac{d_{\mcF}(X,Y\mid \mcE)}{\epsilon^3},
\end{equation}}
where we wrote $A_{\epsilon}:= \{x\in\RR \mid \exists y\in A~{\rm s.t}~|x-y|\le \epsilon\}$. Moreover, if $[a,b]$ is a confidence interval at level $1-\alpha$ for $Y-\E\bb{Y\mid\mcE}$, conditional on $\mcE$, then we have:
{\begin{equation}
    P\bp{ X-\E\bb{X\mid \mcE} \in [a - 6\epsilon, b + 6\epsilon] \mid \mcE}\ge 1-\alpha-\frac{2d_{\mcF}(X,Y\mid \mcE)}{\epsilon^3}.
\end{equation}}
\end{prop}

\noindent For instance, suppose that we care about estimating $\theta_n:=\E\bb{ g_n(Y^n)}$. Then the bootstrap method, if consistent, can be used to build consistent confidence intervals for $\theta_n$. Indeed since we can estimate the conditional distribution of $\theta_n^{{\rm bootstrap}}:=g_n(\tilde{Z}^n)$, by drawing sufficiently many bootstrap sub-samples, we can find $C^{\alpha,n}$ such that
\begin{equation}
P\bp{\hat\theta_n^{{\rm bootstrap}}-\E\bb{\hat\theta_n^{{\rm bootstrap}} \mid X^n}\in C^{\alpha,n} \mid X^n }=1-\alpha.
\end{equation}
Then, if we write $\hat{\theta}_n:=g_n(Y^n)$, using the consistency of the bootstrap method we obtain that:
\begin{equation}
{\liminf_{\epsilon\downarrow 0}}\liminf_{n\to \infty}P\bp{ \hat\theta_n-\theta_n\in {C^{\alpha,n}_{\epsilon}}}\ge 1-\alpha.     
\end{equation}
Therefore, confidence intervals built using the bootstrap method achieve asymptotically nominal level of confidence. {We note that prior works (e.g. \cite{chernozhukov2013gaussian,chernozhukov2017central}), typically provide a slightly stronger statement that $\liminf_{n\to \infty} P(\hat{\theta}_n - \theta_n\in C^{\alpha,n})\geq 1-\alpha$, by proving anti-concentration results on the limit distribution of $\hat{\theta}_n - \theta_n$. Such anti-concentration, allows one to argue that the mass of the random variable $\hat{\theta}_n-\theta_n$ contained in $C_{\epsilon}^{\alpha,n}$ converges to the mass contained in $C^{\alpha,n}$ as $\epsilon \downarrow 0$ and thereby,  $\liminf_{\epsilon\downarrow 0}\liminf_{n\to \infty} P(\hat{\theta}_n - \theta_n\in C_{\epsilon}^{\alpha,n}) = \liminf_{n\to \infty} P(\hat{\theta}_n - \theta_n\in C^{\alpha,n})$. Given that these results typically require stronger conditions on the statistic and many times Gaussian limits, we omit this step in this work and note that a slightly weaker, albeit still practically useful, statement on coverage is achievable in a more general setup.}

\subsection{Notations and definitions}
For a scalar random variable $X$ we denote with $\|X\|_{L_p}$, the $L_p$-norm:  $\|X\|_{L_p}:=\E[X^p]^{1/p}$. Moreover, for vector $x\in \RR^d$, we denote with $\|x\|_p$, the $\ell_p$ vector norm: $\|x\|_p = \left(\sum_{i=1}^d x_i^p\right)^{1/p}$. For simplicity, given a sequence $(x_i)$, with $x_i \in \RR^{d}$ and a constant $c\in \RR^d$, we shorthand 
\begin{align}
x_{1:n}~:=~& (x_1,\dots,x_n), & 
x_{1:n}+c ~:=~& (x_1+c,\dots,x_n+c), &
cx_{2:n} ~:=~& (c,x_2,\dots,x_n).    
\end{align}
We denote the $k$-th coordinate of $x_i\in \RR^d$ as $x_{i,k}$. For a function $f:\times_{l=1}^n\RR^{d_n}\rightarrow \RR$ and a random variable $X$ taking values in $\RR^{d_n}$, we designate $f(\cdot+X)$ the random function: $x_{1:n}\rightarrow f(x_{1:n}+X)$.

\paragraph{Lindenberg path interpolation} Let $Z^{n,i}$ and $Z^{n, i,x}$ be the following interpolating processes between $Z^n$ and $\tilde{Y}^n$:
\begin{align}
    Z^{n, i} :=~& \bp{\tilde{Y}_1^n, \ldots, \tilde{Y}_{i}^n, Z_{i+1}^n, \ldots, Z_n^n }\\
    Z^{n, i, x} :=~& \bp{\tilde{Y}_1^n, \ldots, \tilde{Y}_{i-1}^n, x, Z_{i+1}^n, \ldots, Z_n^n }
\end{align}

\paragraph{Higher-order derivatives and bounds} If a function $f$ is three-times differentiable then we let:
\begin{align}
    \partial_{i,k} f(x_{1:n}) ~:=~& \partial_{x_{i,k}} f(x_{1:n})\\
    \partial^2_{i,k_{1:2}} f(x_{1:n}) ~:=~& \partial_{x_{i,k_1}} \partial_{ x_{i,k_2}} f(x_{1:n})\\
    \partial^3_{i,k_{1:3}}f(x_{1:n}) ~:=~& \partial_{x_{i,k_1}} \partial_{x_{i,k_2}} \partial_{x_{i,k_3}} f(x_{1:n})
\end{align}
Moreover, for a potentially random function $f$ we define the constants:
\begin{align}
M^n_{k}:=~& 2\, \|X_{1,k}^n\|_{L_{12}},\\
D_{k_{1}}^n(f) ~:=~& M_{k_1}^n\, \max_{i\le n} \bn{ \partial_{i,k_{1}} f(Z^{n, i,\bar{X}^n}) }_{L_{12}} \\
D_{k_{1:2}}^{n}(f) ~:=~&  M_{k_1}^n\, M_{k_2}^n\, \max_{i\le n}\bn{ \partial^2_{i,k_{1:2}}f(Z^{n, i,\bar{X}^n}) }_{L_{12}}\\
D_{k_{1:3}}^{n}(f) ~:=~& M_{k_1}^n \, M_{k_2}^n \, M_{k_3}^n\, \max_{i\le n} \bn{\max_{x\in \bb{\bar{X}^n,\tilde{Y}^{n}_1} \cup \bb{\bar{X}^n, Z_1^n}} \partial^3_{i,k_{1:3}} f(Z^{n, i,x}) }_{L_{12}}
\end{align}
where for any two vectors $a, b\in \RR^{d}$, we denote with $[a,b]$ their convex closure, i.e.
\begin{equation}
    [a,b] := \{t\, a + (1-t)\, b: t\in [0, 1]\}
\end{equation}

\section{Main Results}\label{sec:main}

If the statistics $(g_n)$ were linear, i.e. $g_n(x_{1:n})=\sum_{i\le n}x_i$, then the influence of a single observation $X_1^n$, on the estimate $\hat{\theta}_n$, would depend uniquely on the value of the random variable itself, i.e. $g_n(X^n)-g(0X_{2:n}^n)=X_1^n$. This is not the case for non-linear statistics. For instance, if $g_n(x_{1:n}) = \max\bp{\sum_{i\le n}x_{i,1}, \sum_{i\le n}x_{i,2}}$, then the influence of observation $x_1$ depends on the relative size of $\sum_{i>2 }x_{i,1}$ and $\sum_{i>2 }x_{i,2}$.
In this paper, we  want to study the asymptotics of the bootstrap method for such non-linear statistics, with complex influence functions. To control the degree of non-linearity, we assume that the statistics $(g_n)$ can be approximated by three times differentiable functions.

\begin{assumption}[Approximability by $\CC^3$]\label{ass:approx}There exists a sequence of functions $(f_n)$ with $f_n\in \CC^3$ s.t.:
\begin{enumerate}
\item The functions $(f_n)$ approximate the estimators $(g_n)$: 
\begin{align}\label{ass:1}
    \bn{ f_{n}(Z^n)-g_n(Z^n) }_{L_1} + \bn{ f_{n}(\tilde{Y}^n)-g_n(\tilde{Y}^n) }_{L_1} \xrightarrow{n\to\, \infty} 0. \tag{$H_0$}
\end{align}
\item The first, second and third order derivatives are respectively of size $o(n^{-1/3})$, $o(n^{-1/2})$, $o(n^{-1})$:
\begin{align}\label{ass:2}
\begin{aligned}
     R_{n,1} ~:=~& n^{1/3}\sum_{k_1\le d_n} D_{k_{1}}^{n}(f_{n})=o(1) 
     & 
     R_{n,2} ~:=~& \sqrt{n}\sum_{k_1,k_2\le d_n} D_{k_{1:2}}^{n}(f_{n})=o(1)\\
     R_{n,3} ~:=~& n\sum_{k_1,k_2,k_3\le d_n} D_{k_{1:3}}^{n}(f_{n})=o(1).
\end{aligned}
\tag{$H_1$}
\end{align}
\end{enumerate}
\end{assumption}

To motivate \Cref{ass:approx}, we present two illustrating examples of simple estimators which fail to satisfy conditions \eqref{ass:1}, \eqref{ass:2} and for which the bootstrap method is not consistent. Firstly we note that if \eqref{ass:1} and \eqref{ass:2} hold we have $\bn{ g_n\bp{ X^n}-g_n\bp{0X_{2:n}^n}}_{L_3}=o(n^{-1/3}).$ This is a first-order stability property, i.e. that each sample $i$'s influence on the estimate has to decay at rate $n^{-1/3}$. Our first example is chosen to violate this.
\begin{example}
Let $(X_i^n)$ be a sequence of i.i.d observations distributed as $X_i\sim {\rm unif}(0,1)$. Let $(g_n)$ be the following sequence of functions: $g_n(x_{1:n}):=n\min_{i\le n}x_i$. Then neither the bootstrap method nor the centered bootstrap method are consistent.
Moreover, we note that: $\bn{g_n(X_{1:n})-g_n(0X_{2:n})}_{L_3}\propto n^{-1/3}$. In this example, the bootstrap estimator $g_n( Z^n)\ge g_n(X)$ is systematically larger than the original statistic, which leads to inconsistency of the bootstrap distribution.
\end{example}

Another consequence of having the second and third order derivative of respective order $o(n^{-1/2})$ and $o(n^{-1})$ is that the following two conditional expectations are very similar:
\begin{equation}\label{eqn:second-order-implication}
\bn{ \E\bb{g_n(Z^n) \mid X^n}-\E\bb{ g_n(\tilde{Y}^n) \mid X^n} }_{L_1}=o(1).
\end{equation}
Our second example is chosen to satisfy the main implication of the first order stability conditions: i.e. $\|g_n(X^n)-g_n(0X_{2:n}^n)\|_{L_3}=o(n^{1/3})$; but to fail to respect this new property.

\begin{example}\label{ss_nulle}
Let $(X_i^n)$ be a sequence of i.i.d observations distributed as $X_i^n\sim {\rm unif}(0,1)$. Let $(g_n)$ be the following sequence of functions: \begin{equation}
g_n(x_{1:n}) := \frac{1}{\sqrt{n}} \sum_{i\le n} 1\bc{\min_{i\ne j}|x_i-x_j|>1/n}- P\bp{\min_{i\ne 1} \ba{X_1-X_i}>1/n}.    
\end{equation}
Then neither the bootstrap nor the centered bootstrap are consistent.
Moreover we note that the first order stability result holds, but  $\bn{g_n(Y_{1:n})-g_n(0Y_{2:n})}_{L_3}\propto n^{-1/2}$ and Condition~\eqref{eqn:second-order-implication} is violated. The main driving force of inconsistency in this example is that contrary to the original sample, it is likely that the bootstrap sample will contain repeats. Hence we expect $P\bp{\min_{i\ne 1}|Z_i^n-Z_j^n|>1/n \mid X^n}$ to be smaller than $P\bp{\min_{i\ne 1}|X_1-X_i|>1/n}$. See \Cref{app:ss_nulle} for formal proof.
\end{example}

Under \Cref{ass:approx} we study the limiting distribution of the bootstrap statistic and establish that it is asymptotically the same as $g_n(\tilde{Y}^n)$ (proof in \Cref{app:thm1}).

\begin{theorem}\label{thm1}
Let $(g_n:\times_{l=1}^n\RR^{d_n}\to \RR)$ be a sequence of measurable  functions. Let $(X^n_i)$ be a triangular array of i.i.d processes such that $X^n_1\in L_{12}$. Under \Cref{ass:approx}, there is a constant $K$ independent of $n$ such that:
\begin{equation}
\begin{split}
\bn{d_{\mcF}\bp{ g_n(Z^n),\, g_n(\tilde{Y}^n) \mid X^n } }_{L_1}
\le~& 
\bc{
\begin{split}
&\bn{ g_n(\tilde{Y}^n) -f_{n}(\tilde{Y}^n) }_{L_1}
+ \bn{ g_n(Z^n)-f_{n}(Z^n) }_{L_1} \\
&~ + K\, \bp{ R_{n,1}^2\max\bc{ \frac{1}{n^{1/6}}, R_{n,1}} + R_{n,3} + R_{n,2} }
\end{split}
}
\to 0.
\end{split}
\end{equation} 
\end{theorem}

\begin{remark}
We remark that the theorem also holds under slightly modified stability conditions. See \Cref{jardin} in the appendix for more details. 
Moreover, the hypothesis that $(X_i^n)$ is an i.i.d process can also be relaxed to assuming that the process $(X_i^n)$ is exchangeable. See \Cref{thm8} for more details in the appendix. 
Finally, note that \Cref{thm1} can also be extended to random estimators $(g_n)$, such as ones obtained by stochastic optimization methods (e.g SGD). See \Cref{thm8} for more details in the appendix.
\end{remark}

\noindent \Cref{thm1} guarantees that we can use the bootstrap method to estimate the distribution of $g_n(\tilde{Y}^n)$, which implies that it can also be used to build confidence intervals for $\E\bb{g_n(\tilde{Y}^n)\mid X^n}$.
\begin{cor}\label{cor1}
Let $(g_n:\times_{l=1}^n\RR^{d_n}\rightarrow \RR)$ be a sequence of measurable symmetric functions. Let $(X^n_i)$ be a triangular array of i.i.d processes such that $X^n_1\in L_{12}$. Assume that $(g_n)$ and $(X^n_i)$ satisfy all the conditions of \Cref{thm1}. Then there is a constant $K$ independent of $n$ such that:
\begin{equation}\begin{split}&
\bn{ d_{\mcF} \bp{ g_n(Z^n)-\E\bb{ g_n(Z^n) \mid X^n},\, g_n(\tilde{Y}^n)-\E\bb{ g_n(\tilde{Y}^n) \mid X^n} \mid X^n } }_{L_1}
\\&\le 2\bc{
\begin{split}
&\bn{ g_n(\tilde{Y}^n) -f_{n}(\tilde{Y}^n) }_{L_1}
+ \bn{ g_n(Z^n)-f_{n}(Z^n) }_{L_1} \\
&~ + K\, \bp{ R_{n,1}^2\max\bc{ \frac{1}{n^{1/6}}, R_{n,1}} + R_{n,3} + R_{n,2} }
\end{split}
}
\to 0.
\end{split}
\end{equation} 
\end{cor}

\noindent However the distribution we are interested in is that of $g_n(Y^n)$ rather than $g_n(\tilde{Y}^n)$. Moreover, the shape of the confidence intervals of $g_n(\tilde{Y}^n)$ can be arbitrary compared to the ones of $g_n(Y^n)$, i.e. they are not systematically larger or smaller. This is illustrated in \Cref{sec:illustrative} by a series of examples.
In \Cref{sec:uniform-perturb} we propose conditions that guarantee that the two distributions are  asymptotically identical. In \Cref{sec:tightness} we prove that those conditions are tight and propose the use of the bootstrap method to build adjusted confidence intervals that are guaranteed to have at least (but not necessarily equal to) some minimum asymptotic coverage.

\subsection{Stable Estimators to Uniform Perturbations}
\label{sec:uniform-perturb}

In this subsection we explore conditions guaranteeing that the distribution of $g_n(Y^n)-\E\bb{g_n(Y^n)}$ is asymptotically the same as the distribution of $g_n(\tilde{Y}^n)-\E\bb{g_n(\tilde{Y}^n)\mid X^n}$, conditional on $X^n$, as this would imply that the bootstrap method provides consistent confidence intervals for $\E\bb{g_n(Y^n)}$. We start by noting that if $(g_n)$ are linear then it automatically holds as we have 
$$\sum_{i\le n} \tilde Y_{i}^n-\E\bb{\sum_{i\le n} \tilde Y_{i}^n \mid X^n } =\sum_{i\le n} Y_{i}^n-\E\bb{\sum_{i\le n} Y_{i}^n \mid X^n }.$$
Observe that the random variables $\tilde{Y}_i^n$ differ from $Y_i^n$ in a benign manner: a random offset $\bar{X}^n - \E\bb{X_1^n}$, which is independent of $Y^n$, is added to all the random variables. Moreover, this offset is with high probability $O(n^{-1/2})$, since it is the difference of a sample and a population mean. We will refer to such perturbations of a sample $Y^n$ as a \emph{uniform perturbation}. To study general statistics, we introduce the following assumption which guarantees that small uniform perturbations do not drastically change the distribution of $g_n(Y^n)$:
\begin{assumption}[Stability to Uniform Perturbation]\label{ass:uniform} A statistic sequence $(g_n)$ is stable to small uniform perturbations if for all $B>0$: 
\begin{align}
r^{n,B}:=~ \bn{\sup_{x\in B_{d_n}(0,B)}\ba{g_n(X^n + x/{\sqrt{n}}) - g_n(X^n)-\mathbb{E}\Big[ g_n(X^n + x/{\sqrt{n}}) - g_n(X^n)\Big]}}_{L_1} \overset{n\to\infty}{\to} 0 \tag{$H_2$}\label{eqn:uniform}
\end{align}
where we define $B_{d_n}(0,B):=\{x\in \RR^{d_n} \mid  \|x\|_2\le B\}$.
\end{assumption}

Note that the perturbations considered in hypothesis \eqref{eqn:uniform} are uniform on all the coordinates $i\in [n]$. This notably implies that if $g_n$ depends only the relative distance between the observations then hypothesis \eqref{eqn:uniform} holds. We prove, under hypothesis \eqref{eqn:uniform}, that the bootstrap method is consistent and hence by \Cref{prop1} can be used to build asymptotically consistent confidence intervals for $\E\bb{g_n(X^n)}$ (proof in \Cref{app:thm2}).

\begin{theorem}\label{thm2}Let $(g_n:\times_{l=1}^n\RR^{d_n}\rightarrow \RR)$ be a sequence of measurable functions. Let $(X^n_i)$ be a triangular array of i.i.d processes such that $X^n_1\in L_{12}$. Assume that $(g_n)$ satisfies \Cref{ass:approx} and \Cref{ass:uniform}.
Then there exists a universal constant $K$ such that:
\begin{multline}
\bn{ d_{\mcF}\bp{ g_n(Z^n)-\mathbb{E}\big[g_n(Z^n)|X^n],\, g_n(Y^n) -\mathbb{E}(g_n(Y^n))\mid X^n } }_{L_1}
\\
\le \inf_{B_n\in \mathbb{R}}\bc{
\begin{gathered}
\bn{g_n(\tilde Y^n) -f_n(\tilde Y^n)}_{L_1}
+\big\|g_n( Z^n)-f_n(Z^n)\big\|_{L_1}+K (R_{n,1})^2 \max\bc{ \frac{1}{n^{1/6}},R_{n,1} }\\+K\bp{ R_{n,3} + R_{n,2} }+ \frac{2\sqrt{\sum_{k\le k_n}\|X_{1,k}^n\|_{L_2}^2}}{B_n} \Big[\big\|g_n(Y^n)\big\|_{L_2}+\big\|g_n(\tilde Y^n)\big|_{L_2}\Big]+ r^{n,B_n}
\end{gathered}
}\to 0.
\end{multline}
\end{theorem}

Condition \eqref{eqn:uniform} holds beyond linear statistics. We present two simple illustrative examples of such non-linear estimators, for which \eqref{eqn:uniform} is satisfied. 
 \begin{example}\label{ex3}
 Let $(X_{i})$ be an i.i.d sequence of random variables taking value in $\RR$. We suppose that they are bounded: $\|X_1\|_{L_{\infty}}<\infty$. We define the functions $(g_{n,1})$ and $(g_{n,2})$ as satisfying:
 $$g_{n,1}:x_{1:n}\rightarrow \bp{\frac{1}{\sqrt{n}}\sum_{i\le \lfloor n/2\rfloor}x_{i}-x_{i+\lfloor n/2\rfloor}}^2,\qquad g_{n,2}:x_{1:n}\rightarrow \sqrt{n}\Big[\prod_{i=1}^n\bp{ 1+\frac{x_i-\bar{x}^n}{n}}-1\Big]. $$
 Then the functions $(g_{n,1},g_{n,2})$ satisfy conditions \eqref{ass:1}, \eqref{ass:2} and \eqref{eqn:uniform}. Hence the bootstrap is consistent, i.e.:
 $$d_{\mcF}\bp{g_{n,1}( Z^n)-\E(g_{n,1}(Z^n)|X),g_{n,1}( Y_{1:n})-\E(g_{n,1}(Y_{1:n}))\mid X}\rightarrow 0;$$
 $$d_{\mcF}\bp{g_{n,2}( Z^n)-\E(g_{n,2}(Z^n)|X),g_{n,2}( Y_{1:n})-\E(g_{n,2}(Y_{1:n}))\mid X}\rightarrow 0.$$
 \end{example}

However hypothesis \eqref{eqn:uniform} can be easily violated by simple examples. We prove in the next subsection, under mild conditions, that \emph{violation of \eqref{eqn:uniform} implies that no re-sampling method can provide asymptotically consistent confidence intervals}. We present here a simple example of this phenomenon.
\begin{example}
Let $(X_i)$ be an i.i.d sequence of scalar-valued, bounded observations with mean $0$. Write $(Y_i)$ an independent copy of $(X_i)$. Define the following functions $g_n:x_{1:n}\rightarrow \bp{ \frac{1}{\sqrt{n}}\sum_{i\le n}x_i }^2$. Then hypothesis \eqref{eqn:uniform} does not hold and the centered distributions of $(g_n(Y_{1:n}))$ and $(g_n(\tilde Y_{1:n}))$ are not asymptotically identical
\begin{equation}
d_{\mcF}\bp{ g_n(\tilde{Y}_{1:n})-\E\bb{ g_n(\tilde{Y}_{1:n})\mid X},\, g_n( Y_{1:n})-\E\bb{g_n(Y_{1:n})} \mid X}\not\rightarrow 0.
\end{equation}
\end{example}

\subsection{Impossibility for Unstable Estimators to Uniform Perturbations}
\label{sec:tightness}

In this section we prove that if the estimators are sensitive to small uniform perturbations then the bootstrap method is not consistent. Then we offer three solutions on how to use the bootstrap to build confidence intervals with a guaranteed minimum coverage.

\paragraph{Non-consistency of the bootstrap if the estimators are unstable.}
Let $\mathcal{P}'_n$ be a class of probability distributions on $\RR^{d_n}$.
Write $\mathcal{P}_M(\RR)$ the set of probability measures on $\RR$. We say that the centered distribution of $g_n(\cdot)$ can be estimated over the class of distributions $\mathcal{P}'_n$ if there is a measurable function $\mathcal{Q}_n:x_1,\dots,x_n\rightarrow \mathcal{P}_M(\RR)$ such that for all sequences of distributions $(\nu_n)\in \prod_{i=1}^{\infty}\mathcal{P}'_n$ we have 
\begin{equation}
\E_{X^n\sim \nu_n}\bb{ \ba{ d_{\mcF}\bp{ \mathcal{Q}_n(X^n),\, g_n(Y^n)-\E\bb{g_n(Y^n)} \mid X^n } } }\xrightarrow{n\rightarrow \infty }0;
\end{equation}
where $(Y_i^n)$ is taken to be to be an independent copy of $(X_i^n)$.
We prove that the centered distribution of $g_n(\cdot)$ cannot be estimated if a hypothesis similar to \eqref{eqn:uniform} is not respected (proof in \Cref{app:thm3}).
\begin{theorem}\label{thm3}
Let $(g_n:\times_{l=1}^n\RR^{d_n}\rightarrow \RR)$ be a sequence of measurable functions. Define $\Omega_n\subset \RR^{d_n}$ to be a non empty open subset of $\RR^{d_n}$ and let  $\mathcal{P}'_n:=\{p_{\theta}^n,\theta\in \Omega_n\}$ be a parametric subset of $\mathcal{P}_n$ such that $\E_{X\sim p^n_{\theta}}(X)=\theta$. Denote $(\mathcal{I}_n(\theta))$ the Fisher information matrix of $(p^n_{\theta})$. Suppose that there is a sequence of measures $(p^n_{\theta_n})\in \prod_{n=1}^{\infty}\mathcal{P}'_n$, a sequence $(z_n)\in \prod_{n=1}^{\infty}\RR^{d_n}$ and a real $\epsilon>0$ such that \begin{enumerate}[(i).]
\item $\limsup \sup_{\tilde{\theta}_n\in \bb{ \theta_n,\, \theta_n+\frac{z_n}{\sqrt{n}} }} \bn{ \mathcal{I}_n(\tilde \theta_n)^{1/2}\frac{z_n}{\sqrt{n}} }_{2} < \infty$.
\item The following holds if $(X_i^n) \overset{i.i.d}{\sim} p^n_{\theta_n}$
\begin{equation}
\liminf_{n\rightarrow \infty} d_{\mcF}\bp{ 
g_n\bp{ X^n+\frac{z_n}{\sqrt{n}}} - \E\bb{ g_n\bp{ X^n+\frac{z_n}{\sqrt{n} } }},\, 
g_n(X^n) - \E\bb{ g_n(X^n)} } >\epsilon.
\end{equation}
\item $\theta_n+\frac{z_n}{\sqrt{n}}\in \Omega_n$
\end{enumerate}
Then for all measurable functions $\mathcal{Q}_n:x_1,\dots,x_n\rightarrow \mathcal{P}_M(\RR)$ there is a sequence $(\nu_n)\in\prod_{n=1}^{\infty} \mathcal{P}'_n$ such that:
\begin{equation}
    \bn{ d_{\mcF}\bp{ \mathcal{Q}_n(X^n), g_n(Y^n)-\E\bb{g_n(Y^n)} \mid X^n} }_{L_1}\not \rightarrow{ }0
\end{equation}
where $(X_i^n),(Y_i^n)\overset{i.i.d}{\sim}\nu_n.$
\end{theorem}

\noindent \Cref{thm3} implies that if the means of the observations are unknown then no re-sampling method will in general be consistent. We propose in \Cref{sec:centered}, \Cref{sec:corrected} and \Cref{sec:robust} three alternative ways to build confidence intervals, that bypass this impossibility result and have asymptotically a guaranteed coverage of at least $1-\alpha$. Albeit, some of these intervals will potentially have larger size than needed.

\subsubsection{Consistency of the Centered-Bootstrap}\label{sec:centered}

First we explore the case when we know the mean of the observations $\E\bb{X_1^n}$. In this case, we can leverage this knowledge to build centered bootstrap samples $\tilde{Z}_i^n:=Z_i^n+\E\bb{X_1^n}-\bar{X}^n$. Observe that these centered samples satisfy the crucial property that
$\E\bb{\tilde{Z}_1^n\mid X^n}=\E(X_1^n)$. We prove that, under mild conditions, the centered bootstrap estimator is asymptotically consistent. The conditions needed for the centered bootstrap to be consistent are  hypothesis $(H_0)$ and $(H_1)$ formulated instead for $(Y^n_i)$ and  $(\tilde Z_i^n)$ rather than for  $(\tilde Y^n_i)$ and $(Z_i^n)$.

\begin{assumption}[Approximation by $\CC^3$ of $g_n(\cdot +\E(X_1^n)-\bar{X}^n)$.]\label{ass:approx-center}
There exists a sequence of functions $(f_n)$ with $f_n\in \CC^3$ s.t.:
\begin{enumerate}
    \item The functions $(f_n)$ approximate the estimators $(g_n)$: \begin{align}\sup_n\big\|f_n(\tilde Z^n)-g_n(\tilde Z^n)\big\|_{L_1}+\big\|f_n( Y^n)-g_n( Y^n)\Big\|_{L_1}\xrightarrow{\beta\rightarrow \infty} 0.\tag{$H^{{\rm c}}_0$}\end{align}
    \item  The first, second and third order derivatives are respectively of size $o(n^{-1/3})$, $o(n^{-1/2})$, $o(n^{-1})$:
     \begin{align}&R_{n,1}^{{\rm c}}:= n^{1/3}\sum_{k_1\le d_n} D_{1,k_{1}}^{n}\bp{f_n(\cdot+\E(X_1^n)-\bar{X}^n)}=o(1);\\& R_{n,2}^{{\rm c}}:=\sqrt{ n}\sum_{k_1,k_2\le d_n} D_{2,k_{1:2}}^{n}\bp{f_n(\cdot+\E(X_1^n)-\bar{X}^n)}=o(1);\tag{$H^{{\rm c}}_1$}
     \\&R_{n,3}^{{\rm c}}:= n\sum_{k_1,k_2,k_3\le d_n} D_{3,k_{1:3}}^{n}\bp{f_n(\cdot+\E(X_1^n)-\bar{X}^n)})=o(1).\end{align} 
\end{enumerate}
\end{assumption}
 
We show that under those conditions the centered bootstrap is asymptotically consistent and thereby can be used to build confidence intervals with asymptotically nominal coverage. 

\begin{theorem}\label{thm_center}Let $(g_n:\times_{l=1}^n\RR^{d_n}\rightarrow \RR)$ be a sequence of measurable  functions. Let $(X^n_i)$ be a triangular array of i.i.d processes such that $X^n_1\in L_{12}$. Assume that there is a sequence $(f_n:\times_{l=1}^n\RR^{d_n}\rightarrow \RR)$ of measurable functions satisfying conditions $(H^{{\rm c}}_0)$ and  $(H^{{\rm c}}_1)$. 
Then there is a universal constant $K$ such that:
\begin{equation}
\bn{ d_{\mcF}\bp{ g_n( \tilde Z^n),\, g_n( Y^n) \mid X^n } }_{L_1}
\le 
\bc{\begin{gathered}
\bn{ g_n(Y^n) -f_n( Y^n) }_{L_1}
+ \bn{ g_n(\tilde  Z^n)-f_n(\tilde Z^n) }_{L_1}
\\+K\bp{ (R_{n,1}^{c})^2 \max\bc{ \frac{1}{n^{1/6}},R_{n,1}^{c} } + R_{n,3}^{c}+ R_{n,2}^{c} }
\end{gathered}}
\to 0.
\end{equation} 
\end{theorem}

\begin{example}[Application to hypothesis testing] An important application is hypothesis testing. Suppose we want to test $(H_0):~\mathbb{E}\bb{X_1^n}=\theta$ against an alternative $(H_1)$. In this goal, we build a test statistic $\hat{T}_n(X^n)$ for which we want to compute a p-value. Let $(Z_i^n)$ be a bootstrap sample of $\{X_1^n,\dots,X_n^n\}$; define $(Z^{\theta}_i)$ as the following process:
\begin{equation}
Z^{\theta}_i:=Z^n_i-\bar{X}^n+ \theta.
\end{equation}
{We remark that under the null, $(Z^{\theta}_i)$ is a centered bootstrap sample of $X^n$.} Using \Cref{thm_center} we know, under stability conditions on $(\hat{T}_n)$ (i.e. \Cref{ass:approx-center}), that we can use $\hat{T}^n(Z^{\theta}_{1:n})$ to estimate the p-value of $\hat{T}_n$.
\begin{prop}
Let $(X_i^n)$ be a triangular array of i.i.d processes taking value in $\RR^{d_n}$ and $\hat{T}_n:\times_{i=1}^n \RR^{d_n}\rightarrow \RR$ be a sequence of measurable functions that satisfies \Cref{ass:approx-center}. Then:
$$\bn{d_{\mcF}\bp{ \hat{T}_n(Z^{\theta}_{1:n}), \hat{T}_n(Y_{1:n}^{n,\theta}) \mid X^n }}_{L_1}\to 0.$$
\end{prop}
\end{example}

\subsubsection{Corrected Confidence Interval}\label{sec:corrected}

We now investigate two distinct methods to build  conservative confidence intervals, when the mean $\E[X_1^n]$ is not known.

\noindent  According to \Cref{cor1} the bootstrap method can be used to build consistent confidence intervals for $\mathbb{E}(g_n(\tilde Y^n)\mid X^n)$. Therefore if we can bound the distance from $\mathbb{E}(g_n(\tilde Y^n)\mid X^n)$ to $\mathbb{E}(g_n( Y^n))$ we can use the bootstrap method to build confidence intervals on the latter. This is the first method that we propose. To do so we exploit the fact that under mild conditions $\sqrt{n}\big[\bar{X^n}-\mathbb{E}(X_1^n)\big]$ is approximately normal. We assume that the function $x\rightarrow\E\bb{g_n(Y^n+x)}$ is $\alpha$-Holder and that the moments of $X_1^n$ are bounded. More formally, suppose that there is a sequence $(C_n)$ and a constant $b$ such that:
\begin{align}\label{ass:h3}&
\Big|\E\bb{g_n(Y^n+\frac{x}{\sqrt{n}})-g_n(Y^n)}\Big|\le C_n\max_{k\le d_n}|x_k|^{\alpha}, \quad \forall x\in \mathbb{R}^{d_n}\tag{$H_3$}
 \\&\min_{j\le d_n}\|X^n_{1,j}\|_{L_3}\ge b,\quad \frac{\log(d_n)^{7/6}\|\sup_{k\le d_n}|X^n_{1,k}|\|_{L_4}^4}{n^{1/6}}=o(1).
\end{align}

\begin{theorem}\label{thm_4}
Let $(g_n:\times_{l=1}^n \RR^{d_n}\rightarrow \RR)$ be a sequence of measurable functions satisfying Assumption 1 and \eqref{ass:h3}. Denote $\Sigma_n$ the variance-covariance matrix of $X_1^n$ and $(N^n)$ to be a sequence of Gaussian vectors distributed as $N^n\sim N(0,\Sigma_n)$.
Let $\beta>0$ be a real; write $t_{g,n}^{\beta/2}$ and $t_{{\rm b},n}^{\beta/2}(X^n)$ as quantities satisfying $$P\bp{\big|g_n(Z^n)-\E(g_n(Z^n)|X^n)\big|\ge t_{{\rm b},n}^{\beta/2}(X^n)\mid X^n}\le \beta/2 ;$$$$P\bp{\max_k\big|N^n_k|\ge(t_{g,n}^{\beta/2})^{\frac{1}{\alpha}}{C_n}^{-\frac{1}{\alpha}}}\le \beta/2 .$$

Then the following holds: $$\limsup_{\delta \downarrow 0}\limsup_n P\bp{\E(g_n(Y^n))\le \big[g_n(Z^n)-t_{{\rm b},n}^{\beta/2}- t_{g,n}^{\beta/2}-\delta,~~g_n(Z^n)+t_{{\rm b},n}^{\beta/2}+ t_{g,n}^{\beta/2}+\delta\big]}\le \beta.$$\end{theorem}

\noindent See \Cref{app:thm_4} for proof of \Cref{thm_4}. We present here an illustrative example, and in \Cref{sec:minmax} present an application of \Cref{thm_4} to a classical problem.
\begin{example}\label{ex_3_6}
Let $(X_i)$ be an i.i.d sequence with mean $0$ and variance $1$. Suppose that $X_i\in L_{12} $ and let $c_{\alpha}(X^n)$ be such that: $$P\bp{\ba{\big[\frac{1}{\sqrt{n}}\sum_{i\le n}Z^n_i\big]^2-\mathbb{E}\big(\big[\frac{1}{\sqrt{n}}\sum_{i\le n}Z^n_i\mid X\big]^2\big)}\ge c_{\alpha}(X)\big|X}\le \alpha.$$Denote $z_{\alpha}$  the $1-\alpha$ quantile of a standard normal: $P\bp{Z\ge z_{\alpha}}\le \alpha$ where $Z\ge N(0,1)$. Then the following holds: $$\limsup_{n\rightarrow 0}P\bp{\ba{\big[\frac{1}{\sqrt{n}}\sum_{i\le n}Z^n_i\big]^2-\mathbb{E}\big(\big[\frac{1}{\sqrt{n}}\sum_{i\le n}X_i\big]^2\big)}\ge c_{\alpha/2}(X)+z_{\alpha/4}^2}\le \alpha$$
\end{example}

\vspace{4mm}

\noindent The second method exploits the bootstrap method for slightly shifted observations. The goal is to  use the fact that under moderate conditions we know that $\bn{\bar{X^n}-\mathbb{E}(X_1^n)}$ is of size $O(1/\sqrt{n})$. In this goal, we denote $B_{d_n}(\gamma)$ the ball in $\mathbb{R}^{d_n}$ of radius $\gamma$ for the Euclidean-norm.
\begin{theorem}
\label{cafe_2}
 Let $(g_n)$ be a sequence of measurable functions. Suppose that for all sequence $(\mu_n)\in B_{d_n}(\gamma_n)$ \Cref{ass:approx} is satisfied by $(X^n+\mu_n)$ and $(g_n)$. Define $(\gamma_n)$ to be a sequence such that  $\frac{\sqrt{n}\gamma_n}{\log(d_n)}\rightarrow \infty.$
Set $t_{*}^{\alpha}(X^n)$ to be satisfying $$\sup_{\mu\in B_{d_n}(\gamma_n)}P\bp{\ba{g_n(Z^n+\mu)-\E\bp{g_n(Z^n+\mu)\mid X^n}}\ge t_{*}^{\alpha}(X^n) \mid X^n}\le \alpha.$$ 
Then the following holds:
$$\limsup_{\delta\downarrow 0}\limsup_{n\rightarrow 0}P\bp{\ba{g_n(Y^n)-\E\bp{g_n(Y^n)}}\ge t_{*}^{\alpha}(X^n)+\delta\mid X^n}\le\alpha.$$
\end{theorem}
See \Cref{app:cafe_2} for a proof. We apply this new result to the previous illustrative example.
\begin{example}\label{ex_3_7}
Let $(X_i)$ be an i.i.d sequence with mean $0$ and variance $1$. Suppose that $X_i\in L_{12} $ and let $c_{\alpha}(X^n)$ be such that: $$\sup_{x,|x|\le \log(n)/\sqrt{n}}P\bp{\ba{\big[\frac{1}{\sqrt{n}}\sum_{i\le n}Z_i+x\big]^2-\mathbb{E}\big(\big[\frac{1}{\sqrt{n}}\sum_{i\le n}Z_i+x\big]^2\big)}\ge c_{\alpha}(X^n)\big|X^n}\le \alpha.$$Then the following holds: $$\limsup_{n\rightarrow 0}P\bp{\ba{\big[\frac{1}{\sqrt{n}}\sum_{i\le n}Y_i\big]^2-\mathbb{E}\big(\big[\frac{1}{\sqrt{n}}\sum_{i\le n}Y_i\big]^2\big)}\ge c_{\alpha}(X^n)}\le \alpha$$
\end{example}

\subsubsection{Robust Confidence Interval}\label{sec:robust}

\noindent \Cref{thm_center} states that if  the mean $\E\bb{X_1^n}$ of the observations is known then we can instead study the centered bootstrap estimator, which under technical conditions, is asymptotically consistent. However assuming that the mean is known can be unrealistic. In this section, we instead assume that we know that it belongs to a certain subset $A_n$ and seek to find a confidence interval with a guaranteed coverage level for all potential values of the mean. {To make this more precise, we consider an adversary that can see the draw of the random samples and translate them by any offset in the translation set $B_n:=\{x-\E\bb{X_1^n}: x\in A_n\}$. Our goal is to guarantee that no-matter what perturbation the adversary chooses, we produce a confidence interval with guaranteed coverage.} Let $\mathcal{P}_n$ a set of probability distributions on $\mathbb{R}^{d_n}$ such that there exists a sequence of functions $(f_n)$ with $f_n\in \CC^3$ such that:
\begin{enumerate}
\item The functions $(f_n)$ approximate the estimators $(g_n)$: 
\begin{align}
    \sup_{\nu\in \mathcal{P}_n}\E_{X^n{\sim} \nu^{\otimes\infty}}\bp{\ba{ f_{n}(\tilde Z^n)-g_n(\tilde Z^n) }} + \E_{X^n{\sim} \nu^{\otimes \infty}}\bp{ \ba{f_{n}({X}^n)-g_n({X}^n) }}\xrightarrow{n\to\, \infty} 0. \tag{$H^{rob}_0$}
\end{align}
\item  The first, second and third order derivatives are such that:\begin{align}\sup_{\nu\in \mathcal{P}_n}\max\bp{R_{n,1}^{c,\nu}, R_{n,2}^{c,\nu}, R_{n,3}^{c,\nu}}\rightarrow 0.  \tag{$H_1^{rob}$}\end{align} where for each distribution $\nu\in \mathcal{P}_n$ we denoted by $R_{n,1}^{c,\nu}$, $R_{n,2}^{c,\nu}$ and $R_{n,3}^{c,\nu}$ the coefficients $R_{n,1}^c$, $R_{n,2}^c$ and $R_{n,3}^c$ computed for $(X^n_i)\overset{i.i.d}{\sim} \nu$.  

\end{enumerate}

\noindent Our goal we is to use the bootstrap method to find $(t_n^{\alpha}(X^n))$ such that the following holds:
\begin{equation}
\label{poirot}
\limsup_{n\rightarrow \infty}\sup_{\substack{\nu\in \mathcal{P}_n\\\E_{X\sim\nu}(X)\in A_n}}{{P_{X^n,Y^n\overset{i.i.d}{\sim} \nu}\bp{ \ba{ g_n(Y^{n})-\E\bb{g_n(Y^n)} }\ge t_{n}^{\alpha} (X^n)}}} \le \alpha.
\end{equation}

\noindent If conditions $(H_0^{{\rm rob}})$ and $(H_1^{{\rm rob}})$ hold then the bootstrap method can be used to find  a sequence $(t_n^{\alpha})$ such that \cref{poirot} holds (proof in \Cref{app:cafe}).
\begin{theorem}
\label{cafe}Let $(g_n)$ be a sequence of measurable functions, let $(\mathcal{P}_n)$ be sets of probability measures chosen such that  $(H_0^{\rm{rob}})$ and $(H_1^{\rm{rob}})$ hold. For all $\nu \in \mathcal{P}_n$ and given a sample $X^n\sim \nu$ define $t_{n}^{\alpha}(X^n)$ to be such that:

\begin{equation}
\sup_{\mu\in A_n}{P\bp{ \ba{ g_n(Z^{n}+\mu-\bar{X^n} )-\E\bb{g_n(Z^{n}+\mu-\bar{X^n} )\mid X^n} }\ge t_{n}^{\alpha} (X^n)\mid X^n}} \le \alpha.
\end{equation}
Then if we write $\mathcal{Q}_n:= \{\nu\in \mathcal{P}_n\mid \E_{X\sim \nu}(X)\in A_n \}$ then the following holds
\begin{equation}
\liminf_{\delta\downarrow 0}\limsup_{n\rightarrow \infty}\sup_{\nu\in \mathcal{Q}_n}{{P_{X^n,Y^n\overset{i.i.d}{\sim} \nu}\bp{ \ba{ g_n(Y^{n})-\E\bb{g_n(Y^n)} }\ge t_{n}^{\alpha} (X^n)+\delta}}} \le \alpha.
\end{equation}
\end{theorem}

\section{Illustrative Examples and Counterexamples}
\label{sec:illustrative}

We present a sequence of simple examples illustrating that our theorems hold even if the estimator is not asymptotically normal. Moreover, we provide negative examples where the shape of the confidence intervals obtained by the bootstrap method can be arbitrary compared to the ones of the original statistics $g_n(Y_{1:n})$. The first example  we consider are polynomials of the empirical average. Their limiting distribution is in not Gaussian for $p>1$.
\begin{example}\label{ex1} Let $p\in \mathbb{N}$ be an integer and let $X:=(X_{i})$ be an i.i.d sequence taking value in $\RR$  with mean $0$ and admitting a $12p$-th moment $\E\bb{\big|X_{i}|^{12p}}<\infty$. We define the functions $(g_n)$  as $g_n:x_{1:n}\rightarrow \bp{\frac{1}{\sqrt{n}}\sum_{i\le n}x_{i}}^p.$ We write $( Z^n_{{i}})$ a  bootstrap sample and $(Y^n_{i})$ an independent copy of $X$. Then the following holds: 
$$\bn{ d_{\mcF}\bp{ g_n\bp{ Z^n}, \bp{\sqrt{n}~\bar{X}^n+\sqrt{n}~\bar{Y}^n}^p \mid X } }_{L_1}=O\bp{\frac{1}{\sqrt{n}}}.$$
Moreover, let $A$ be an $1-\alpha$ confidence-interval for $g_n(\tilde Y_{1:n})$ meaning $P(g_n(\tilde Y_{1:n})\in A)\ge 1-\alpha$. Write: $A_{\bar{X}^n}:=\{x\in\RR\mid~\exists y\in A~{\rm s.t}~x={\rm sign}(x)]\bp{ |y|^{1/p}-\bar{X}^n}^p\}$ then $$P(g_n(Y_{1:n})\in A_{\bar{X}^n})\ge 1-\alpha.$$
\end{example}
See \Cref{pr:ex1} for the proof.

\begin{example}\label{stress} Let $X:=(X_i)$ be an i.i.d sequence of bounded real valued random variables satisfying $\E(X_1)=0$. We define $g_n:\times_{l=1}^n \RR\rightarrow\RR$ to be the following function: $g_n(x_1,\dots,x_n):= \sqrt{n}\Big[\prod_{i=1}^n\bp{1+\frac{x_i}{{n}}}-1\Big]$. Write $(\tilde Z_i^n) $ to be a centered bootstrap sample and let $Y:=(Y_i)$ be an independent copy of $X$. Then the following holds: $$\Big\|d_{\mcF}\bp{g_n( \tilde Z^n)-\E(g_n(\tilde Z^n)|X),g_n( Y_{1:n})-\E(g_n(Y^n))\Big|X}\Big\|_{L_1}\rightarrow 0.$$
See \Cref{app:stress} for a formal proof.
\end{example}

The next example demonstrates that the confidence intervals obtained by the bootstrap method are neither systematically bigger or smaller than the ones of original statistics.  

\begin{example}\label{ex32}
Let $(X_{i})$ be a sequence of i.i.d standard normal observations $X_{i}\sim N(0,1)$. Define $(g_n)$ to be the following sequence of functions: $g_n(x_{1:n}):=\bb{\frac{1}{\sqrt{n}}\sum_{i\le n}x_i}^+$. Let $(Z^n)$ be a bootstrap sample. The following holds:
\begin{equation}
\bn{ d_{\mcF}\bp{ g_n(Z^n),\, \sqrt{n}\bb{\bar{Y}+\bar{X}}^+ \mid X} }_{L_1}\rightarrow 0.
\end{equation}
Moreover given $\alpha<0.5$ and a sequence $(t_n)$ such that: $P(g_n(Z^n)\le t_n\mid X^n)=1-\alpha$ then: 
\begin{equation}
P(g_n(Y^n)\le t_n-\sqrt{n}\bar{X}^n)\sim 1-\alpha.
\end{equation}
We notice that the segment $[0,t]$ is  smaller than $[0,t-\sqrt{n}\bar{X}^n]$ only if $\sqrt{n}\bar{X}^n>0 $ which asymptotically happens with a probability of $1/2$.
\end{example}
See \Cref{pr:ex32} for a formal proof.

\noindent In the next example we show that our results apply to  classical quantities in mathematical physics. We consider the entropy of spin glasses configurations.

\begin{example}\label{roryl} 
Let $X:=(X_{i,j})$ be an array of i.i.d observations satisfying $X_{i,j}\overset{i.i.d}{\sim} N(0,1)$. We denote $X^n:=\bp{X_{i,j}}_{i,j\le n}$ the induced matrix and define $g_n: M_n(\RR)\rightarrow\RR$ to be the following function: $g_n(X):= \frac{1}{n}\log\bp{ \sum_{m\in \{-1,1\}^n}e^{\frac{1}{\sqrt{n}}m^{\top}X^nm} }$. Write $( Z_{i,j}^n) $ and $(\tilde Z_{i,j}^n)$ respectively a  bootstrap and centered bootstrap sample and $Y^n$ an independent copy of $X^n$. Then the following holds: 
\begin{equation}
\bn{ d_{\mcF}\bp{ g_n( Z^n),\, \frac{1}{n}\log\bp{ \sum_{m\in \{-1,1\}^n}e^{\frac{1}{\sqrt{n}}m^{\top} Y^n m}~e^{\bar{X}^n(\sum_i m_i)^2/\sqrt{n}} } \mid X^n } }_{L_1}\rightarrow 0;
\end{equation}and\begin{equation}
\bn{ d_{\mcF}\bp{ g_n(\tilde Z^n),\, \frac{1}{n}\log\bp{ \sum_{m\in \{-1,1\}^n}e^{\frac{1}{\sqrt{n}}m^{\top} Y^n m} } \mid X^n } }_{L_1}\rightarrow 0;
\end{equation}
where we have denoted $\bar{X}^n:=\frac{1}{n^2}\sum_{i,j\le n}X_{i,j}^n$.
\end{example}
See \Cref{app:roryl} for a formal proof.

\section{Uniform Confidence Bands}
\label{sec:bands}
In this section, we study the maximum of centered empirical processes. This is motivated by its application to uniform confidence bounds (see e.g. \cite{chernozhukov2013gaussian,chernozhukov2017central}). Let $(X_{i}^n)$ be a triangular array of i.i.d process with $X_1^n$ taking value in $\RR^{p_n}$ where $(p_n)$ is an increasing sequence. We want to estimate the distribution of
$$\max_{j\le p_n}\frac{1}{\sqrt{n}}\sum_{i\le n} X_{i,j}^n-\E\big[ X_{1,j}^n\big].$$
For fast growing sequences of $(p_n)$ this statistics is not asymptotically Gaussian \cite{deng2017beyond}. Therefore to study its distribution one might want to use the bootstrap method. Using our results we recover the results of \cite{deng2017beyond} and establish conditions under which the bootstrap is asymptotically consistent (proof in \Cref{app:seule}).
\begin{prop}\label{seule} Let $(p_n)$ be a sequence of integers satisfying $\log(p_n)=o(n^{1/4})$. Define $(X^n_i)$ to be a triangular array of sequences of i.i.d random variables taking value in $\RR^{p_n}$. We suppose that $\bn{\sup_{k\le p_n} \ba{ X_{1,k}^n }}_{L_{12}}<\infty$. 
We denote $\mathcal{M}_n(x_{1:n})=\max_{j\le p_n} \frac{1}{\sqrt{n}}\sum_{i\le n}x_{i,j}$.
Then the following holds {\small $$\Big\|d_{\mcF}\bp{\mathcal{M}_n(Z^n-\bar{X}^n),~\mathcal{M}_n(Y^n-\mathbb{E}[Y^n_1])\mid X^n}\Big\|_{L_1}=o(1).$$}
\end{prop}

\section{Value of a Min-Max Objective}
\label{sec:minmax}

In this section, motivated by problems in sample average approximations of stochastic linear programs \cite{guigues2017non} and in structural econometric problems \cite{syrgkanis2017inference}, we want to estimate distribution of the value of a stochastic min-max objective. Given a set of potential actions $\{\theta_i\}$, a payoff function $f_n$ and random states of nature $(\xi^n_\ell)$ we wish to estimate the expected payoff of a stochastic min-max objective:{
\begin{equation}
\theta_0 :=\min_{i\le p_{n}}\max_{j\le p_{n}} \E\bb{f_n(\xi^n_{\ell}, u_i, v_j)}.
\end{equation}}
We assume we have access to $n$ i.i.d. samples of $\xi^n:=(\xi_{1^n}, \ldots, \xi_{n}^n)$ and want to build confidence intervals for $\theta_0$. To achieve this, we will estimate the distribution of the root-$n$ normalized version of the empirical analogue of the stochastic program:
\begin{equation}
\min_{i\le p_n}\max_{j\le p_n} \frac{1}{\sqrt{n}}\sum_{\ell\le n}f_n(\xi^n_\ell, u_i, v_j).
\end{equation}
We will investigate the consistency of the bootstrap method for this problem. For ease of notations, we  write $X_{\ell,i,j}^n:= f_n(\xi^n_\ell, i, j)$, and set $(p_n)$ to be an increasing sequence and write:
\begin{equation}
    \hat{\theta}_n(x_{1:n}):=\min_{i\le p_{n}}\max_{j\le p_{n}} ~\frac{1}{n}\sum_{\ell\le n} x_{\ell,i,j}.
\end{equation}
We show that the centered bootstrap method is consistent as long as  $\log(p_{n})=o(n^{1/4})$ (proof in \Cref{app:prop3}).

\begin{prop}\label{prop3}
Let $(p_n)$ be a sequence of integers satisfying $\log(p_n)=o(n^{1/4})$. Define $(X^n_i)$ to be a triangular array of sequences of i.i.d random variables taking value in $\RR^{p_n\times p_n}$. Suppose that $\|\sup_{l_1,l_2\le p_n}|X_{1,l_1,l_2}^n|\|_{L_{12}}<\infty$.
~Then the following holds  $$\bn{d_{\mcF}\bp{\sqrt{n}\Big[\hat \theta_n( Z^n)-\theta_{Z^n}\Big],~\sqrt{n}\Big[\hat \theta_n( \tilde Y^n)-\theta_{\tilde Y^n}\Big]\mid X^n}}_{L_1}=o(1);$$     $$\bn{d_{\mcF}\bp{\sqrt{n}\bp{\hat \theta_n( \tilde Z^n)-\theta_{\tilde Z^n}},~\sqrt{n}\bp{\hat \theta_n(  Y^n)-\theta_{0}}\mid X^n}}_{L_1}=o(1);$$ where  for a process $(Y'_{1:n})$ we have written, by abuse of notations, 
$\theta_{Y'}:=\min_{i\le p_{n}}\max_{j\le p_{n}} ~\E\big[Y^{'}_{l,i,j}|X^n\big].$ 
Therefore if we let $t_{n,\alpha}$ to be a threshold such that:
\begin{equation}
    \Pr\bp{ \sqrt{n} \left(\hat{\theta}_n(\tilde{Z}^n) - \theta_{\tilde{Z}^n}\right) \geq t_{n,\alpha} \mid X^n } = \alpha
\end{equation}
Then for every sequence $(\epsilon_n)$ satisfying $\epsilon_n\downarrow 0$ and $\bp{\frac{\log(p_n)^2}{\sqrt{n}}}^{1/9}=o(\epsilon_n)$ we have
\begin{equation}
    \limsup_{n\rightarrow\infty}\Pr\bp{\sqrt{n}\left(\hat{\theta}(X^n) - \theta_0\right) \geq t_{n,\alpha}+\epsilon_n} \le \alpha 
\end{equation}
\end{prop}

\vspace{2mm} 

\noindent We note that if the means are sufficiently spaced:$$\frac{\log(p_n)}{\sqrt{n}}=o\Big[\inf_{(i,j)\ne (l,k)}\big|\E\big[X^{n}_{1,l,k}\big]-\E\big[X^{n}_{1,i,j}\big]\big|\Big]$$ then the bootstrap method is consistent. See \Cref{seperate} for a precise statement and a proof.

\noindent We generalize those results to estimating the payoff of minmax strategies over a continuous space. Let $(d'_n)$ be a non-decreasing sequence and denote $B_{d'_n}^1\subset \RR^{d'_n}$  the ball of radius $1$ in  $\RR^{d'_n}$. Choose $(f_n:\prod_{i=1}^n \RR^{d_n}\times B_{d'_n}^1\times B_{d'_n}^1) $ to be a sequence of Lipschitz functions for which there are constants $(C_n)$ and $(c_n)$ satisfying:
\begin{align}
    \ba{ f_n(x, u, v) - f_n(x, u', v') } < ~& C_n \bp{ \|u-u'\| + \|v-v'\| } \tag{$H^{{\rm minmax}}_0$} \label{minimax:lip:1}\\
    \ba{ \E\bb{ f_n(X_1^n, u, v)}-\E\bb{f_n(X_1^n, u', v'} } >~& c_n \bp{ \|u-u'\|+\|v-v'\| } \tag{$H^{{\rm minmax}}_1$} \label{minimax:lip:2}
\end{align}

\noindent Similarly as in the discrete case our goal is to estimate \begin{equation}
\theta_n:=\inf_{u \in B_{d'_n}^1}\sup_{v \in B_{d'_n}^1} \E\bb{f_n(\xi_1^n, u, v)}
\end{equation}
and to do so we propose as an estimator
\begin{equation}
\hat{\theta}_n(X^n):=\inf_{u\in B_{d'_n}^1}\sup_{v\in B_{d'_n}^1} \frac{1}{{n}}\sum_{i\le n} f_n(\xi_i^n, u, v).
\end{equation}
To build confidence intervals around $\theta_n$, we want to use the bootstrap method. Using \Cref{thm1}, we prove that it is consistent under distinct set of assumptions:   1) If
the dimensions $(d'_n)$ grow as $o(n^{1/4})$ and if 
the sequences $(C_n)$ and $(c_n)$ are bounded respectively from bellow and above  or 2)  If
the dimensions $(d'_n)$ grow as $o(n^{1/7})$ and if 
the sequences $(C_n)$ is bounded from above (proof in \Cref{app:prop4}). 
\begin{prop}\label{prop4}
Let  $\bp{f_n:\prod_{i=1}^n \RR^{d_n}\times B_{d'_n}^1\times B_{d'_n}^1}$  be a sequence of Lipschitz functions 
satisfying condition \eqref{minimax:lip:1}.  Assume that $\sup_{u, v\in B_{d'_n}^1, n\le\infty}\|f_n(\xi_1^n, u, v)\|_{L_{\infty}}<\infty$.

\noindent If in addition  we know that \eqref{minimax:lip:2} holds and that $d'_nC_n\log(nC_n)=o(c_nn^{1/4})$. Then the bootstrap method is consistent:
$$\Big\|d_{\mcF}\bp{\sqrt{n}\big[\hat{\theta}_n(Z^n)-\theta_{Z^n}\big], ~\sqrt{n}\big[\hat{\theta}_n(\xi^n_{1:n})-\theta_n\big]\Big|\xi^n}\Big\|_{L_1}\rightarrow 0,$$
where  for a process $(Y'_{1:n})$ we have written
$\theta_{Y'}:=\inf_{u\in B_{d'_n}^1}\sup_{v\in B_{d'_n}^1} ~\E\big[f_n(Y^{'}_1,u,v)|X^n\big].$ 

\noindent Otherwise we suppose that  $d'_n\log(nC_n)C_n=o(n^{1/7})$, and choose $t_{\beta/2}(X^n)$ such that the following holds:
$$P\bp{\sqrt{n}\ba{\hat \theta_n( Z^n)-\E\bp{\hat \theta_n( Z^n)\mid X^n}}\ge t_{\beta/4}(X^n)\mid X^n}\le \beta/4.$$
 We have:
 $$\lim_{\delta \downarrow 0}\limsup_{n\rightarrow \infty} P\bp{\sqrt{n}\ba{\hat \theta_n( \xi^n)-\theta_{X^n}}\ge t_{\beta/4}(X^n)+3t_{g,n}^{*,\beta/4}+\delta}\le \beta;$$ where $t_{g,n}^{*\beta/4}$ is chosen such that: $P\bp{\sup_{u,v\in B_{d'_n}^1}|N_{u,v}|\ge t_{g,n}^{*,\beta/4}}\le \beta/4 $ where $(N_{u,v})$ is a gaussian process with covariance function $\Sigma^2\bp{(u_1,v_1),(u_2,v_2)}:=\rm{Cov}\bp{f_n(\xi_1^n,u_1,v_1),f_n(\xi_1^n,u_2,v_2)}$.
\end{prop}

\section{P-value of a Two-Sample Kernel Test}\label{sec:p_value}

In this subsection, we show how the bootstrap method can be used to obtain consistent p-values for kernel two sample tests. Given two independent i.i.d processes $(X_{i,1}^n)$ and $(X_{i,2}^n)$ taking value in $\mathcal{X}_n\subset\RR^{d_n}$, the goal of two-sample tests is to determine if the observations $(X_{i,1}^n)$ and $(X_{i,2}^n)$ are sampled from the same distribution. 
For ease of notations, we designate by $\mu_{n,1}$ and $\mu_{n,2}$ respectively the distribution of the first sample and second sample; and want test if the null hypothesis holds $$(H_0^{n}):  \mu_{n,1}=\mu_{n,2}$$ against the alternative $$(H_1^n): \mu_{n,1}\ne \mu_{n,2}.$$ A popular method to do so are non-parametric kernel two samples tests \cite{liu2020learning,gretton2012kernel,wilson2016deep,wenliang2019learning}.

\noindent Let $\mcF_n$ be a class of functions from $\mathcal{X}_n$ into $\RR$. If the two distributions are the same $\mu_{n,1}=\mu_{n,2}$ then we have:$$\sup_{f\in \mcF_n}\big|\E(f(X_{1,1}^n ))-\E(f(X_{1,2}^n))\big|=0.$$
Moreover if $\mcF_n$ is dense in the space of bounded continuous functions then the opposite also holds. The main difficulty therefore consists of choosing the set $\mcF_n$ to be big enough to differentiate between the distributions $\mu_{n,1}$ and $\mu_{n,2}$ but structured enough that we can estimate of $\sup_{f\in \mcF_n}\big|\E(f(X_{1,1}^n))-\E(f(X_{1,2}^n))\big|$. To do so, we choose a reproducing kernel space $\mathcal{H}_n$ with kernel $K_n:\mathcal{X}_n\times \mathcal{X}_n\rightarrow \RR$ and set the class of functions $\mcF_n$ to be the unit ball of $\mathcal{H}_n$. Different choices of kernels will lead to various level of power  of our test especially for structured or high dimensional data. The goal is to choose the kernel that is the most likely to maximize the power of the test. 
 
\noindent Let $\big(K_{\theta_k}(\cdot,\cdot)\big)_{k\le p_n}$ be a finite set of potential Kernel candidates.
We write for all $i,j\le n$ and for all $k\le p_n$
$$H_{i,j}^{\theta_k}:=K_{\theta_k}(X^{n}_{j,1},X^{n}_{i,1})+K_{\theta_k}(X^{n}_{j,2},X^{n}_{i,2})-K_{\theta_k}(X^{n}_{j,1},X^{n}_{i,2})-K_{\theta_k}(X^{n}_{j,2},X^{n}_{i,1});$$
and for all subsets $B\subset \dbracket{n}$ we denote $\hat{M}_{\theta_k}(X^n_{B}):=\frac{1}{|B|^2}\sum_{i,j\in B}H_{i,j}^{\theta_k}$.
~The idea proposed in \cite{liu2020learning} is to select the  kernel that gives rise to a test with the highest (estimated) power. This is done by  selecting a subset $B_n\subset\dbracket{n}$ and maximizing the following quantity $\hat{\theta}_n^{B_n}:={\rm argmax}_{\theta\in \{\theta_1,\dots,\theta_{p_n}\}}~ p_{\theta}(X^{n}_{B_n})$ where we have set 
$$p_{\theta}(X^{n}_{B_n}):=\frac{\hat{M}_{\theta}(X^n_{B_n})}{\frac{4}{|B_n|^3}\sum_{i\in B_n}\big[\sum_{j\in B_n} H_{i,j}^{\theta}]^2-\frac{4}{|B_n|^4}\Big[\sum_{i,j\le B_n} H_{i,j}^{\theta}\Big]^2+\lambda_n}$$ where $(\lambda_n)$ are tuning parameters. Once the kernel is chosen the test statistics is computed on $\dbracket{n}
\setminus B_n$ the remaining data: $\frac{1}{n^2}\sum_{i,j\in \dbracket{n}\setminus B_n} H^{\hat{\theta}_n}_{i,j}$. The fact that the kernel is chosen on a different sample than the test statistics is computed on, means that the conditional limiting distribution of the test statistics, under $H_0$, is known to be a chi-square \cite{liu2020learning}. Hence one can compute a consistent estimate of the p-value. However under this approach only a portion of the data is used to select the kernel. This could be problematic when dealing with high-dimensional kernels. 

\noindent We propose a different method that does not require data splitting and uses the bootstrap method to estimate the p-value. The test statistics that we propose is a softmax: 
$$\hat{T}_n(X^n):= \sum_{k\le p_n} \frac{1}{n^2}\sum_{i,j\le n}  H^{\theta_k}_{i,j}~ \omega_k(X_{1:n}),\qquad {\rm where}~\omega_k(X_{1:n}):=\frac{e^{\beta_np_{\theta_k}(X^n)}}{\sum_{k'\le p_n}e^{\beta_np_{\theta_k'}(X^n)}}~;$$ and where  $(\beta_n)$ are hyper-parameters. 
The bigger $\beta_n$ is the more weight we give to the kernel maximizing $p_{\theta}(X^n)$.

\noindent We note that the distribution of $\hat T_n$ is unknown and depends in an intricate fashion on the set of kernels $\{K_{\theta_k},~k\le p_n\}$ as well as on $p_n$. Therefore to be able to compute the p-value we want to estimate its distribution under $H_0$. In this goal, we remark that under the null hypothesis the distribution of $X_{i,1}^n$ and $X_{i,2}^n$ are the same which implies that the samples are interchangeable $\bp{X_{i,1}^n,X_{i,2}^n}\overset{d}{=}\bp{X_{i,2}^n,X_{i,1}^n}$. It is therefore natural to compare the distribution of $(X^n_i)$ to the corresponding randomly permuted process.  This is the idea behind permutation tests \cite{liu2020learning}. In general, for an i.i.d random process $(\tilde X_i)$ taking value in $\RR^2$ we define the process  $(\tilde X^M_i)$ obtained by randomly permuting the observations $\tilde X_{i,1}$ and $\tilde X_{i,2}$: $$\tilde X^M_{i}:=\begin{cases}\tilde X_i~{\rm with~probability~}0.5\\(\tilde X_{i,2},\tilde X_{i,1})^T~{\rm with~probability~}0.5.\end{cases}$$ We note that this permuted process has coordinates with identically distributed coordinates $\tilde X^M_{1,1}\overset{(a)}{=}\tilde X^M_{1,2}$. Moreover, we have $d_{W}\bp{\tilde X_1,\tilde X_1^M}\le d_{W}\bp{\tilde X_{1,1},\tilde X_{1,2}}$ and if the distribution of $(\tilde X_i)$ are already in $H_0$ then its distribution is left invariant by those permutations. We show that the bootstrap method allows us to estimate the p-value consistently even when $p_n$ grows exponentially fast (proof in \Cref{app:nulle}).

\begin{prop}\label{nulle}
Let $(X_i^n):=\big((X^{n}_{i,1},X^{n}_{i,2})\big)$ be a triangular array of i.i.d processes. Let $\{K_{\theta_k},~k\le p_n\}$ be a sequence of positive definite continuous kernels. We suppose that $$\max_{k\le p_n}\rm{tr}(K_{\theta_k})<\infty;\quad \frac{\beta_n\log(p_n)D_n^4}{\lambda_n^2}=o({n}^{1/6});$$ where we shorthanded $D_n:=\max\bp{\big\|\sup_{k\le p_n}K_{\theta_k}(X_{1,1}^{M},X_{1,1}^{M})\big\|_{L_{120}},~1}$. Let $(Y_i^n)$ be an independent copy of $(X_i^n)$ and $(Z^n_{i})$ be bootstrap samples of $(X_i^n)$. We have:
$$\Big\|d_{\mcF}\Big(n\hat{T}_n( Z^M_{1:n}),n\hat{T}_n(Y^M_{1:n})\mid X^n\Big)\Big\|_{L_1}\rightarrow 0.$$
\end{prop}

\section{Empirical Risk of Smooth Stacked Ensemble Estimator}
\label{sec:stacking}

A ubiquitous and popular approach for model selection and ensembling in machine learning practice is known as stacking \cite{wolpert1992,breiman1996stacked,Laan2007}. Given a set of trained \emph{base estimators} $\{\hat{\theta}^1,\dots,\hat{\theta}^{p_n}\}$, for example representing a fitted neural network, a random forest and a nearest-neighbour estimator, we call the \emph{smooth-stacked estimator} the linear ensemble of those estimators $\{\hat{\theta}^k\}$ weighted by coefficients that are related to the out-of-sample risk of each estimator. An important question: if we use all the samples to estimate the weights of the ensemble, then can we construct confidence intervals on the risk of the ensemble estimator? 

\noindent The most straightforward version of stacking is to put all the weight on the model with the smallest out-of-sample risk. Other approaches proposed in practice are to fit a linear regression model using the outputs of each model as an input co-variate to the linear model and using the learned coefficients as coefficients on the ensemble \cite{Laan2007}. 

\noindent In this subsection, we propose a smooth version of stacking that adds stability to the chosen ensemble, while putting most weight on the best performing model. This ensemble can be viewed as a regularized instance of the linear regression stacking approach where an entropic regularizer is added to the square loss objective. This regularization adds smoothness and stability to the chosen ensemble and allows us to show that the distribution of the ensemble's risk can be estimated with the bootstrap, even if the all the data are used to estimate the weights or fit the base models.

\noindent Let $(X_i^n)$ be a triangular array of i.i.d observations taking value in $\RR^{d_n}$; and let $(m_n)$ be  an increasing sequence. {Define $\mathcal{F}_n$ as the space of measurable functions from $\times_{n=1}^{\infty} \RR^{d_n}$ to $\RR^{d'_n}$.} We estimate $p_n$ different estimators  $\Omega_n:=\big\{\hat\theta_n^k(X_{1:m_n}^n),~k\le p_n\big\}$ built on the first $m_n$ data-points. The loss is measured by  a common loss function {$\mathcal{L}_n:\RR^{d_n}\times\RR^{d'_n}\rightarrow\RR$} and the empirical risk of the $k$-th estimator is computed on all the remaining $n-m_n$ data points as:
\begin{equation}
\mcR^k_n(x_{1:n}):= \frac{1}{n-m_n}\sum_{u= m_n+1}^{n} \mathcal{L}_n(x_u,\hat{\theta}^{k}_n(x_{1:m_n})(x_u)).    
\end{equation}
The smooth-stacked estimator is defined as the following  ensemble learner \begin{equation}
    \hat\Theta_n(x_{1:n})=\sum_{k\le p_n}\hat\theta^k_n(x_{1:m_n})\frac{e^{-\beta_n\mcR^k_n(x_{1:n})}}{\sum_{k'\le p_n}e^{-\beta_n\mcR^{k'}_n(x_{1:n})}}.
\end{equation}
We denote the empirical risk of an estimator $\Theta$ as
{\begin{equation}
    \mcR^{{\rm s}}_{\Theta}(x_{1:n}):=\frac{1}{\sqrt{n-m_n}}\sum_{i= m_n+1}^{n} \mathcal{L}_n\big(x_i,\Theta(x_i)\big).
\end{equation}}
Let $(Z^n)$ be a bootstrap sample of $(X^n_i)_{i\ge m_n}$. We show that  the bootstrap method is systematically consistent if  and only if  $\beta_n=o(\sqrt{n-m_n})$. 
For simplicity we suppose that the estimators $\hat\theta_n^k$  have bounded coordinates; and that the loss function $\mathcal{L}_n$ is smooth, and have bounded partial derivatives in its second argument.%

\noindent We write the set of all convex combinations of the estimators:
\begin{equation}
\Omega\bp{\bc{\theta_p,\, p\le p_n }}:=\bc{ \sum_{p\le p_n} \omega_p \theta_p 
\mid \omega_p\ge 0 ~{\rm and}~ \sum_{p\le p_n}\omega_p=1 }
\end{equation}
and introduce the following notations:
\begin{equation}
\begin{split}
T_{n}:=~& \sup_{\ell\le d'_n}\bn{ \sup_{p\le p_n}\ba{\hat\theta^p_{n,\ell}(X_{1:m_n}^n)} }_{L_{\infty}}\lor 1,\\
L_{n}:=~& \bn{ \sup_{p\le p_n}\ba{\mathcal{L}_n\bp{ X_{n}^n,\hat\theta^p_{n}(X^n_{1:m_n})(X_n^n) } }}_{L_{\infty}}\lor~\sup_{\ell \le d'_n} \bn{ \sup_{\theta\in \Omega\bp{\bc{\hat{\theta}_p(X_{1:m_n})~p\le p_n }}} \partial_{2,\ell} \ba{\mathcal{L}_n(X_{n}^n,\theta(X_n^n))} }_{L_{\infty}}\lor 1,
\end{split}\end{equation} where by $\partial_{2,\ell}\mathcal{L}_n(x,y)$ we designate $\partial_{y_{l}}\mathcal{L}_n(x,y)$. We show that if the following hypothesis \eqref{ass:h1ss} holds then the bootstrap method is asymptotically consistent (proof in \Cref{app:common_stacking_msbien}). 
\begin{align}
    \frac{\beta_n}{\sqrt{n-m_n}d'_n} L_nT_{n}~e^{\frac{\beta_n}{n-m_n}L_n}\longrightarrow 0. \tag{$H_1^{\rm stacked}$} \label{ass:h1ss}
\end{align}

\begin{prop}\label{common_stacking_msbien}
Choose $(m_n)$, $(\beta_n)$ and $(p_n)$ be increasing sequences.
Let $(X_i^n)$ be a triangular array of i.i.d observations taking value in $\RR^{d_n}$. Set $(\mathcal{L}_n:\RR^{d_n}\times \RR^{d_n'}\rightarrow \RR)$ to be a sequence of smooth  loss functions. Let $(Z_i^n)$ and $(Y_i^n)$ be respectively a bootstrap sample and an independent copy of $(X_{m_n+1}^n,\dots,X_{n}^n)$. 
Suppose that the hypothesis \eqref{ass:h1ss} holds then we have:
\begin{equation}
 \begin{split}
     \Big\|d_{\mcF}\Big(&\mcR^{{\rm s}}_{\hat\Theta_n}(Z^n_{m_n+1:n})-\E\big[\mcR^{{\rm s}}_{\hat\Theta_n}(Z^n_{m_n+1:n})\big|\hat\Theta_n\big],~ \mcR^{{\rm s}}_{\hat\Theta_n}(Y^n_{m_n+1:n})-\E\big[\mcR^{{\rm s}}_{\hat\Theta_n}(Y^n_{m_n+1:n})\big|\hat\Theta_n\big]\mid X^n\Big)\Big\|_{L_1}\rightarrow 0;
 \end{split}
\end{equation}
where we have shorthanded $\hat\Theta_n:=\hat\Theta_n(X^n)$.
Therefore if we choose $t_{n,\alpha}(X^n)$ to be such that:
\begin{equation}
 \begin{split}
   P\bp{\ba{\mcR^{{\rm s}}_{\hat\Theta_n}(Z^n_{m_n+1:n})-\E\big[\mcR^{{\rm s}}_{\hat\Theta_n}(Z^n_{m_n+1:n})\big|\hat\Theta_n\big]}\ge t_{n,\alpha}(X^n)\mid X^n}\le \alpha
 \end{split}
\end{equation}
then the following holds
\begin{equation}
 \begin{split}
  \limsup_{n\rightarrow \infty} P\bp{\ba{\mcR^{{\rm s}}_{\hat\Theta_n}(Y^n_{m_n+1:n})-\E\big[\mcR^{{\rm s}}_{\hat\Theta_n}(Y^n_{m_n+1:n})\big|\hat\Theta_n\big]}\ge t_{n,\alpha}(X^n)\mid \hat\Theta_n}\le \alpha
 \end{split}
\end{equation}
\end{prop}

\noindent If $\beta_n$ grows proportionally to  $\beta_n \propto \sqrt{n-m_n}$ then the bootstrap method is not a systematically consistent estimator of the risk of the smooth stacked estimator. We present a simple example illustrating this.
\begin{example}\label{nulle_mt}
Let $(X_i)$ be a process of i.i.d random variables taking value in $\RR$. Suppose that $X_{1}\sim N( 0,1)$. We choose $m_n=\lfloor n/2\rfloor$ and $\beta_n=\sqrt{n}$ and define the estimators $(\hat\theta_n^1,\hat\theta_n^2)$ as constantly equal to $$\hat\theta_n^1(X_{1:n}):=1~\quad {\rm and}~\quad \hat\theta_n^2(X_{1:n}):=-1.$${We shorthand by $\hat \Theta_n$ the corresponding stacked estimator.} We choose the loss function $\mathcal{L}$ to be the square loss $\mathcal{L}(x,\theta):=(x-\theta)^2$. Let $(Z_1,Z_2)\sim N(0,\begin{bmatrix}4 & 0\\ 0&1\end{bmatrix})$ be a Gaussian vector. 
Then the asymptotic centered distribution of the empirical loss is $Z_1+Z_2 {\rm tanh}(4Z_2)$.
However the asymptotic distribution of the bootstrap empirical loss  is $$\mcR^{{\rm s}}_{\hat\Theta_n}( Z^n_{m_n+1:n}) -\E(\mcR^{{\rm s}}_{\hat\Theta_n}( Z^n_{m_n+1:n})\big|\hat{\Theta}_n)\xrightarrow{d}
Z_1+Z_2{\rm tanh}\Big(4Z_2+4\sqrt{n}\bar{X}_{m_n+1:n}\Big).$$
Therefore  the bootstrap method is not asymptotically consistent.
\end{example}

\noindent{As \cref{nulle_mt} just demonstrated the bootstrap method is not in general consistent for the empirical risk of the stacked estimator. One of the reason for this is the dependence between the weights of the stacked model  and the bootstrap samples $Z^n$. We therefore slightly adapt the bootstrap method by bootstrapping both the weights and the observations on which the loss is estimated. We establish the limiting distribution of this boostrap estimator as long as $\beta_n=O(\sqrt{n-m_n})$. }

\noindent For ease of notations, we denote \begin{equation}\begin{split}&
L_{n}^*:=\sup_{\substack{i\le 3\\l_{1:i}\le d'_n}}\bn{\max_{\substack{p\le p_n}}\ba{\partial^i_{2,l_{1:i}}\mathcal{L}_n\bp{X^n_n,\hat\theta^p_{n}(X^n_{1:m_n})(X^n_n)}}}^{1/i}_{L_{\infty}}\lor 1.\end{split}\end{equation}%
We establish the limiting distribution of our bootstrap estimate under the following hypothesis: 
\begin{align}
\frac{\beta_nd'_n\log(p_n)^{2/3}}{{(n-m_n)}^{2/3}}L_nT_n\longrightarrow 0 \tag{$H_1^{\rm st~bis}$} \label{ass:h1bis}
\end{align}

 \begin{prop}\label{nulle_mt2}
 Choose $(m_n)$, $(\beta_n)$ and $(p_n)$ be increasing sequences.
Let $(X_i^n)$ be a triangular array of i.i.d observations taking value in $\RR^{d_n}$. Set $(\mathcal{L}_n:\RR^{d_n}\times \RR^{d_n'}\rightarrow \RR)$ to be a sequence of smooth loss functions verifying condition $(H_1^{\rm st~bis})$.
Let $(Z_i^{n,1})$ and  $(Z_i^{n,2})$ be independent bootstrap samples; and $(Y_i^{n,1})$ and $(Y_i^{n,2})$ be independent copies of $(X_{m_n+1:n}^n)$.

\noindent Then we have:
  \begin{equation}
     \begin{split}         \Big\|d_{\mcF}\Big(&\mcR^{{\rm s}}_{\hat\Theta_n^{Z^{n,2}}}(Z_{1:n-m_n}^{n,1})-\E\bp{\mcR^{{\rm s}}_{\hat\Theta_n^{Z^{n,2}}}(Z_{1:n-m_n}^{n,1})\mid\hat\Theta_n^{Z^{n,2}},X^n},\\&\quad \mcR^{{\rm s}}_{\hat\Theta_n^{'}}(Y_{1:n-m_n}^{n,1})-\E\bp{\mcR^{{\rm s}}_{\hat\Theta_n^{'}}(Y_{1:n-m_n}^{n,1})\Big|\hat\Theta_n^{'}}\mid X^n\Big)\Big\|_{L_1}\rightarrow 0;
     \end{split}
 \end{equation} where we have set $\hat\Theta_n^{Z^{n,2}}:=\hat\Theta_n(X^n_{1:m_n}Z_{1:n-m_n}^{n,2})$ and defined $$\hat\Theta_n^{'}:=\sum_{p\le p_n}\hat\theta_n^p(X_{1:m_n}^n)\frac{e^{-\beta_n \big[\mcR_n^p(X_{1:m_n}^nY_{1:n-m_n}^{n,2})+\tilde \mcR_n^p(X^n)\big]}}{\sum_{p'\le p_n}e^{-\beta_n \big[\mcR_n^{p'}(X_{1:m_n}^nY_{1:n-m_n}^{n,2})+\tilde\mcR_n^{p'}(X^n)\big]}};$$ where we wrote   $\tilde \mcR_n^p(X^n):=\mcR_n^p(X^n)-\sqrt{n-m_n}\E\big[\mathcal{L}_n(X_n^n, \theta_p^n(X_{1:m_n}^n))\big| \theta_p^n(X_{1:m_n}^n)\big]$.
 
\end{prop} 

\noindent We notice that this implies that the bootstrap method is in general not consistent if $\beta_n=o(\sqrt{n-m_n})$, as it is illustrated in \cref{nulle_mt}.

\bibliography{references}
\bibliographystyle{plain}
\newpage
\appendix

\hypersetup{linkcolor=black}
\newpage
\renewcommand{\contentsname}{Contents of Appendix}
\tableofcontents
\addtocontents{toc}{\protect\setcounter{tocdepth}{3}}

\section{Preliminary Lemmas and Notation}\label{app:preliminaries}

If a function $f$ is three-times differentiable then we let:
\begin{align}
    \partial_{i,k} f(x_{1:n}) ~:=~& \partial_{x_{i,k}} f(x_{1:n})
    &
    \partial_if(x_{1:n}) ~:=~& (\partial_{i,1}f(x_{1:n}),\dots, \partial_{i,d}f(x_{1:n}))^{\top}\\
    \partial^2_{i,k_{1:2}} f(x_{1:n}) ~:=~& \partial_{x_{i,k_1}} \partial_{ x_{i,k_2}} f(x_{1:n})
    &
    \partial^2_i f(x_{1:n}) ~:=~& \bp{ \partial^2_{i,k_{1:2}} f(x_{1:n}) }_{k_1,k_2\le d}\\
    \partial^3_{i,k_{1:3}}f(x_{1:n}) ~:=~& \partial_{x_{i,k_1}} \partial_{x_{i,k_2}} \partial_{x_{i,k_3}} f(x_{1:n})
    &
    \partial_i^3 f(x_{1:n}) ~:=~ & \bp{ \partial^3_{i,k_{1:3}}f(x_{1:n}) }_{k_1,k_2,k_3\le d}
\end{align}

\subsection{Preliminary results}

\begin{lemma}\label{casser}
Let $(\tilde X^n_i)$ be an array of martingale differences taking value in $\RR^{p_n}$. Suppose that $\bn{ \max_{k\le p_n}\tilde{X}^n_{1,k}}_{L_p}<\infty$ where $p\ge 3$. Then there exists a constant $C_p$, that does not depend on the distribution of $(\tilde{X}_i^n)$, such that
\begin{equation}
    \bn{ \max_{k\le p_n}\frac{1}{\sqrt{n}}\sum_{i\le n}\tilde{X}^n_{i,k} }_{L_p}
    \le \bn{ \max_{k\le p_n}\tilde X^n_{i,k} }_{L_p} \bp{ 1+C_p \bp{ \log(p_n) +\frac{\log(p_n)^2}{\sqrt{n}} } }
\end{equation}
Thus if $\log(p_n)=o(n^{1/4})$ then
\begin{equation}
    \bn{ \max_{k\le p_n}\frac{1}{\sqrt{n}}\sum_{i\le n}\tilde{X}^n_{i,k} }_{L_p}
    = O\bp{ \log(p_n) \bn{ \sup_{k\le p_n}\tilde{X}^n_{i,k} }_{L_p} }.
\end{equation}
Moreover let $(X^n_i)$ be a triangular array of i.i.d process and $(g_{k,n})$ be sequences of measurable functions, for each $k\in [p]$. Then:
\begin{equation}
     \bn{ \max_{k\le p_n} g_{k,n}\bp{ X^n } }_{L_p}
    =O\bp{ \log(p_n)\, \sqrt{n}\, \bn{ \sup_{k\le p_n} \ba{ g_{k,n}\bp{ X^n } - g_{k,n}\bp{ X^{n, i} } } }_{L_p} },
\end{equation}
where we have defined $X^{n, i}_j:=\begin{cases} X^n_j, & {\rm if}~~ j\ne i\\  X'_i, & {\rm if}~~ i=j\end{cases}$ with $(X_i')$ being an independent copy of $( X^n_i)$.
\end{lemma}

See \Cref{app:casser} for proof of \Cref{casser}.

\begin{lemma}\label{lem:expand-collapse}
For any set of random variables $U_1,\ldots, U_m$:
\begin{equation}
    \E\left[ \left(\sum_{t\leq m} U_t\right)^d \right] \leq \left( \sum_{t\leq m} \|U_t\|_{L_d} \right)^d
\end{equation}
\end{lemma}
\begin{proof} By expanding the polynomial, applying a repeated version of Cauchy-Schwarz inequality and collapsing the polynomial again, we get:
\begin{align}
    \E\left[ \left(\sum_{t\leq m} U_t\right)^d \right] =~& \E\left[ \sum_{t_{1:d} \leq m} \prod_{\ell=1}^d U_{t_{\ell}} \right]
    \leq \sum_{t_{1:d} \leq m} \prod_{\ell=1}^d \|U_{t_{\ell}}\|_{L_d}
    \leq \left( \sum_{t\leq m} \|U_t\|_{L_d} \right)^d
\end{align}
\end{proof}

\begin{lemma}\label{lem:df-is-metric}
The distribution distance $d_{\mcF}$ satisfies the triangle inequality.
\end{lemma}
\begin{proof}
For any three random variables $X, Y, Z$:
\begin{align}
    d_\mcF(X, Z) :=~& \sup_{h\in \mcF} \E[h(X)] - \E[h(Z)]\\
    =~& \sup_{h\in \mcF} \E[h(X)] - \E[h(Y)] + \E[h(Y)] - \E[h(Z)]\\
    \leq~& \sup_{h\in \mcF} \E[h(X)] - \E[h(Y)] + \sup_{h\in \mcF} \E[h(Y)] - \E[h(Z)]
    =: d_\mcF(X, Y) + d_\mcF(Y, Z)
\end{align} 
\end{proof}
\begin{lemma}\label{lem:trans_df}
The distribution distance $d_{\mcF}$ is translation invariant: For all random variables $X$ and $Y$ and all constant $z$ we have $$d_{\mathcal{F}}\bp{X,Y}=d_{\mathcal{F}}\bp{X-z,Y-z}$$
\end{lemma}
\begin{proof}
For all $h\in \mathcal{F}$ define $h_{z}x\rightarrow h(x-z)$. We have:
\begin{align}&
{\E{h(X-z)-h(Y-z)}}={\E{h_z(X)-h_z(Y)}}\le d_{\mcF}\bp{X,Y}.
\end{align}
As this holds for all $h\in \mcF $ it implies  that 
\begin{align}&
d_{\mcF}\bp{X-z,Y-z}\le d_{\mcF}\bp{X,Y}.
\end{align}
The reverse inequality is proved in exactly the same fashion.
\end{proof}
\begin{lemma}\label{lemm:tv}
Let $p_1,p_2$ be two distributions that are uniformly continuous with respect to a measure $\mu$. Then the following holds:
$$1-\frac{1}{2}e^{-KL(p_2,p_1)}\ge\|p_1(\cdot)-p_2(\cdot)\|_{TV} $$
\end{lemma}
\begin{proof}We denote $f_1,f_2$ the Radon-Nikodym densities of respectively $p_1$ and $p_2$ with respect to $\mu$.
By the Cauchy-Swartz inequality we have:
\begin{align}
    &
    \bp{\int \sqrt{f_1(x)f_2(X)}d\mu(x)}^2
    \\&\le    \bp{\int \sqrt{\min(f_1(x),f_2(x))\max(f_1(x),f_2(X))}d\mu(x)}^2
        \\&\le    {\int {\min(f_1(x),f_2(x))d\mu(x)\int\max(f_1(x),f_2(X))d\mu(x)}}
        \\&\le    {\int {\min(f_1(x),f_2(x))d\mu(x)\bp{2-\int\min(f_1(x),f_2(X))}d\mu(x)}}
        \\&\le 2(1-\|p_1(\cdot)-p_2(\cdot)\|_{TV}).
\end{align}
Moreover by another application of Cauchy-Swartz we know that
\begin{align}
    &
    \bp{\int \sqrt{f_1(x)f_2(X)}d\mu(x)}^2
    \\&=e^{  \log\bp{\bp{\int \sqrt{f_1(x)f_2(X)}d\mu(x)}^2}}
  \\&=e^{ 2 \log\bp{{\int\frac{ \sqrt{f_1(x)}}{\sqrt{f_2(x)}}{f_2(X)}d\mu(x)}}}
      \\&\le e^{-KL(p_2,p_1)}
\end{align}
Therefore by combining those two inequalities we obtain that:
$$1-\frac{1}{2}e^{-KL(p_2,p_1)}\ge\|p_1(\cdot)-p_2(\cdot)\|_{TV} $$
\end{proof}
\section{Extensions and Variations of Main Theorem}

In this section we present some supplementary results that have been motivated in the main body of the article.

\subsection{Alternative Condition to \eqref{ass:2}}

As mentioned in the main body of the text the results also hold under a slightly modified condition \eqref{ass:2}. We denote $\|\cdot\|_{v,d_n}$, $\|\cdot\|_{m,d_n}$ and $\|\cdot\|_{t,m_n}$ respectively the $L_1$ norm for $d_n$ dimensional vectors, $d_n\times d_n$ dimensional matrices and $d_n\times d_n\times d_n$ dimensional tensors.
We define the following quantities:
\begin{equation}
\begin{split} 
R^{*}_{1,n}:=~&  2n^{\frac{1}{3}}\sup_{i\le n}\Big\|\big\|\partial_if_n( Z^{n, i, \bar{X}^n}_{1:n})\big\|_{v,d_n} \Big\|_{L_{12}}\Big\|\sup_{i\le d_n}|X_{1,i}^{n}|\Big\|_{L_{12}}\\
R^{*}_{2,n}:=~& 4n^{1/2}\sup_{i\le n}\Big\|\big\|\partial_{i}^2f_n(Z^{n, i, \bar{X}^n})\big\|_{m,d_n} \Big\|_{L_{12}}\Big\|\sup_{i\le d_n}|X_{1,i}^{n}|\Big\|^2_{L_{12}}\\
R^{*}_{3,n}:=~& 8
n\sup_{i\le n}\Big\|\max_{x\in [\bar{X}^n,\tilde Y_i^n]\cup[\bar{X}^n,Z_1^n]}\big\|\partial_{i}^3f_n(Z^{n,i,x})\big\|_{t,d_n} \Big\|_{L_{12}}\Big\|\sup_{i\le d_n}|X_{1,i}^{n}|\Big\|^4_{L_{12}}.
 \end{split}\end{equation}
We assume that the functions $(f_n)$ satisfy the following conditions: \begin{align}\label{ass:2-star}
    \log(d_n) \bp{ (R_{n,1}^{*})^2\max\bc{ \frac{1}{n^{1/6}},R^*_{n,1} }+ R_{n,2}^{*} } + R_{n,3}^{*} \rightarrow 0. \tag{$H_1^*$}
\end{align}

\begin{theorem}\label{jardin}Let $(g_n:\times_{l=1}^n\RR^{d_n}\rightarrow \RR)$ be a sequence of symmetric measurable symmetric functions. Let $(X^n_i)$ be a triangular array of i.i.d processes such that $X^n_1\in L_{12}$. Assume that there is a sequence $(f_n)$ of measurable functions satisfying condition \eqref{ass:1} and \eqref{ass:2-star}. 
Then there exists a universal constant $K$ such that:
\begin{align}
\bn{d_{\mcF}\Big(g_n( Z^n),g_n(\tilde Y^n)\mid X^n\Big) }_{L_1}
\le~& \Big\|g_n(\tilde Y^n) -f_n(\tilde Y^n)\Big\|_{L_1}
+\big\|g_n( Z^n)-f_n(Z^n)\big\|_{L_1}
\\
& +K \bp{ \log(d_n) \bp{ (R_{n,1}^{*})^2\max\bc{ \frac{1}{n^{1/6}}, R^*_{n,1} }+ R_{n,2}^{*} } + R_{n,3}^{*} }
\end{align} 
\end{theorem}

\subsection{Extension to Exchangeable Sequences and Random Estimators $(g_n)$}

In this subsection we generalize \cref{thm1} to random estimators $(g_n)$  and to exchangeable processes $(X_i^n)$.
We say that a process $(X_i^n)$ is exchangeable if and only if for all permutations $\pi\in \mathbb{S}(\mathbb{N})$ and all indexes $i_1,\dots,i_k\in \mathbb{N}$ we have:$$(X_{i_1}^n,\dots,X_{i_k}^n)\overset{d}{=}(X_{\pi(i_1)}^n,\dots,X_{\pi(i_k)}^n).$$We designate by $\tau(X^n)$ the tail $\sigma-$algebra of $X^n$ which is defined as $\tau(X^n):=\bigcap_{i=1}^{\infty}\sigma\big(X^n_{k},~k\ge i\big).$ By the De Finitti theorem we know that $(X_i^n)$ is exchangeable if and only if conditionally on $\tau(X^n)$ the process $(X_i^n)$ is an i.i.d process.

We assume that the sequence of (potentially random) functions $(g_n)$  is such that there is a net of (potentially random) three-times differentiable functions $(f_n)$ respecting conditions \eqref{ass:1} and \eqref{ass:2}. 
We establish under those conditions the limiting distribution of the bootstrap estimator.

\begin{theorem}\label{thm8}
Let $(g_n:\times_{l=1}^n\RR^{d_n}\rightarrow \RR)$ be a sequence of measurable  functions. Let $(X^n_i)$ be a triangular array of exchangeable processes such that $X^n_1\in L_{12}$. Assume that there is a net $(f_n)$ of (potentially random) functions satisfying \Cref{ass:approx}. 
Let $Y^n:=(Y^n_i)$ be a process that is, conditionally on $\tau(X^n)$, an independent copy of $X^n$ that is also independent from $(g_n)$. Define $(Z^n)$ to be a boostrap sample of $X^n$ that is independent of $(g_n)$ conditionally on $X^n$. Then there exists a universal constant $K$ such that:
\begin{align}
\bn{ d_{\mcF}\bp{ g_n( Z^n), g_n(\tilde{Y}^n) \mid X^n } }_{L_1}
\leq~& \bn{ g_n(\tilde{Y}^n) -f_{n}(\tilde{Y}^n) }_{L_1}
+ \bn{ g_n(Z^n)-f_{n}(Z^n) }_{L_1} \\
&~ + K\, \bp{ R_{n,1}^2\max\bp{ \frac{1}{n^{1/6}}, R_{n,1}} + R_{n,3} + R_{n,2} }
\\
\longrightarrow~& 0.
\end{align} 
\end{theorem}

\begin{remark}
We note that \Cref{thm2}, \Cref{thm_center}, \Cref{cafe} and \Cref{jardin}, can be generalized in the exact same fashion.
\end{remark}

\section{Proof of \Cref{thm1}, \Cref{jardin}, \Cref{cafe} and \Cref{thm8}}\label{app:thm1}

As the proof of \Cref{thm1} and \Cref{jardin} are very similar, we present the proof for \Cref{thm1} and highlight the differences with the proof of \cref{jardin}. The proof of \Cref{thm1} and \Cref{thm8} are identical.

Throughout the proofs we will use the following notations. We write $({X^{c}_i})$ and $(\tilde Y^{c}_i)$ the re-centered processes, around the empirical mean:
\begin{align}
    X^{c}_i :=~& X^n_i-\bar{X}^n, &
    \tilde{Y}^{c}_i :=~& \tilde{Y}^n_i-\bar{X}^n, &
    Y^{c}_i :=~& Y^n_i-\bar{X}^n, &
    Z^{c}_i :=~& Z^n_i-\bar{X}^n.
\end{align} 

\subsection{Main Lemmas}

\begin{lemma}[Approximation Error]\label{lem:approx}
Let $(f_n)$ be a sequence of $\CC^3$ functions that approximates $(g_n)$ as designated by \Cref{ass:1}.
Then:
\begin{multline}\label{wed_dr_2}
  \bn{d_{\mcF}\bp{ g_n(Z^n), g_n(\tilde{Y}^n)\mid X^n }}_{L_1} \leq \bn{
  d_{\mcF}\bp{ f_n(Z^n), f_n(\tilde{Y}^n) \mid X^n }}_{L_1}\\
    + \bn{ g_n(\tilde{Y}^n) - f_{n}(\tilde{Y}^n) }_{L_1} + \bn{ g_n(Z^n) - f_n(Z^n) }_{L_1}
\end{multline}

\end{lemma}
\begin{proof}
Let $(f_n)$ be a sequence of $\CC^3$ functions that approximate approximate $(g_n)$ as designated by \Cref{ass:approx}. By Condition~\eqref{ass:1} and the fact that for all $h\in \mcF$, $\sup_{x\in \RR} |h'(x)| \leq 1$, we have:
\begin{align}
\forall h\in \mcF: \bn{ \E\bb{ h\bp{ g_n( \tilde{Y}^n) } - h\bp{ f_n( \tilde{Y}^n)} \mid X^n }}_{L_1}
 \leq~& \bn{ \E\bb{ \ba{ g_n(\tilde{Y}^n) - f_{n}(\tilde{Y}^n) }  \mid X^n } }_{L_1}\\
 \leq~& \bn{ g_n(\tilde{Y}^n) - f_{n}(\tilde{Y}^n) }_{L_1}
\end{align}
\begin{align}
\forall h\in \mcF: \bn{ \E\bb{ h\bp{ g_n( Z^n) } - h\bp{ f_n( Z^n) } \mid X^n } }_{L_1}
 \leq~& \bn{ \E\bb{ g_n(Z^n) - f_n(Z^n) \mid X^n } }_{L_1} \\
 \leq~& \bn{ g_n(Z^n) - f_n(Z^n) }_{L_1}
 \end{align}

Thus we can conclude that:
\begin{multline}
   \bn{ d_{\mcF}\bp{ g_n(Z^n), g_n(\tilde{Y}^n)\mid X^n }}_{L_1} \leq\bn{ d_{\mcF}\bp{ f_n(Z^n), f_n(\tilde{Y}^n) \mid X^n }}_{L_1}\\
    + \bn{ g_n(\tilde{Y}^n) - f_{n}(\tilde{Y}^n) }_{L_1} + \bn{ g_n(Z^n) - f_n(Z^n) }_{L_1}
\end{multline}

Therefore it is enough to study the metric distance between the distributions of $f_n({Z}^n)$ and $f_n(\tilde{Y}^n)$. 
\end{proof}

\begin{lemma}[Lindenberg Path Decomposition]\label{lem:telescoping}
For any statistic $f_n$ and $i\in [n]$, let:
\begin{equation}
    A_i := \bn{ \sup_{h\in \mcF} \E\bb{ h\bp{f_n\bp{Z^{n,i}}}-h\bp{f_n\bp{Z^{n,i-1}}} \mid X^n } }_{L_1}
\end{equation}

Then:
\begin{equation}
    \bn{ d_{\mcF}\bp{ f_n(Z^n), f_n(\tilde{Y}^n) \mid X^n } }_{L_1} \leq \sum_{i=1}^n A_i
\end{equation}

\end{lemma}
\begin{proof}
By the triangle inequality and writing the difference between $h(f_n(\tilde{Y}^n))$ and $h(f_n(Z^n))$ as a ``Lindenberg'' telescoping sum of interpolating differences, we have for all $h\in \mcF$:
\begin{align}
\bn{ d_{\mcF}\bp{ f_n(Z^n), f_n(\tilde{Y}^n) \mid X^n } }_{L_1} =~& \bn{ \sup_{h\in \mcF} \E\bb{ h\bp{f_n\bp{\tilde{Y}^n}} - h\bp{f_n\bp{Z^n}} \mid X^n } }_{L_1}\\
=~& \bn{\sup_{h\in \mcF} \E\bb{ h\bp{ f_n\bp{Z^{n,n}} } - h\bp{f_n\bp{Z^{n, 0}}}\mid X^n } }_{L_1}\\
=~& \bn{\sup_{h\in \mcF} \sum_{i=1}^n \E\bb{ h\bp{ f_n\bp{Z^{n, i}} } - h\bp{f_n\bp{Z^{n, i-1}}} \mid X^n } }_{L_1}\\
\leq~& \sum_{i=1}^n \bn{\sup_{h\in \mcF} \E\bb{ h\bp{f_n\bp{Z^{n,i}}}-h\bp{f_n\bp{Z^{n,i-1}}} \mid X^n } }_{L_1}
\end{align}

\end{proof}

\begin{lemma}[Third-Order Approximation of Test Function]\label{lem:test_taylor}
For any statistic $f_n$, let:
\begin{align}
\bar{f}^{i}_n :=~& f_n\bp{ Z^{n, i, \bar{X}^n} } = f_n(\tilde{Y}_1^n, \ldots, \tilde Y_{i-1}^n, \bar{X}^n, Z_{i+1}^n,\ldots, Z_n^n)\\
    \Delta_i(f_n,x) :=~& f_n(Z^{n,i,x}) - \bar{f}^{i}_n.
\end{align}

Then each quantity $A_i$ as defined in \Cref{lem:telescoping} can be bounded as:
\begin{equation}\label{wed_dr}
    A_i \leq \QQ_{1i} + \QQ_{2i} + \QQ_{3i}
\end{equation}
with:
\begin{align}
    \QQ_{1i} :=~& \bn{ \E\bb{ \Delta_i(f_n, \tilde{Y}_i^n)  \mid Z^{n,i, \bar{X}^n}, X^n} - \frac{1}{n}\sum_{\ell=1}^n \Delta_i(f_n, X_\ell^n)}_{L_1} \label{wed_dr_3}\\
    \QQ_{2i} :=~& \frac{1}{2}\bn{ \E\bb{ \Delta_i(f_n, \tilde{Y}_i^n)^2  \mid Z^{n,i, \bar{X}^n}, X^n} - \frac{1}{n} \sum_{\ell=1}^n \Delta_i(f_n, X_\ell^n)^2 }_{L_1} \label{wed_dr_4}\\
    \QQ_{3i} :=~& \frac{1}{6} \bp{ \bn{\Delta_i(f_n, \tilde{Y}_i^n)}^3_{L_3} + \bn{\Delta_i(f_n, Z_i^n)}^3_{L_3} }
\end{align}
\end{lemma}
\begin{proof}
By centering around $h(\bar{f}_n^i)$ we can re-write $A_i$ as:
\begin{align}
    A_i = \bn{\sup_{h\in \mcF} \E\bb{ h\bp{ f_n\bp{Z^{n,i}}}-h\bp{\bar{f}^{i}_n} - h\bp{ f_n\bp{ Z^{n,i-1} }} + h\bp{\bar{f}^{i}_n} \mid X^n } }_{L_1}
\end{align}

Applying a third-order Taylor expansion of each difference around $\bar{f}_n^i$ and using the fact that $h\in\mcF$ has third order derivatives, uniformly bounded by $1$:
\begin{equation}
\begin{aligned}
    A_i \leq~& \bn{\sup_{h\in \mcF}  \E\bb{ h'\bp{\bar{f}^{i}_n}\, \bp{ f_n\bp{Z^{n,i}}-\bar{f}^{i}_n-\bp{ f_n\bp{Z^{n,i-1}}-\bar{f}^{i}_n} } \mid X^n } }_{L_1}\\
    & + \frac{1}{2} \bn{\sup_{h\in \mcF} \E\bb{ h''\bp{\bar{f}^{i}_n}\, \bp{ \bp{f_n\bp{Z^{n,i}}-\bar{f}^{i}_n}^2- \bp{f_n\bp{Z^{n,i-1}}-\bar{f}^{i}_n}^2 } \mid X^n } }_{L_1}\\
    & + \frac{1}{6} \bp{ \bn{ f_n\bp{Z^{n,i}} - \bar{f}^{i}_n}^3_{L_3} + \bn{f_n\bp{Z^{n,i-1}} - \bar{f}^{i}_n }^3_{L_3} } =: \II_1 + \II_2 + \II_3
\end{aligned}
\end{equation}

\paragraph{Bounding $\II_1$.} We now upper bound the $\II_1$ term. Observe that:
\begin{align}
    \II_1 :=~& \bn{ \sup_{h\in \mcF} \ba{ \E\bb{ h'\bp{\bar{f}^{i}_n}\, \bp{ f_n\bp{Z^{n,i}}-\bar{f}^{i}_n-\bp{ f_n\bp{Z^{n,i-1}}-\bar{f}^{i}_n} } \mid X^n } } }_{L_1}\\
    =~& \bn{ \sup_{h\in \mcF}  \ba{ \E\bb{ h'\bp{\bar{f}^{i}_n}\, \bp{ \Delta_i(f_n, \tilde{Y}_i^n) - \Delta_i(f_n, Z_i^n) } \mid X^n } } }_{L_1}
\end{align}
By a tower law of expectations and the fact that $|h'(\bar{f}_n^i)|\leq 1$, for all $h\in \mcF$:
\begin{align}
\II_1 =~&
    \bn{ \sup_{h\in \mcF} \E\bb{ \E\bb{h'\bp{\bar{f}^{i}_n}\, \bp{ \Delta_i(f_n, \tilde{Y}_i^n) - \Delta_i(f_n, Z_i^n) } \mid Z^{n,i, \bar{X}^n}, X^n}  \mid X^n } }_{L_1}\\
    =~& \bn{ \sup_{h\in \mcF} \E\bb{h'\bp{\bar{f}^{i}_n}\, \E\bb{ \Delta_i(f_n, \tilde{Y}_i^n) - \Delta_i(f_n, Z_i^n) \mid Z^{n,i, \bar{X}^n}, X^n}  \mid X^n } }_{L_1}\\
    \leq~& \bn{ \sup_{h\in \mcF} \E\bb{\ba{h'\bp{\bar{f}^{i}_n}}\, \ba{\E\bb{ \Delta_i(f_n, \tilde{Y}_i^n) - \Delta_i(f_n, Z_i^n) \mid Z^{n,i, \bar{X}^n}, X^n} }  \mid X^n } }_{L_1}\\
    \leq~& \bn{ \E\bb{\ba{\E\bb{ \Delta_i(f_n, \tilde{Y}_i^n) - \Delta_i(f_n, Z_i^n) \mid Z^{n,i, \bar{X}^n}, X^n} } \mid X^n } }_{L_1}\\
    =~& \bn{ \E\bb{ \Delta_i(f_n, \tilde{Y}_i^n) - \Delta_i(f_n, Z_i^n) \mid Z^{n,i, \bar{X}^n}, X^n} }_{L_1}
\end{align}

Moreover, observe that conditional on $X^n$ and $Z^{n, i, \bar{X}^n}$, the only thing that varies in the random variable $\Delta_i(f_n, Z_i^n)$ is $Z_i^n$. Moreover, $Z_i^n$ is distributed uniformly over $\{X_1^n, \ldots, X_n^n\}$, conditional on $Z^{n, i, \bar{X}^n}, X^n$ (since conditional on $X^n$, $Z_i^n$ is independent of $\tilde{Y}^n$):
\begin{equation}
    \E\bb{\Delta_i(f_n, Z_i^n) \mid Z^{n,i, \bar{X}^n}, X^n} = \frac{1}{n} \sum_{\ell=1}^n \Delta_i(f_n, X_\ell^n)
\end{equation}
We can then conclude that:
\begin{equation}
    \II_1 \leq \bn{ \E\bb{ \Delta_i(f_n, \tilde{Y}_i^n)  \mid Z^{n,i, \bar{X}^n}, X^n} - \frac{1}{n}\sum_{\ell=1}^n \Delta_i(f_n, X_\ell^n)}_{L_1}
\end{equation}

\paragraph{Bounding $\II_2$.} Observe that:
\begin{align}
    \II_2 :=~& \frac{1}{2} \bn{\sup_{h\in \mcF} \E\bb{ h''\bp{\bar{f}^{i}_n}\, \bp{ \bp{f_n\bp{Z^{n,i}}-\bar{f}^{i}_n}^2- \bp{f_n\bp{Z^{n,i-1}}-\bar{f}^{i}_n}^2 } \mid X^n } }_{L_1}\\
    =~& \frac{1}{2} \bn{\sup_{h\in \mcF} \E\bb{ h''\bp{\bar{f}^{i}_n}\, \bp{ \Delta_i(f_n, \tilde{Y}_i^n)^2- \Delta_i(f_n, Z_i^n)^2 } \mid X^n } }_{L_1}
\end{align}
By a tower law of expectations and the fact that $|h''(\bar{f}_n^i)|\leq 1$, for all $h\in \mcF$:
\begin{align}
\II_2 =~&
    \frac{1}{2}\bn{ \sup_{h\in \mcF} \E\bb{ \E\bb{ h''\bp{\bar{f}^{i}_n}\, \bp{ \Delta_i(f_n, \tilde{Y}_i^n)^2- \Delta_i(f_n, Z_i^n)^2} \mid Z^{n,i, \bar{X}^n}, X^n}  \mid X^n } }_{L_1}\\
    =~& \frac{1}{2}\bn{ \sup_{h\in \mcF} \E\bb{ h''\bp{\bar{f}^{i}_n}\, \E\bb{\Delta_i(f_n, \tilde{Y}_i^n)^2- \Delta_i(f_n, Z_i^n)^2\mid Z^{n,i, \bar{X}^n}, X^n}  \mid X^n } }_{L_1}\\
    \leq~& \frac{1}{2}\bn{ \sup_{h\in \mcF} \E\bb{\ba{h''\bp{\bar{f}^{i}_n}}\, \ba{\E\bb{ \Delta_i(f_n, \tilde{Y}_i^n)^2- \Delta_i(f_n, Z_i^n)^2\mid Z^{n,i, \bar{X}^n}, X^n} }  \mid X^n } }_{L_1}\\
    \leq~& \frac{1}{2}\bn{ \E\bb{\ba{\E\bb{ \Delta_i(f_n, \tilde{Y}_i^n)^2- \Delta_i(f_n, Z_i^n)^2 \mid Z^{n,i, \bar{X}^n}, X^n} } \mid X^n } }_{L_1}\\
    =~& \frac{1}{2}\bn{ \E\bb{ \Delta_i(f_n, \tilde{Y}_i^n)^2- \Delta_i(f_n, Z_i^n)^2 \mid Z^{n,i, \bar{X}^n}, X^n} }_{L_1}
\end{align}
By the conditional independence reasoning we presented in the bound for $\II_1$, the latter can further be written as:
\begin{equation}
\II_2 \leq \frac{1}{2}\bn{ \E\bb{ \Delta_i(f_n, \tilde{Y}_i^n)^2  \mid Z^{n,i, \bar{X}^n}, X^n} - \frac{1}{n} \sum_{\ell=1}^n \Delta_i(f_n, X_\ell^n)^2 }_{L_1} 
\end{equation}
\paragraph{Bounding $\II_3$.} We simply observe that $\II_3$ can be re-written as:
\begin{align}
    \II_3 :=~& \frac{1}{6} \bp{ \bn{ f_n\bp{Z^{n,i}} - \bar{f}^{i}_n}^3_{L_3} + \bn{f_n\bp{Z^{n,i-1}} - \bar{f}^{i}_n }^3_{L_3} }\\
    =~& \frac{1}{6} \bp{ \bn{\Delta_i(f_n, \tilde{Y}_i^n)}^3_{L_3} + \bn{\Delta_i(f_n, Z_i^n)}^3_{L_3} }
\end{align}
\end{proof}

\begin{lemma}[Third Order Approximation of Smooth Statistic]\label{lem:c3-stats}
Consider any statistic $f_n\in \CC^3$ and any random vector $V\in \RR^d$. Consider the random variables:
\begin{align}
    \dfmidi :=~& \partial_i f_n\bp{Z^{n, i, \bar{X}^n}}\\
    \ddfmidi :=~& \partial_i^2 f_n\bp{Z^{n, i, \bar{X}^n}}
\end{align}
If $\|V_k\|_{L_{12}}\leq \|X_k^n\|_{L_{12}}$ and \eqref{ass:2} holds, then:
\begin{equation}
    \bn{ \Delta_i(f_n, V)- \dfmidi^{\top} V^c -\frac{1}{2}\bp{V^{c}}^{\top}\ddfmidi V^c }_{L_3} \leq \frac{ R_{n,3}}{6n}
\end{equation}
If $\|\sup_{k\leq d} V_k\|_{L_{12}} \leq \|\sup_{k\leq d} X_k^n\|_{L_{12}}$ and \eqref{ass:2-star} holds, then:
\begin{equation}
    \bn{ \Delta_i(f_n, V)-\dfmidi^{\top} V^c -\frac{1}{2}\bp{V^{c}}^{\top}\ddfmidi V^c }_{L_3} \leq \frac{ R_{n,3}^*}{6n}
\end{equation}
\end{lemma}
\begin{proof}
Since $f_n$ is three-times differentiable, for any random vector $V\in \RR^d$ with $\|V_k\|_{L_{12}}\leq \|X_k\|_{L_{12}}$, for all $k\in [d]$, if we let $V^c = V - \bar{X}^n$, then by a Taylor expansion and \Cref{lem:expand-collapse}:
\begin{multline}
 \bn{ \Delta_i(f_n, V)-\dfmidi^{\top} V^c -\frac{1}{2}\bp{V^{c}}^{\top}\ddfmidi V^c }_{L_3}\\
 \leq \frac{1}{6} \bn{ \sup_{x\in [\bar{X}^n, V]}\sum_{k_1,k_2,k_3\le d_n}\partial^3_{i, k_{1:3}}  f_n(Z^{n, i,x}) V_{k_1}^c V_{k_2}^c V_{k_3}^c }_{L_3}
 \leq \frac{1}{6}\sum_{k_1, k_2, k_3\le d_n} D_{k_{1:3}}^{n}(f_n)
 \le\frac{ R_{n,3}}{6n}. 
\end{multline}
where we used the fact that $\|V_k^c\|_{L_{12}}\leq \|V_k\|_{L_{12}} + \|\bar{X}_k^n\|_{L_{12}} \leq 2 \|X_k^n\|_{L_{12}} = M_k^n$. The second part of the lemma follows along identical lines, but in the second-to-last inequality we instead bound by:
\begin{align}
    \frac{1}{6} \bn{ \sup_{x\in [\bar{X}^n, V]}\sum_{k_1,k_2,k_3\le d_n}\ba{ \partial^3_{i, k_{1:3}}   f_n(Z^{n, i,x})} \bp{\sup_{k\leq d} V_{k}^c}^3}_{L_3}
\end{align}
By a Cauchy-Schwarz inequality the latter is upper bounded by:
\begin{align}
    \frac{1}{6} \bn{ \sup_{x\in [\bar{X}^n, V]}\sum_{k_1,k_2,k_3\le d_n} \ba{ \partial^3_{i, k_{1:3}} f_n(Z^{n, i,x})}}_{L_{12}}  \bn{  \sup_{k\leq d} V_{k}^c}_{L_{12}}^4
\end{align}
Since $\|\sup_{k\leq d} V_k\|_{L_{12}} \leq \|\sup_{k\leq d} X_k^n\|_{L_{12}}$, we also have that: $\|\sup_{k\leq d} V_k^c\|_{L_{12}} \leq 2\,\|\sup_{k\leq d} X_k^n\|_{L_{12}}$. By the definition of $R_{n,3}^*$, we get the result.
\end{proof}
   
\subsection{Proof of \Cref{thm1}}

\begin{lemma}[Bounding $\QQ_{3i}$ under \eqref{ass:2}]\label{lem:third-order}
For any statistic sequence $(f_n)$, with $f_n\in \CC^3$ that satisfies \eqref{ass:2}, we have for all $i\in [n]$:
\begin{align}
    \max\bc{ \|\Delta_i(f_n, Z_i^n)\|_{L_3}^3,\, \|\Delta_i(f_n, \tilde{Y}_i^n)\|_{L_3}^3 } \leq \frac{9}{n}\bp{ R_{n,1}^3+\frac{1}{2\sqrt{n}} R_{n,2}^3+\frac{1}{6n^2} R_{n,3}^3 }
\end{align}
Therefore:
\begin{equation}
    \QQ_{3i} \leq \frac{18}{n} \bp{ R_{n,1}^3+\frac{1}{2\sqrt{n}} R_{n,2}^3+\frac{1}{6n^2} R_{n,3}^3 }
\end{equation}
\end{lemma}
\begin{proof}
By \Cref{lem:c3-stats} and the fact that for any $a,b\in \RR$: $|a+b|^3\le 3 \left(|a|^3 +|b|^3\right)$, we have that:
\begin{align}
\bn{\Delta_i(f_n, \tilde{Y}_i^n) }_{L_3}^3 \leq~& 3 \bn{ \dfmidi^{\top} \tilde{Y}_i^{c}+\frac{1}{2} \bp{\tilde{Y}_i^{c}}^{\top}\ddfmidi {{\tilde{Y}_i^{c}}} }^3_{L_3} + \frac{R_{n,3}^3}{72\, n^3}\\
\leq~& 9 \bn{ \dfmidi^{\top} \tilde{Y}_i^{c} }_{L_3}^3 + \frac{9}{8} \bn{ \bp{\tilde{Y}_i^{c}}^{\top}\ddfmidi{{\tilde{Y}_i^{c}}} }^3_{L_3} + \frac{R_{n,3}^3}{72\, n^3}
\end{align}
Moreover, by \Cref{lem:expand-collapse} we have:
\begin{align}
\E\bb{ \ba{ \dfmidi^{\top} \tilde{Y}_i^{c} }^3 }
\leq~& \bp{ \sum_{k\le d_n}\|\tilde Y_{i,k}^c\|_{L_6}\|\dfmidik[k]\|_{L_6} }^3\le \frac{\bp{R_{n,1}}^3}{n} \\
 \E\bb{ \ba{ \bp{\tilde{Y}_i^{c}}^{\top}\ddfmidi{{\tilde{Y}_i^{c}}} }^3 }
\leq~& \bp{ \sum_{k_1, k_2\le d_n}\|\tilde Y_{i,k_1}^c\|_{L_9}\|\tilde Y_{i, k_2}^c\|_{L_9}\|\ddfmidik\|_{L_9} }^3
\leq \frac{(R_{n,2})^3}{n^{3/2}}
\end{align}
Combining the above yields the bound on $\|\Delta_i(f_n, \tilde{Y}_i^n)\|_{L_3}^3$. The same bound on $\|\Delta_i(f_n, Z_i^n)\|_{L_3}^3$ can be obtained in an identical manner. The lemma then follows.
\end{proof}

\begin{lemma}[Bounding $\QQ_{1i}$ under \eqref{ass:2}]
For any statistic sequence $(f_n)$, with $f_n\in \CC^3$, which satisfies \eqref{ass:2}, we have for all $i\in [n]$:
\begin{equation}
    \QQ_{1i} \leq \frac{R_{n,3}}{3\,n}+\frac{2\,R_{n,2}}{n}
\end{equation}
\end{lemma}
\begin{proof}
Applying \Cref{lem:c3-stats} for $V=X_{\ell}^n$ and $V=\tilde{Y}_i^n$ and replacing the terms $\Delta_i(f_n, X_\ell^n)$ and $\Delta_i(f_n, \tilde{Y}_i^n)$ in $\QQ_{1i}$, with their corresponding second degree approximations, we have:
\begin{multline}
    \QQ_{1i} \leq \frac{R_{n,3}}{3n} + \bn{ \dfmidi^{\top} \bp{ \E\bb{\tilde{Y}_i^n \mid X^n} - \bar{X}^n } }_{L_1} \\
    + \frac{1}{2} \bn{ \frac{1}{n}\sum_{\ell=1}^n \bp{X_\ell^c}^{\top} \ddfmidi X_{\ell}^c - \E\bb{\bp{\tilde{Y}_i^c}^{\top} \ddfmidi \tilde{Y}_i^c \mid Z^{n,i,\bar{X}^n}, X^n} }_{L_1}
\end{multline}
\emph{Importantly, observe that by the definition of $\tilde{Y}_i^n$, we have that $\E\bb{\tilde{Y}_i^n\mid X^n} = \bar{X}^n$. Thus the first order term in this expansion vanishes}. Hence:
\begin{equation}
    \QQ_{1i} \leq \frac{R_{n,3}}{3n} + 
    \frac{1}{2} \bn{ \frac{1}{n}\sum_{\ell=1}^n \bp{X_\ell^c}^{\top} \ddfmidi X_{\ell}^c - \E\bb{\bp{\tilde{Y}_i^c}^{\top} \ddfmidi \tilde{Y}_i^c \mid Z^{n,i,\bar{X}^n}, X^n} }_{L_1}
\end{equation}
We can further split the second term on the right hand side as:
\begin{multline}\label{humor_3}
    \QQ_{1i} \leq \frac{R_{n,3}}{3n} + 
    \frac{1}{2} \bn{ \frac{1}{n}\sum_{\ell=1}^n \bp{X_\ell^c}^{\top} \ddfmidi X_{\ell}^c - \E\bb{\bp{Y_i^c}^{\top} \ddfmidi Y_i^c \mid Z^{n,i,\bar{X}^n}, X^n} }_{L_1}\\
    + \frac{1}{2} \bn{ \E\bb{\bp{Y_i^c}^{\top} \ddfmidi Y_i^c - \bp{\tilde{Y}_i^c}^{\top} \ddfmidi \tilde{Y}_i^c \mid Z^{n,i,\bar{X}^n}, X^n} }_{L_1}
\end{multline}
Moreover by exploiting the independence of the observations $(X_i^n)$ we remark that 
\begin{align}
    \QQ_{1i}^{(a)} :=~& \bn{ \frac{1}{n}\sum_{\ell=1}^n \bp{X_\ell^c}^{\top} \ddfmidi X_{\ell}^c - \E\bb{\bp{Y_i^c}^{\top} \ddfmidi Y_i^c \mid Z^{n,i,\bar{X}^n}, X^n} }_{L_1}\\
    \leq~& \sum_{k_1, k_2\le d_n} \bn{ \frac{1}{n} \sum_{\ell=1}^n X_{\ell, k_1}^c X_{\ell, k_2}^c - \E\bb{ Y_{i, k_1}^c Y_{i, k_2}^c\mid X^n} }_{L_2}\,  \bn{\ddfmidik }_{L_2}\\
    \le~&  \sum_{k_1, k_2\le d_n}\sqrt{\Var\bb{\frac{1}{n}\sum_{\ell\le n} X_{\ell, k_1} X_{\ell, k_2} }} \, \bn{\ddfmidik }_{L_2}\\
    & + \sum_{k_1, k_2\le d_n} \bn{ \bar{X}^n_{k_2} \bp{\bar{X}^n_{k_1}-\E\bb{\bar{X}^n_{k_1}} } + \bar{X}^n_{k_1} \bp{\bar{X}^n_{k_2}-\E\bb{\bar{X}^n_{k_2} } } }_{L_2} \, \bn{ \ddfmidik }_{L_2}
    \\
    \leq~& \frac{3}{\sqrt{n}}\sum_{k_1, k_2\le d_n} M_{k_1}^n M_{k_2}^n \bn{\ddfmidik}_{L_2}\\
    \leq~& \frac{3}{\sqrt{n}}\sum_{k_1, k_2\le d_n} D_{k_{1:2}}(f_n) \leq \frac{3}{n} R_{n,2}\label{humor}
\end{align}
Moreover, since for any two vectors $a, b$ and symmetric matrix $M$, we have that: $a^\top M a - b^\top M b = (a - b)^\top M a + b^\top M (b - a)$ and since $Y_i^c - \tilde{Y}_i^c = \E[X_1^n] - \bar{X}^n$, we have:
\begin{align}
\QQ_{1i}^{(b)}:=~& \bn{ \E\bb{\bp{Y_i^c}^{\top} \ddfmidi Y_i^c - \bp{\tilde{Y}_i^c}^{\top} \ddfmidi \tilde{Y}_i^c \mid Z^{n,i,\bar{X}^n}, X^n} }_{L_1}\\
    =~&  \bn{ \E\bb{ \bp{ \bar{X}^n-\E\bb{X_1^n} }^{\top}\ddfmidi\, Y^{c}_i + \bp{ \tilde{Y}^{c}_i }^{\top} \ddfmidi \bp{ \bar{X}^n-\E\bb{X_1^n} } \mid Z^{n, i, \bar{X}^n},X^n } }_{L_1}
\end{align}
Moreover, since $\E\bb{Y_i^c \mid Z^{n, i, \bar{X}^n},X^n}=\E[X_1^n] - \bar{X}^n$ and $\E\bb{\tilde{Y}_i^c\mid Z^{n, i, \bar{X}^n},X^n} = 0$, we have:
\begin{align}
    \QQ_{1i}^{(b)}=~&  \bn{\bp{ \bar{X}^n-\E\bb{X_1^n} }^{\top}\ddfmidi\, \bp{\E\bb{X_1^n} - \bar{X}^n} }_{L_1}\\
    \leq~&  \sum_{k_1,k_2\le d_n} \bn{\bar{X}_{k_1}^n-\E\bb{X_{1,k_1}^n}}_4\, \bn{\bar{X}_{k_2}^n-\E\bb{X_{1,k_2}^n}}_4\, \bn{ \ddfmidik }_{L_2}\\
    \leq~&  \frac{1}{n}\sum_{k_1,k_2\le d_n}M_{k_1}^n M_{k_2}^n \bn{ \ddfmidik }_{L_2} \leq \frac{1}{n^{3/2}} R_{n,2}\label{humor_2}
\end{align}
Combining \cref{humor_3}, \cref{humor} and \cref{humor_2} we obtain the result.
\end{proof}

\begin{lemma}[Bounding $\QQ_{2i}$ under \eqref{ass:2}]
For any statistic sequence $(f_n)$ that satisfies \Cref{ass:approx}, we have for all $i\in [n]$:
\begin{equation}
    \QQ_{2i} \leq \frac{R_{n,3}}{3n^{4/3}} \bp{ 4 R_{n,1} + \frac{2}{n^{1/6}} R_{n,2}+\frac{1}{n^{2/3}} R_{n,3} } + \frac{3}{n^{7/6}} \bp{ \bp{ R_{n,1}}^2+\frac{1}{n^{1/3}}\bp{R_{n,2}}^2 }
\end{equation}
\end{lemma}
\begin{proof}
For shorthand notation, let:
\begin{align}
    U(x) =~& \Delta_i(f_n, x) & V(x)=~& \dfmidi^{\top} x+\frac{1}{2} x^{\top}\ddfmidi^{\top} x
\end{align}
We will use the fact that for any two random variables $U, V$: 
\begin{equation}
\|U^2 - V^2\|_{L_1} = \|(U-V)\, (U+V)\|_{L_1}\leq \|U-V\|_{L_2}\left(\|U\|_{L_2} + \|V\|_{L_2}\right).
\end{equation}
We instantiate the latter with $U=U(X_{\ell}^n)$ and $V=V(X_{\ell}^c)$. Then by \Cref{lem:c3-stats}, we then have that:
$\|U-V\|_{L_2}\leq \frac{R_{n,3}}{6n}$. By \Cref{lem:third-order}, we have that $\|U\|_{L_2}\leq \frac{3}{n^{1/3}}\bp{ R_{n,1}+\frac{1}{2n^{1/6}} R_{n,2}+\frac{1}{2{n}^{2/3}} R_{n,3} }$. Moreover, by a sequence of triangle and Cauchy-Schwarz inequalities, we also have that: $\|V\|_{L_2} \leq \frac{R_{n,1}}{n^{1/3}} + \frac{R_{n,2}}{\sqrt{n}}$.
We can thus measure the approximation error of a second degree Taylor approximation: 
\begin{align}
    \|U(X_{\ell}^n)^2 - V(X_{\ell}^c)^2\|_{L_1} \leq~& \frac{R_{n,3}}{6n} \bp{ \frac{3}{n^{1/3}}\bp{ R_{n,1}+\frac{1}{2n^{1/6}} R_{n,2}+\frac{1}{2{n}^{2/3}} R_{n,3} } +\frac{R_{n,1}}{n^{1/3}} + \frac{R_{n,2}}{\sqrt{n}} }\\
    \leq~&  \frac{R_{n,3}}{6n^{4/3}} \bp{ 4 R_{n,1} + \frac{2}{n^{1/6}} R_{n,2}+\frac{1}{n^{2/3}} R_{n,3} } =: \epsilon_n
\end{align}
With identical steps the same bound holds for the analogous quantities $U(\tilde{Y}_i^n), V(\tilde{Y}_i^c)$.
Therefore we have
\begin{align}\label{poirot_2} 
\QQ_{2i} \leq~& 2\,\epsilon_n + \frac{1}{2}
\bn{
\frac{1}{n}\sum_{\ell \le n}
    V(X_{\ell}^c)^2 - \E\bb{ V(\tilde{Y}_{i}^c)^2 \mid Z^{n,i,\bar{X}^n}, X^n }
}_{L_1}
\end{align}

Moreover, if we denote $\tilde{X}_{\ell}^n:=X_{\ell}^n-\E\bb{X_1^n}$ and $\zeta:=\bar{X}^n - \E\bb{X_1^n}$, then we have:
\begin{equation}
\bn{ V(\tilde X^n_\ell)^2-V(X^c_\ell)^2 }_{L_1}
\le
\bn{ \big[\dfmidi^T
\zeta +\frac{1}{2}({X_\ell^{c}})^T\ddfmidi\, \zeta + \frac{1}{2}\zeta^T \ddfmidi \tilde X^n_\ell }_{L_2} \, \bp{\|V(X_\ell^c)\|_{L_2}+[\|V(\tilde{X}_\ell^n)\|_{L_2}}
\end{equation}
Observe that the first term in the product on the right-hand side is at most $\frac{1}{\sqrt{n}} \left(\frac{R_{1,n}}{n^{1/3}} + \frac{R_{2,n}}{\sqrt{n}}\right)$; by applying a series of Cauchy–Schwarz  and triangle inequalities, and invoking concentration of the vector $\zeta$, i.e. $\|\zeta_k\|_{L_2}\leq \frac{M_k^n}{\sqrt{n}}$. Moreover, each of the summands in the second term is at most $\left(\frac{R_{1,n}}{n^{1/3}} + \frac{R_{2,n}}{\sqrt{n}}\right)$; by Cauchy–Schwarz  and traingle inequality. Thus we get:
\begin{align}
\Big\|V(\tilde X^n_\ell)^2-V(X^c_\ell)^2\Big\|_{L_1}\leq  \frac{2}{n^{7/6}} \bp{ R_{1,n}+\frac{1}{n^{1/6}}R_{2,n} }^2
\end{align} 
Thus it suffices to upper bound the term:
\begin{align}
\bn{
\frac{1}{n}\sum_{\ell \le n}
    V(\tilde{X}_{\ell}^n)^2 - \E\bb{ V(\tilde{Y}_{i}^c)^2 \mid Z^{n,i,\bar{X}^n}, X^n }
}_{L_1}
\end{align}
Moreover, note that $\tilde{X}_\ell^n \overset{d}{=}\tilde Y_\ell^c$. Noting that by the form of $V$, we can expand the latter as:
\begin{align}
& \sum_{k_{1:2}\leq d_n} \bn{\dfmidik[k_1] \dfmidik[k_2]
\bp{\frac{1}{n}\sum_{\ell\leq n} \tilde{X}_{\ell, k_1}^n \tilde{X}_{\ell,k_2} - \E\bb{\tilde{X}_{\ell, k_1}^n \tilde{X}_{\ell,k_2}} } }_{L_1}\\
& + 2 \sum_{k_{1:3}\leq d_n} \bn{ \dfmidik[k_1] \ddfmidik[{k_2, k_3}]
\bp{\frac{1}{n}\sum_{\ell\leq n} \tilde{X}_{\ell, k_1}^n \tilde{X}_{\ell,k_2} \tilde{X}_{\ell,k_3} - \E\bb{\tilde{X}_{\ell, k_1}^n \tilde{X}_{\ell,k_2} \tilde{X}_{\ell,k_3}} } }_{L_1}\\
& + \sum_{k_{1:4}\leq d_n} \bn{ \ddfmidik[{k_1, k_2}] \ddfmidik[{k_3, k_4}]
\bp{\frac{1}{n}\sum_{\ell\leq n} \tilde{X}_{\ell, k_1}^n \tilde{X}_{\ell,k_2} \tilde{X}_{\ell,k_3} \tilde{X}_{\ell,k_4} - \E\bb{\tilde{X}_{\ell, k_1}^n \tilde{X}_{\ell,k_2} \tilde{X}_{\ell,k_3} \tilde{X}_{\ell,k_4} } } }_{L_1}
\end{align}
By invoking Cauchy–Schwarz  inequality and the concentration of each of the centered empirical averages, we have that the latter is bounded by:
\begin{align}
& \frac{1}{\sqrt{n}} \sum_{k_{1:2}\leq d_n} \bn{\dfmidik[k_1]}_{L_4} \bn{\dfmidik[k_2]}_{L_4}
M_{k_1}^n M_{k_2}^n
 + \frac{2}{\sqrt{n}} \sum_{k_{1:3}\leq d_n} \bn{ \dfmidik[k_1] }_{L_4} \bn{ \ddfmidik[{k_2, k_3}] }_{L_4} M_{k_1}^n M_{k_2}^n M_{k_3}^n\\
& + \frac{1}{\sqrt{n}} \sum_{k_{1:4}\leq d_n} \bn{ \ddfmidik[{k_1, k_2}] }_{L_4} \bn{ \ddfmidik[{k_3, k_4}] }_{L_4} M_{k_1}^n M_{k_2}^n M_{k_3}^n M_{k_4}^n
\end{align}
which in turn is upper bounded by:
\begin{equation}
    \frac{1}{\sqrt{n}} \left( \sum_{k_1\leq d_n} \bn{\dfmidik[k_1]}_{L_4} M_{k_1}^n + \sum_{k_{1:2}\leq d_n} \bn{\ddfmidik[{k_1, k_2}]}_{L_4} M_{k_1}^n M_{k_2}^n \right)^2 \leq \frac{1}{\sqrt{n}} \left(\frac{R_{n,1}}{n^{1/3}} + \frac{R_{n,2}}{\sqrt{n}}\right)^2
\end{equation}
We can then conclude that:
\begin{align}
    \QQ_{2i}
\leq~& 2\epsilon_n + \frac{3}{n^{7/6}} \bp{ \bp{ R_{n,1}}^2+\frac{1}{n^{1/3}}\bp{R_{n,2}}^2 }
\end{align}
\end{proof}

\noindent Therefore by combining \Cref{lem:approx}, \Cref{lem:telescoping}, \Cref{lem:test_taylor}, with the three lemmas in this section we obtain that for some sufficiently large universal constant $K$:
  \begin{equation}
     \begin{split}
         \bn{d_{\mcF}\bp{ g_n(Z^n),g_n(\tilde{Y}^n) \mid X^n }} _{L_1}
\le~& \bn{g_n(\tilde{Y}^n) -f_n(\tilde{Y}^n)}_{L_1}
+\bn{g_n(Z^n)-f_n(Z^n)}_{L_1}
\\
& +K\bp{ R_{n,3}+R_{n,2} + (R_{n,1})^2 \max\bc{ \frac{1}{n^{1/6}}, R_{n,1} } }.
\end{split}
 \end{equation}

\subsection{Proof of \Cref{jardin}}

\begin{lemma}[Bounding $\QQ_{3i}$ under \eqref{ass:2-star}]\label{lem:third-order-star}
For any statistic sequence $(f_n)$ that satisfies \eqref{ass:2-star}, we have for all $i\in [n]$:
\begin{align}
    \max\bc{ \|\Delta_i(f_n, Z_i^n)\|_{L_3}^3,\, \|\Delta_i(f_n, \tilde{Y}_i^n)\|_{L_3}^3 } \leq \frac{9}{n}\bp{ R_{n,1^*}^3+\frac{1}{2\sqrt{n}} R_{n,2^*}^3+\frac{1}{6{n}^2} R_{n,3^*}^3 }
\end{align}
Therefore:
\begin{equation}
    \QQ_{3i} \leq \frac{18}{n} \bp{ (R^*_{n,1})^3+\frac{1}{2\sqrt{n}} (R^*_{n,2})^3+\frac{1}{6{n}^2} (R^*_{n,3})^3 }
\end{equation}
\end{lemma}
\begin{proof}
By \Cref{lem:c3-stats} and the fact that for any $a,b\in \RR$: $|a+b|^3\le 3 \left(|a|^3 +|b|^3\right)$, we have that:
\begin{align}
\bn{\Delta_i(f_n, \tilde{Y}_i^n) }_{L_3}^3 \leq~& 3 \bn{ \dfmidi^{\top} \tilde{Y}_i^{c}+\frac{1}{2} \bp{\tilde{Y}_i^{c}}^{\top}\ddfmidi{{\tilde{Y}_i^{c}}} }^3_{L_3} + \frac{(R_{n,3}^*)^3}{72\, n^3}\\
\leq~& 9 \bn{ \dfmidi^{\top} \tilde{Y}_i^{c} }_{L_3}^3 + \frac{9}{8} \bn{ \bp{\tilde{Y}_i^{c}}^{\top}\ddfmidi{{\tilde{Y}_i^{c}}} }^3_{L_3} + \frac{(R_{n,3}^*)^3}{72\, n^3}
\end{align}
Moreover, by Cauchy-Schwarz inequality we have:
\begin{align}
\E\bb{ \ba{ \dfmidi^{\top} \tilde{Y}_i^{c} }^3 }
\leq~& \bn{ \sup_{k\le d_n} |\tilde Y^c_{i,k}|^3 \bp{\sum_{k\le d_n}\big|  \dfmidik[k]\big| }^3 }_{L_1}
\leq \bn{ \sup_{k\le d_n} |\tilde Y^c_{i,k}|}_{L_6}^3\bn{\sum_{k\le d_n}\big|  \dfmidik[k]\big| }^3_{L_6}
\leq \frac{\bp{R^*_{n,1}}^3}{n}
\end{align}
By similar applications of the Cauchy-Schwarz inequality we obtain that:
\begin{equation}\begin{split}
 \E\bb{ \ba{ \bp{\tilde{Y}_i^{c}}^{\top}\ddfmidi{{\tilde{Y}_i^{c}}} }^3 }
\leq~& \bp{ \bn{ \sup_{k\le d_n} \ba{ \tilde Y_{i,k}^c } }^2_{L_9}\, \bn{ \sum_{k_1, k_2\le d_n} \ba{ \ddfmidik } }_{L_9} }^3
\leq \frac{(R^*_{n,2})^3}{n^{3/2}}
\end{split}\end{equation}
Combining the above yields the bound on $\|\Delta_i(f_n, \tilde{Y}_i^n)\|_{L_3}^3$. The same bound on $\|\Delta_i(f_n, Z_i^n)\|_{L_3}^3$ can be obtained in an identical manner. The lemma then follows.

\end{proof}

\begin{lemma}[Bounding $\QQ_{1i}$ under \eqref{ass:2-star}]
For any statistic sequence $(f_n)$ that satisfies \eqref{ass:2-star}, we have that there is a constant $C$ that does not depend on $n$ such that  for all $i\in [n]$:
\begin{equation}
    \QQ_{1i} \leq \frac{R^*_{n,3}}{3n}+ \frac{C\max(\log(d_n),1)R_{2,n}^*}{n}
\end{equation}
\end{lemma}
\begin{proof}
Applying \Cref{lem:c3-stats} for $V=X_{\ell}^n$ and $V=\tilde{Y}_i^n$ and replacing the terms $\Delta_i(f_n, X_\ell^n)$ and $\Delta_i(f_n, \tilde{Y}_i^n)$ in $\QQ_{1i}$, with their corresponding second degree approximations, we have:
\begin{multline}
    \QQ_{1i} \leq \frac{R^*_{n,3}}{3n} + \bn{ \dfmidi^{\top} \bp{ \E\bb{\tilde{Y}_i^n \mid X^n} - \bar{X}^n } }_{L_1} \\
    + \frac{1}{2} \bn{ \frac{1}{n}\sum_{\ell=1}^n \bp{X_\ell^c}^{\top} \ddfmidi X_{\ell}^c - \E\bb{\bp{\tilde{Y}_i^c}^{\top} \ddfmidi \tilde{Y}_i^c \mid Z^{n,i,\bar{X}^n}, X^n} }_{L_1}
\end{multline}
\emph{Importantly, observe that by the definition of $\tilde{Y}_i^n$, we have that $\E\bb{\tilde{Y}_i^n\mid X^n} = \bar{X}^n$. Thus the first order term in this expansion vanishes}. Hence:
\begin{equation}
    \QQ_{1i} \leq \frac{R^*_{n,3}}{3n} + 
    \frac{1}{2} \bn{ \frac{1}{n}\sum_{\ell=1}^n \bp{X_\ell^c}^{\top} \ddfmidi X_{\ell}^c - \E\bb{\bp{\tilde{Y}_i^c}^{\top} \ddfmidi \tilde{Y}_i^c \mid Z^{n,i,\bar{X}^n}, X^n} }_{L_1}
\end{equation}
We can further split the second term on the right hand side as:
\begin{multline}\label{humor_3star}
    \QQ_{1i} \leq \frac{R^*_{n,3}}{3n} + 
    \frac{1}{2} \bn{ \frac{1}{n}\sum_{\ell=1}^n \bp{X_\ell^c}^{\top} \ddfmidi X_{\ell}^c - \E\bb{\bp{Y_i^c}^{\top} \ddfmidi Y_i^c \mid Z^{n,i,\bar{X}^n}, X^n} }_{L_1}\\
    + \frac{1}{2} \bn{ \E\bb{\bp{Y_i^c}^{\top} \ddfmidi Y_i^c - \bp{\tilde{Y}_i^c}^{\top} \ddfmidi \tilde{Y}_i^c \mid Z^{n,i,\bar{X}^n}, X^n} }_{L_1}
\end{multline}
Moreover by the triangular inequality we remark that 
\begin{align}
    \QQ_{1i}^{(a)} :=~& \bn{ \frac{1}{n}\sum_{\ell=1}^n \bp{X_\ell^c}^{\top} \ddfmidi X_{\ell}^c - \E\bb{\bp{Y_i^c}^{\top} \ddfmidi Y_i^c \mid Z^{n,i,\bar{X}^n}, X^n} }_{L_1}\\
    \leq~&  \bn{\sup_{k_1,k_2\le d_n}\Big| \frac{1}{n} \sum_{\ell=1}^n X_{\ell, k_1}^c X_{\ell, k_2}^c - \E\bb{ Y_{i, k_1}^c Y_{i, k_2}^c\mid X^n} \big|}_{L_2}\,  \bn{\sum_{k_1, k_2\le d_n}\big|\ddfmidik[{k_1, k_2}] \big|}_{L_2}
    \\\le~ &  \bn{\sup_{k_1,k_2\le d_n}\Big| \frac{1}{n} \sum_{\ell=1}^n X^n_{\ell, k_1}X^n_{\ell, k_2} - \E\bb{ Y^n_{i, k_1} Y^n_{i, k_2}\mid X^n} \big|}_{L_2}\,  \bn{\sum_{k_1, k_2\le d_n}\big|\ddfmidik[{k_1, k_2}] \big|}_{L_2}
    \\~+&~  2\bn{\sup_{k_1,k_2\le d_n}\big|\bar{X}^n_{k_1}\big|\Big| \frac{1}{n} \sum_{\ell=1}^n X^n_{\ell, k_2} - \E\bb{ Y^n_{i, k_2}\mid X^n} \big|}_{L_2}\,  \bn{\sum_{k_1, k_2\le d_n}\big|\ddfmidik[{k_1, k_2}] \big|}_{L_2}
   \label{eq:q_2:1}
\end{align}
Using \cref{casser} we know that there is a constant $C\in \mathbb{R}$ that does not depend on $n$ such that 
\begin{align}
     & \bn{\sup_{k_1,k_2\le d_n}\Big| \frac{1}{n} \sum_{\ell=1}^n X^n_{\ell, k_1}X^n_{\ell, k_2} - \E\bb{ Y^n_{i, k_1} Y^n_{i, k_2}\mid X^n} \big|}_{L_2}
        \le \frac{C \log(d_n)}{\sqrt{n}}\big\|\sup_{k_1}|X^n_{l,k_1}|\big\|_{L_4}^2;
            \end{align}
and such that
       \begin{align}&\bn{\sup_{k_1,k_2\le d_n}\Big| \frac{1}{n} \sum_{\ell=1}^n X^n_{\ell, k_2} - \E\bb{ Y^n_{i, k_2}\mid X^n} \big|}_{L_4}
       \le \frac{C \log(d_n)}{\sqrt{n}}\big\|\sup_{k_1}|X^n_{l,k_1}|\big\|_{L_4}.
    \end{align}
Therefore we can upper-bound \cref{eq:q_2:1} as: 
\begin{align}\label{humorstar}
\bn{ \frac{1}{n}\sum_{\ell=1}^n \bp{X_\ell^c}^{\top} \ddfmidi X_{\ell}^c - \E\bb{\bp{Y_i^c}^{\top} \ddfmidi Y_i^c \mid Z^{n,i,\bar{X}^n}, X^n} }_{L_1} \le\frac{3\max\big[1,C\log(d_n)\big]R_{2,n}^*}{n}
\end{align}
Moreover, since for any two vectors $a, b$ and symmetric matrix $M$, we have that: $a^\top M a - b^\top M b = (a - b)^\top M a + b^\top M (b - a)$ and since $Y_i^c - \tilde{Y}_i^c = \E[X_1^n] - \bar{X}^n$, we have:
\begin{align}
\QQ_{1i}^{(b)}:=~& \bn{ \E\bb{\bp{Y_i^c}^{\top} \ddfmidi Y_i^c - \bp{\tilde{Y}_i^c}^{\top} \ddfmidi \tilde{Y}_i^c \mid Z^{n,i,\bar{X}^n}, X^n} }_{L_1}\\
    =~&  \bn{ \E\bb{ \bp{ \bar{X}^n-\E\bb{X_1^n} }^{\top}\ddfmidi\, Y^{c}_i + \bp{ \tilde{Y}^{c}_i }^{\top} \ddfmidi \bp{ \bar{X}^n-\E\bb{X_1^n} } \mid Z^{n, i, \bar{X}^n},X^n } }_{L_1}
\end{align}
As we established that $\bn{\sup_{k\le d_n}\big|\bar{X}_{k}^n-\E\bb{X_{1,k}^n}\big|}_{L_4}\le C\log(d_n)\bn{\sup_{k\le d_n}\big|X_{i,k}\big|}_{L_4}$ we have:
\begin{align}
   \QQ_{1i}^{(b)}\le~& \bn{ \E\bb{ \bp{ \bar{X}^n-\E\bb{X_1^n} }^{\top}\ddfmidi\, Y^{c}_i}  }_{L_1} + \bn{ \E\bb{\bp{ \tilde{Y}^{c}_i }^{\top} \ddfmidi \bp{ \bar{X}^n-\E\bb{X_1^n} } \mid Z^{n, i, \bar{X}^n},X^n } }_{L_1}\\  
   \leq~&  \bn{\sup_{k\le d_n}\big|\bar{X}_{k}^n-\E\bb{X_{1,k}^n}\big|}_{L_3}\,\bn{\sup_{k\le d_n}\big|Y_{i,k}^c\big|+\big|\tilde Y_{i,k}^c\big|}_{L_3}\,\bn{  \sum_{k_1,k_2\le d_n}\big|\ddfmidik \big|}_{L_3}\\
  \leq~& \frac{4C\log(d_n)}{n}\bn{\sup_{k\le d_n}\big|X_{i,k}\big|}^2_{L_4} \bn{  \sum_{k_1,k_2\le d_n}\big|\ddfmidik \big|}_{L_3}\\
  \leq~& \frac{2C\log(d_n)R_{2,n}^*}{n} \label{humor_2star}
\end{align}
Combining \cref{humor_3star}, \cref{humorstar} and \cref{humor_2star} we obtain the result. 
\end{proof}

\begin{lemma}[Bounding $\QQ_{2i}$ under \eqref{ass:2-star}]
For any statistic sequence $(f_n)$ that satisfies \Cref{ass:2-star}, we have that there is a constant $C$ that does not depend on $n$ such that for all $i\in [n]$:
\begin{equation}
    \QQ_{2i} \leq\frac{R^*_{n,3}}{3n^{4/3}} \bp{ 4 R^*_{n,1} + \frac{2}{n^{1/6}} R^*_{n,2}+\frac{1}{n^{2/3}} R^*_{n,3} } + \frac{C\max(\log(d_n),1)}{n^{7/6}} \bp{ \bp{ R^*_{n,1}}^2+\frac{1}{n^{1/3}}\bp{R^*_{n,2}}^2 }
\end{equation}
\end{lemma}
\begin{proof}
For shorthand notation, let:
\begin{align}
    U(x) =~& \Delta_i(f_n, x) & V(x)=~& \dfmidi^{\top} x+\frac{1}{2} x^{\top}\ddfmidi^{\top} x
\end{align}
We will use the fact that for any two random variables $U, V$: 
\begin{equation}
\|U^2 - V^2\|_{L_1} = \|(U-V)\, (U+V)\|_{L_1}\leq \|U-V\|_{L_2}\left(\|U\|_{L_2} + \|V\|_{L_2}\right).
\end{equation}
We instantiate the latter with $U=U(X_{\ell}^n)$ and $V=V(X_{\ell}^c)$. Then by \Cref{lem:c3-stats}, we then have that:
$\|U-V\|_{L_2}\leq \frac{R^*_{n,3}}{6n}$. By \Cref{lem:third-order-star}, we have that $\|U\|_{L_2}\leq \frac{3}{n^{1/3}}\bp{ R^*_{n,1}+\frac{1}{2n^{1/6}} R^*_{n,2}+\frac{1}{2{n}^{2/3}} R^*_{n,3} }$. Moreover, by a sequence of triangle and Cauchy-Schwarz inequalities, we also have that: $\|V\|_{L_2} \leq \frac{R^*_{n,1}}{n^{1/3}} + \frac{R^*_{n,2}}{\sqrt{n}}$.
We can thus measure the approximation error of a second degree Taylor approximation: 
\begin{align}
    \|U(X_{\ell}^n)^2 - V(X_{\ell}^c)^2\|_{L_1} \leq~& \frac{R^*_{n,3}}{6n} \bp{ \frac{3}{n^{1/3}}\bp{ R^*_{n,1}+\frac{1}{2n^{1/6}} R^*_{n,2}+\frac{1}{2{n}^{2/3}} R^*_{n,3} } +\frac{R^*_{n,1}}{n^{1/3}} + \frac{R^*_{n,2}}{\sqrt{n}} }\\
    \leq~&  \frac{R^*_{n,3}}{6n^{4/3}} \bp{ 4 R^*_{n,1} + \frac{2}{n^{1/6}} R^*_{n,2}+\frac{1}{n^{2/3}} R^*_{n,3} } =: \epsilon_n
\end{align}
With identical steps the same bound holds for the analogous quantities $U(\tilde{Y}_i^n), V(\tilde{Y}_i^c)$.
Therefore we have
\begin{align}\label{poirot_2} 
\QQ_{2i} \leq~& 2\,\epsilon_n + \frac{1}{2}
\bn{
\frac{1}{n}\sum_{\ell \le n}
    V(X_{\ell}^c)^2 - \E\bb{ V(\tilde{Y}_{i}^c)^2 \mid Z^{n,i,\bar{X}^n}, X^n }
}_{L_1}
\end{align}

Moreover, if we denote $\tilde{X}_{\ell}^n:=X_{\ell}^n-\E\bb{X_1^n}$ and $\zeta:=\bar{X}^n - \E\bb{X_1^n}$, then we have:
\begin{multline}\label{eq:q_3_1}
\Big\|V(\tilde X^n_\ell)^2-V(X^c_\ell)^2\Big\|_{L_1}\\
\le
\bn{ \big[\dfmidi^T
\zeta +\frac{1}{2}({X_\ell^{c}})^T\ddfmidi\, \zeta + \frac{1}{2}\zeta^T \ddfmidi\tilde X^n_\ell }_{L_2} \, \bp{\|V(X_\ell^c)\|_{L_2}+[\|V(\tilde{X}_\ell^n)\|_{L_2}}
\end{multline}
Observe that by \Cref{casser} we know that there is a constant $C$ that does not depend on $n$ such that: $\|\sup_k \zeta_k\|_{L_6}\leq C\log(d_n) \frac{\|\sup_{k}|X^n_{1,k}|\|_{L_6}}{\sqrt{n}}$. Therefore, by applying a series of Cauchy–Schwarz  and triangle inequalities, the first term in the product on the right-hand side of \cref{eq:q_3_1} is at most $\frac{C\log(d_n)}{\sqrt{n}} \left(\frac{R^*_{1,n}}{n^{1/3}} + \frac{R^*_{2,n}}{\sqrt{n}}\right)$.%
Thus we get:
\begin{align}
\Big\|V(\tilde X^n_\ell)^2-V(X^c_\ell)^2\Big\|_{L_1}\leq  \frac{2C\log(d_n)}{n^{7/6}} \bp{ R^*_{1,n}+\frac{1}{n^{1/6}}R^*_{2,n} }^2
\end{align} 
Thus it suffices to upper bound the term:
\begin{align}
\bn{
\frac{1}{n}\sum_{\ell \le n}
    V(\tilde{X}_{\ell}^n)^2 - \E\bb{ V(\tilde{Y}_{i}^c)^2 \mid Z^{n,i,\bar{X}^n}, X^n }
}_{L_1}
\end{align}
Moreover, note that $\tilde{X}_\ell^n \overset{d}{=}\tilde Y_\ell^c$. Noting that by the form of $V$, we can expand the latter as:
\begin{align}
& \bn{\sum_{k_{1:2}\leq d_n} \Big|\dfmidik[k_1] \dfmidik[k_2]
\bp{\frac{1}{n}\sum_{\ell\leq n} \tilde{X}_{\ell, k_1}^n \tilde{X}_{\ell,k_2}^n - \E\bb{\tilde{X}_{\ell, k_1}^n \tilde{X}_{\ell,k_2}^n} } \Big|}_{L_1}\\
& + 2 \bn{\sum_{k_{1:3}\leq d_n} \Big| \dfmidik[k_1] \ddfmidik[{k_2, k_3}]
\bp{\frac{1}{n}\sum_{\ell\leq n} \tilde{X}_{\ell, k_1}^n \tilde{X}_{\ell,k_2}^n \tilde{X}_{\ell,k_3}^n - \E\bb{\tilde{X}_{\ell, k_1}^n \tilde{X}_{\ell,k_2}^n \tilde{X}_{\ell,k_3}^n} } \Big|}_{L_1}\\
& +  \bn{\sum_{k_{1:4}\leq d_n}\Big| \ddfmidik[{k_1, k_2}] \ddfmidik[{k_2, k_3}]
\bp{\frac{1}{n}\sum_{\ell\leq n} \tilde{X}_{\ell, k_1}^n \tilde{X}_{\ell,k_2}^n \tilde{X}_{\ell,k_3}^n \tilde{X}_{\ell,k_4}^n - \E\bb{\tilde{X}_{\ell, k_1}^n \tilde{X}_{\ell,k_2}^n \tilde{X}_{\ell,k_3}^n \tilde{X}_{\ell,k_4}^n } } \Big|}_{L_1}\label{eq:3:2}
\end{align}
By invoking \cref{casser} and Cauchy–Schwarz we can find $C'$ such that:
\begin{align}
& \bn{\sup_{k_{1:2}\leq d_n} 
\ba{\frac{1}{n}\sum_{\ell\leq n} \tilde{X}_{\ell, k_1}^n \tilde{X}_{\ell,k_2}^n - \E\bb{\tilde{X}_{\ell, k_1}^n \tilde{X}_{\ell,k_2}^n} } }_{L_2}\le \frac{C'\log(d_n)\bn{\sup_{k\leq d_n} \big|X_{l,k}^n\big|}_{L_4}^2}{\sqrt{n}}\\
&   \bn{\sup_{k_{1:3}\leq d_n} 
\ba{\frac{1}{n}\sum_{\ell\leq n} \tilde{X}_{\ell, k_1}^n \tilde{X}_{\ell,k_2}^n \tilde{X}_{\ell,k_3}^n - \E\bb{\tilde{X}_{\ell, k_1}^n \tilde{X}_{\ell,k_2}^n \tilde{X}_{\ell,k_3}^n} } }_{L_2}\le \frac{C'\log(d_n)\bn{\sup_{k\leq d_n} \big|X_{l,k}^n\big|}_{L_6}^3}{\sqrt{n}}\\
& +  \bn{\sup_{k_{1:4}\leq d_n}
\ba{\frac{1}{n}\sum_{\ell\leq n} \tilde{X}_{\ell, k_1}^n \tilde{X}_{\ell,k_2}^n \tilde{X}_{\ell,k_3}^n \tilde{X}_{\ell,k_4}^n - \E\bb{\tilde{X}_{\ell, k_1}^n \tilde{X}_{\ell,k_2}^n \tilde{X}_{\ell,k_3}^n \tilde{X}_{\ell,k_4}^n } } }_{L_2}\le \frac{C'\log(d_n)\bn{\sup_{k\leq d_n} \big|X_{l,k}^n\big|}_{L_8}^4}{\sqrt{n}}.
\end{align}Therefore by Cauchy–Schwarz  inequality we have that 
\begin{align}
\bn{
\frac{1}{n}\sum_{\ell \le n}
    V(\tilde{X}_{\ell}^n)^2 - \E\bb{ V(\tilde{Y}_{i}^n)^2 \mid Z^{n,i,\bar{X}^n}, X^n }
}_{L_1}
\le \frac{\max(C'\log(d_n),1)}{\sqrt{n}}\left(\frac{R^*_{n,1}}{n^{1/3}} + \frac{R^*_{n,2}}{\sqrt{n}}\right)^2.
\end{align}
We can then conclude that:
\begin{align}
    \QQ_{2i}
\leq~& 2\epsilon_n + \frac{2\max((C'+C)\log(d_n),1)}{n^{7/6}} \bp{ \bp{ R^*_{n,1}}^2+\frac{1}{n^{1/3}}\bp{R^*_{n,2}}^2 }
\end{align}
\end{proof}

\noindent Therefore by combining \Cref{lem:approx}, \Cref{lem:telescoping}, \Cref{lem:test_taylor}, with the three lemmas in this section we obtain that for some sufficiently large universal constant $K$:
  \begin{equation}
     \begin{split}
         \bn{d_{\mcF}\bp{ g_n(Z^n),g_n(\tilde{Y}^n) \mid X^n }} _{L_1}
\le~& \bn{g_n(\tilde{Y}^n) -f_n(\tilde{Y}^n)}_{L_1}
+\bn{g_n(Z^n)-f_n(Z^n)}_{L_1}
\\
& +K\bp{ \log(d_n)\bp{ (R_{n,1}^{*})^2\max\bc{ \frac{1}{n^{1/6}},R^*_{n,1} }+ R_{n,2}^{*} } + R_{n,3}^{*} }.
\end{split}
 \end{equation}

\subsection{Proof of \cref{cor1}}
\begin{proof}
Firstly, we remark that by using \Cref{lem:df-is-metric} we have:
\begin{align}&
     \bn{d_{\mathcal{F}}\big(g_n(Z^n)-\E(g_n(Z^n)\mid X^n),~g_n(\tilde Y^n)-\E(g_n(\tilde Y^n)\mid X^n)\mid X^n\big)}_{L_1}
    \\&\le           \bn{d_{\mathcal{F}}\big(g_n(Z^n)-\E(g_n(Z^n)\mid X^n),~g_n(Z^n)-\E(g_n(\tilde Y^n)\mid X^n)\mid X^n\big)}_{L_1}
    \\&+     \bn{d_{\mathcal{F}}\big(g_n(Z^n)-\E(g_n(\tilde Y^n)\mid X^n),~g_n(\tilde Y^n)-\E(g_n(\tilde Y^n)\mid X^n)\mid X^n\big)}_{L_1}
    \\&\le (A)+ (B)
\end{align}
In the goal of bounding $(A)$, denote $h:x\rightarrow x$ the identity function. We easily note that $h$ belongs to the function class $\mathcal{F}$. Indeed it is three times differentiable with all its derivatives  bounded by $1$. Therefore by \cref{thm1} we have
$$(A)\le \E\bp{\ba{\E\bp{g_n(Z^n)-g_n(\tilde Y^n)\mid X^n}}}\le \bn{d_{\mathcal{F}}\big(g_n(Z^n),g_n(\tilde Y^n)\mid X^n\big)}_{L_1}.$$

\noindent To upper-bound $(B)$ we note that  \Cref{lem:trans_df} guarantees that:
\begin{align}&
     \bn{d_{\mathcal{F}}\big(g_n(Z^n)-\E(g_n(Z^n)\mid X^n),~g_n(\tilde Y^n)-\E(g_n(Z^n)\mid X^n)\mid X^n\big)}_{L_1}
    \\&=     \bn{d_{\mathcal{F}}\big(g_n(Z^n),~g_n(\tilde Y^n)\mid X^n\big)}_{L_1}
\end{align}
This implies that $(A)+(B)\le 2\bn{d_{\mathcal{F}}\big(g_n(Z^n),~g_n(\tilde Y^n)\mid X^n\big)}_{L_1}$ which proves the desired result.

\end{proof}

\section{Proofs from \Cref{sec:preliminaries}}

\subsection{Proof of \Cref{prop1}}\label{app:prop1}
\begin{proof}
We choose a  measurable subset $A\subset \RR$ and  define the characteristic function {
\begin{equation}
f:x\rightarrow \mathbb{I}(x\in A_{3\epsilon}).
\end{equation}} 
Choose $\epsilon>0$ and set $h_{\epsilon}:\RR\rightarrow\RR$ to be the following three-times differentiable function: $$h_{\epsilon}(x):=\frac{1}{\epsilon^3}\int_{x-\epsilon}^x\int_{t-\epsilon}^{t} \int_{y-\epsilon}^y f(z)dz\, dy\, dt.$$
By simple observation we obtain that $\sup_{x\in\RR}\max_{i\le 3}\big|h^{(i)}_{\epsilon}(x)\big|\le \frac{1}{\epsilon^3}.$
Therefore we have for any two random variables $U, V$ and any event $\mcE$: 
{\begin{equation}
{\E\bb{h_{\epsilon}(V)\mid \mcE}-\E\bb{h_{\epsilon}(U)\mid \mcE}} \le \frac{d_{\mcF}(U, V\mid \mcE)}{\epsilon^3}.
\end{equation}}

Moreover, we remark that $h_{\epsilon}(x)\neq 0$ only if $x\in A_{6\epsilon}$ and that $h_{\epsilon}(x)=1$ if $x\in A$. Thus for any random variable $Z$:
\begin{equation}
    \E[h_{\epsilon}(Z) \mid \mcE] \in \bb{ \Pr(Z\in A\mid \mcE), \Pr(Z\in A_{6\epsilon}\mid \mcE) }
\end{equation}
which then implies that:
\begin{equation}
    {\E\bb{h_{\epsilon}(V)\mid \mcE}-\E\bb{h_{\epsilon}(U)\mid \mcE}}\geq P(V\in A\mid \mcE)-P(U\in A_{6\epsilon}\mid \mcE).
\end{equation}
Thus we have that for any two random variables $U, V$:
\begin{equation}\label{eqn:prob-lower-bound}
    P(U\in A_{6\epsilon}\mid \mcE) \geq P(V\in A\mid \mcE) - \frac{d_{\mcF}(U,V\mid \mcE)}{\epsilon^3}
\end{equation}

Finally we observe that since for any $h\in \mcF$, we have $|h'(u)|\leq 1$ for all $u$, we have that $h(X - \E[X\mid \mcE] ) \leq h(X) +\E[X\mid \mcE] $ and $h(Y-\E[Y])\geq h(Y) - \E[Y\mid \mcE] $. Thus:
\begin{align}
    d_{\mcF}(X-\E\bb{X\mid\mcE}, Y-\E\bb{Y\mid \mcE} \mid \mcE) =~& \sup_{h\in \mcF} \E[h(X - \E[X\mid \mcE])\mid \mcE] - \E[h(Y-\E[Y\mid \mcE])\mid \mcE]\\
    \leq~& \sup_{h\in \mcF}\E[h(X)\mid\mcE] - \E[h(Y)\mid \mcE] + \E[X\mid\mcE] - \E[Y\mid\mcE] \big|\\
    \leq~& \sup_{h\in \mcF}\E[h(X)\mid \mcE] - \E[h(Y)\mid \mcE] + \sup_{h\in \mcF} \E[h(X)\mid\mcE] - \E[h(Y)\mid\mcE]\\
    \leq~& 2\, d_{\mcF}(X, Y\mid\mcE)
\end{align}
Thus applying Equation~\eqref{eqn:prob-lower-bound} to the centered random variables and invoking the fact that $\Pr(Y-\E[Y] \in A)\geq 1-\alpha$, we get that:
\begin{equation}
P(X - \E[X\mid \mcE] \in A_{6\epsilon} \mid \mcE) \geq P(Y - \E[Y\mid \mcE] \in A\mid \mcE) - \frac{2d_{\mcF}(X,Y\mid \mcE)}{\epsilon^3} \geq 1 - \alpha - \frac{2d_{\mcF}(X,Y\mid \mcE)}{\epsilon^3}
\end{equation}
which concludes the proof of the proposition.

\end{proof}

\section{Further Proofs from \Cref{sec:main}}

\subsection{Proof of \Cref{ss_nulle}}\label{app:ss_nulle}
\begin{proof}
For simplicity, we write:
\begin{align}
    J_i:=~& \mathbb{I}(\min_{j\ne i}|X_j-X_i|>1/n) & 
    g_n(X_{1:n}):=~& \frac{1}{\sqrt{n}}\sum_{i\le n}J_i -\E(J_i)
\end{align}
Moreover we note that for all $i\ne j$ we have $\tilde Z_i-\tilde Z_j=Z_i-Z_j$ this implies that $$g_n(\tilde Z_{1:n})=g_n(Z_{1:n}).$$ It is therefore enough to study $g_n(Z_{1:n})$.  We remark that $g_n$ is stable in the perturbation of one of the observations, since changing the value of $X_1$ can change at most $3$ of the random variables $(J_i)$. To make this rigorous, we denote the distance to closest neighbour of $x\in [0,1]$, larger than $x$, as $d^+(x)={\min}_{\substack{j\ge 2,X_j\ge x}}\mid X_j-x\mid$ and to the closest neighbour, smaller than $x$, as $d^-(x)={\min}_{\substack{j\ge 2,X_j\le x}}\mid X_j-x\mid$. By convention, if there is no $j\ge 2$ such that  $X_j\ge x$ (respectively $X_j\le x$) then we take $d^+(x)$ to be $0$ (respectively $d^{-}(x)=0$). We then have
\begin{align}
        \big\|g_n(X_{1:n})-g_n(0X_{2:n})\Big\|_{L_3}
        \le~&    \frac{1}{\sqrt{n} }\bp{   \big\|J_1-\mathbb{I}(\min_{j\ge 2}|X_j|>1/n)\Big\|_{L_3}+ \big\|d^+(X_1)-d^+(0)\big\|_{L_3}+\big\|d^-(X_1)\big\|_{L_3} }
        \\
        \overset{(a)}{\le}~& \frac{3}{\sqrt{n}}
\end{align}

\noindent We  show that $d_{\mcF}\Big(g_n(Y^n),\, g_n(Z^n)|X^n\Big)$ does not go to $0$. We prove it by contradiction. Suppose that $\|d_{\mcF}\Big(g_n( Y^n),g_n(Z^n)|X^n\Big)\|_{L_1}\rightarrow 0$;,
as the random variables  $(g_n(Y^n))$ and $(g_n(Z^n))$ are uniformly integrable we have that  $\E\bb{g_n(Y^n)\mid X^n}$ converges to $\E\bb{g_n(Z^n) \mid X^n}$. Moreover, we note that by the definition of $g_n$: $\E\bb{g_n(Y^n) \mid X^n}=0$. 
We remark that   
\begin{equation}
P\bp{ \min_{i\ne 1}|X_i-X_1|\ge 1/n } = \bp{1-\frac{1}{n}}^{n-1}+O\bp{\frac{1}{n}} = e^{-1}+o(1/\sqrt{n}). 
\end{equation}
Moreover if we denote $C^n_{X_i}:= {\rm card}\bp{ j\ne i~{\rm s.t}~ |X_i-X_j|\le \frac{1}{n} }$ we have
\begin{equation}
P\bp{ \min_{i\ne 1}|Z^n_i-Z^n_1|\ge 1/n \mid X^n }=\frac{1}{n}\sum_{i\le n}e^{-1-C^n_{X_i}}+o(1/\sqrt{n}).
\end{equation}
This notably implies that:
\begin{equation}
    \begin{split}
        \E\bb{ g_n(Z^n)\mid X^n}=~& \sqrt{n}\bp{ \frac{1}{n}\sum_{i\le n}e^{-1-C^n_{X_i}}-e^{-1}}+o(1)
        \\
        =~& \sqrt{n}e^{-1}\bp{\frac{1}{n}\sum_{i\le n}e^{-C^n_{X_i}}-1 }+o(1)
    \end{split}
\end{equation}
{We show that $\E\bb{g_n(Z^n)\mid X^n}$ is asymptotically non-positive and takes asymptotically, strictly negative values with non-zero probability. This would then imply that $\E\bb{g_n(Z^n)\mid X^n}$ is not asymptotically converging to $\E\bb{g_n(Y^n)}=0$, which contradicts the fact that  $\bn{d_{\mcF}\bp{g_n(Y^n),g_n(Z^n)\mid X^n }}_{L_1}\rightarrow 0$. The first part follows since, $e^{-C_{X_i}^n} \leq 1$, and we therefore have that $\limsup_{n\to\infty} \E\bb{g_n(Z^n)\mid X^n}\le 0$. For the second part, we note that it is enough to lower bound the probability that  $\frac{1}{n}\sum_{i\le n}e^{-C^n_{X_i}}-1$ is strictly negative. We observe that $\frac{1}{n}\sum_{i\le n}e^{-C^n_{X_i}}-1$ is bounded by $1$ and that 
{$\E\bb{e^{-C^n_{X_1}}}$ is the moment generating function of a binomial distribution with $n-1$ trials and success probability at most $2/n$. This is at most $\bp{1 - 2/n + 2\, e^{-1}/n}^{n-1} \leq e^{- (2 - 2/e)\, (n-1)/n} \leq e^{-1.2} + o(1)$}. Thus for sufficiently large $n$, we have that with probability bounded away from zero: $\frac{1}{n}\sum_{i\le n}e^{-C^n_{X_i}}-1 < 0$, implying:
$$\liminf_{n\rightarrow\infty }P\bp{ \E\bb{g_n(Z^n)\mid X^n} < 0} > 0;$$
Thus $\bn{d_{\mcF}\bp{g_n(Y^n),g_n(Z^n)\mid X^n }}_{L_1}\nrightarrow 0$ and the bootstrap method is not consistent.}
\end{proof}

\subsection{Proof of \Cref{thm2}}\label{app:thm2}
\begin{proof}
Let $(Y^n_i)$ be an independent copy of $(X^n_i)$. Since, by \Cref{lem:df-is-metric}, $d_{\mcF}$ satisfies the triangular inequality:
\begin{multline}
\label{eq:nn_5}
\bn{d_{\mcF}\bp{ g_n(Z^n)-\mathbb{E}[g_n(Z^n)|X^n],\, g_n(Y^n)-\mathbb{E}[g_n(Y^n)]\mid X^n } }_{L_1}
\\
\le \bn{d_{\mcF}\bp{ g_n(Z^n)-\mathbb{E}[g_n(Z^n)|X^n],~g_n(\tilde Y^n)-\mathbb{E}[g_n(\tilde Y^n)|X^n],~ \mid X^n }}_{L_1}
 \\
  + \bn{d_{\mcF}\bp{ g_n(\tilde Y^n)-\mathbb{E}[g_n(\tilde Y^n)|X^n], ~g_n(Y^n)-\mathbb{E}[g_n(Y^n)]\mid X^n }}_{L_1}
 =: \II_1 + \II_2
 \end{multline}
The first term $\II_1$ can be upper-bounded using  \Cref{cor1}. We therefore focus on bounding the second term $\II_2$ of Equation~\cref{eq:nn_5}

\noindent Let $(B_n)$ be a an increasing sequence such that (i) $B_n\rightarrow \infty$ and (ii) $r^{n,B_n}\rightarrow 0$. We remark that under \Cref{ass:uniform} such a sequence always exists.
~For example set $B_1=1$ and $L_1=1$; then for all $n$ if $r^{n, 2\, B_n}\le 2^{-L_n}$ then set $L_{n+1}=L_n+1$ and $B_{n+1}=2\, B_n$ (by \Cref{ass:uniform} the latter will occur at some finite $n$); otherwise keep $B_{n+1}=B_n$ and $L_{n+1}=L_n$.

We note that:
\begin{align}
\II_2 \le~& \bn{g_n(Y^n)-\mathbb{E}[ g_n(Y^n)] - \Big(g_n(\tilde  Y^n)-\mathbb{E}[ g_n(\tilde Y^n)\mid X^n]\Big)}_{L_1}
\\
\overset{(a)}{\le}~& \Big\|\mathbb{I}(\sqrt{n}\big\|\bar{X}^n-\E\bb{X_1^n}\|_{2}\le B_n) \\
& \times\sup_{x\in B_{d_n}(0,B_n)}\ba{ g_n\bp{Y^n+\frac{x}{\sqrt{n}}} -\mathbb{E}\Big[g_n\bp{Y^n+\frac{x}{\sqrt{n}}} \mid X^n\Big]- g_n( Y^n)+\mathbb{E}\big[g_n\bp{Y^n}\big] }\Big\|_{L_1} 
\\
& +\Big\|\mathbb{I}(\sqrt{n}\big\|\bar{X}^n-\E\bb{X_1^n}\|_{2}\ge B_n)  \times \Big[g_n(Y^n)-\mathbb{E}[ g_n(Y^n)] - \big(g_n(\tilde  Y^n)-\mathbb{E}[ g_n(\tilde Y^n)\mid X^n]\big)\Big]\Big\|_{L_1} 
\end{align}
where (a) is a consequence of the triangle inequality. The first term is bounded by $$\bn{\sup_{x\in B_{d_n}(0,B_n)}\ba{ g_n\bp{Y^n+\frac{x}{\sqrt{n}}} -\mathbb{E}\Big[g_n\bp{Y^n+\frac{x}{\sqrt{n}}} \mid X^n\Big]- g_n( Y^n)+\mathbb{E}\big[g_n\bp{Y^n}\big] }}_{L_1} \le r^{n,B_n}.$$The second term can be bounded by the use of the Cauchy-Swartz inequality:
\begin{multline}
        \Big\|\mathbb{I}(\sqrt{n}\big\|\bar{X}^n-\E\bb{X_1^n}\|_{2}\ge B_n)  \times \Big[g_n(Y^n)-\mathbb{E}[ g_n(Y^n)] - \big(g_n(\tilde  Y^n)-\mathbb{E}[ g_n(\tilde Y^n)\mid X^n]\big)\Big]\Big\|_{L_1} 
        \\
        \le\Big\|\mathbb{I}\bc{\sqrt{n}\big\|\bar{X}^n-\E\bb{X_1^n}\|_{2}\ge B_n} \Big\|_{L_2}   \bp{ \bn{g_n(Y^n)}_{L_2}+ \bn{g_n(\tilde Y^n)}_{L_2}}
\end{multline}

Observe that:
\begin{align}
& \E\bb{\bn{\bar{X}^n-\E\bb{X_{1}^n}}_{2}^2} = \E\bb{ \sum_{k\leq d_n} \bp{\bar{X}_k^n-\E\bb{X_{1,k}^n}}^2}\\
&= \sum_{k\leq d_n} \Var\bp{\bar{X}_k^n} \leq \frac{4\sum_{k\le d_n}\|X_{1,k}^n\|_{L_2}^2}{n}
\end{align}
Hence, by Chebyshev's inequality:
\begin{align}
    P\bp{\bn{\bar{X}^n-\E\bb{X_{1}^n}}_{2} \geq B_n/\sqrt{n}} \le \frac{n \Var\bp{\bn{\bar{X}^n-\E\bb{X_{1}^n}}_{2}}}{B_n^2}\leq \frac{4\sum_{k\le d_n}\|X_{1,k}^n\|_{L_2}^2}{B_n^2}
\end{align}
Thus we have:
\begin{equation}
    \bn{\mathbb{I}\bc{ \sqrt{n}\bn{\bar{X}^n-\E\bb{X_{1}^n}}_2\ge B_n} }_{L_2} \leq \sqrt{P\bp{\bn{\bar{X}^n-\E\bb{X_{1}^n}}_2 \geq B_n/\sqrt{n}}} \leq \frac{2\sqrt{\sum_{k\le k_n}\|X_{1,k}^n\|_{L_2}^2}}{B_n} 
\end{equation}
Thus we conclude that:
\begin{multline}
\bn{d_{\mcF}\bp{ g_n(Y^n)-\mathbb{E}[ g_n(Y^n)], g_n(\tilde  Y^n)-\mathbb{E}[ g_n(\tilde Y^n)\mid X^n]\mid X^n }}_{L_1}
\\
\le \frac{2\sqrt{\sum_{k\le k_n}\|X_{1,k}^n\|_{L_2}^2}}{B_n} \bp{ \bn{g_n(Y^n)}_{L_2}+ \bn{g_n(\tilde Y^n)}_{L_2}}+ r^{n,B_n}.
\end{multline}

\end{proof}
\subsection{Proof of \Cref{ex3}}
\begin{proof}
We remark that both $g_{n,1}$ and $g_{n,2}$ are invariant under uniform perturbations. Indeed for all $x\in \RR$ we have:
\begin{equation}
    \begin{split}
        g_{n,1}(X_{1:n}+x/\sqrt{n})&=\bp{\frac{1}{\sqrt{n}}\sum_{i\le \lfloor n/2\rfloor}X_{i}-\frac{x}{\sqrt{n}}-\bp{x_{i+\lfloor n/2\rfloor}-\frac{x}{\sqrt{n}}}}^2
        =g_{n,1}(X_{1:n});
    \end{split}
\end{equation}
and \begin{equation}
    \begin{split}
        g_{n,2}(X_{1:n}+x/\sqrt{n})&=\sqrt{n}\bp{ \prod_{i=1}^n\bp{ 1+\frac{x_i-{x/\sqrt{n}}-\big(\bar{x}^n-x/\sqrt{n}\big)}{n}}-1}
        =g_{n,2}(X_{1:n}).
    \end{split}
\end{equation}
Therefore both functions $(g_{n,1})$ and $(g_{n,2})$ satisfy \eqref{eqn:uniform}. Moreover by a direct application of the chain rule we can verify that both $(g_{n,1})$ and $(g_{n,2})$ verify conditions $(H_0)$ and $(H_1)$. Hence \Cref{thm2} implies that:$$d_{\mcF}\bp{g_{n,1}( Z^n)-\E(g_{n,1}(Z^n)|X),g_{n,1}( Y_{1:n})-\E(g_{n,1}(Y_{1:n}))}\rightarrow 0;$$
 $$d_{\mcF}\bp{g_{n,2}( Z^n)-\E(g_{n,2}(Z^n)|X),g_{n,2}( Y_{1:n})-\E(g_{n,2}(Y_{1:n}))}\rightarrow 0.$$
\end{proof}

\subsection{Proof of \Cref{thm3}}\label{app:thm3}
\begin{proof}
The proof works by contradiction. Suppose that there is a measurable function $\mathcal{Q}_n:x_1,\dots,x_n\rightarrow \mathcal{P}_M(\RR)$ such that for all sequence of measures $(\nu_n)\in \prod_{n=1}^{\infty}\mathcal{P}_n$ the following holds: $$\Big\|d_{\mcF}\Big(\mathcal{Q}_n(X^n),g_n(Y^n)-\E(g_n(Y^n))\mid X^n\Big)\Big\|_{L_1}\xrightarrow{n\rightarrow \infty }0.$$
By hypothesis we know that there is a sequence of measures $(p_{\theta_n}^n)\in\times_{l=1}^{\infty} \mathcal{P}_n$, an $\epsilon>0$ and and a sequence of vectors $(z_n)$ such that 
\begin{enumerate}[(i.)] 
\item $\limsup \sup_{\tilde \theta_n\in[\theta_n,\theta_n+\frac{z_n}{\sqrt{n}}]}~~\|\mathcal{I}_n(\tilde \theta_n)^{1/2}\frac{z_n}{\sqrt{n}}\|_{v,d_n}<\infty$;
\item The following holds for $(X^n)\overset{i.i.d}{\sim}{p_{\theta_n}^n}$
\begin{equation}\label{contrac}
d_{\mcF}\bp{ g_n\bp{X^n+\frac{z_n}{\sqrt{n}}}-\E\bb{g_n\bp{X^n+\frac{z_n}{\sqrt{n}}}},g_n(X^n)-\E\bb{g_n(X^n)} }>\epsilon.
\end{equation}
\item $\theta_n+\frac{z_n}{\sqrt{n}}\in \Omega_n$
\end{enumerate}
By abuse of notations we denote $p_{\theta_n}^n+\frac{z_n}{\sqrt{n}}$ the distribution of $X+\frac{z_n}{\sqrt{n}}$ for $X\sim p_{\theta_n}^n$. We define the following subset of distributions $\mathcal{P}_n^*:=\{p_{\theta_n}^n,p_{\theta_n}^n+\frac{z_n}{\sqrt{n}}\}$ and let $X^n$ be i.i.d random variables distributed according to $\mu_n\in \mathcal{P}^*_n$. 
We want to test if $H_0: \E(X_1^n)=\theta_n$ against the alternative hypothesis $H_1: \E(X_1^n)=\theta_n+\frac{z_n}{\sqrt{n}}$. Using \cref{contrac} we know that  it is possible to find a rejection region $R_n$ such that: $$P(X^n\in R_n|H_0)+P(X^n\not \in R_n|H_1)\rightarrow 0.$$
However by hypothesis the Kullback-Leibler divergence is smooth and according to the Taylor expansion we know that 
$$KL\bp{ p_{\theta_n}^n+\frac{z_n}{\sqrt{n}},p_{\theta_n}^n }=\frac{1}{2n}z_n^{\top}\mathcal{I}_n(\tilde \theta_n)z_n,$$ 
where $\tilde\theta_n\in [\theta_n,\theta_n+\frac{z_n}{\sqrt{n}}]$. Moreover by definition of the total variation distance and the inequality: $\|p(\cdot)-q(\cdot)\|_{TV}\ge 1-\frac{1}{2}e^{-KL(p,q)}$ (see \Cref{lemm:tv}) we know that for all rejection region $R_n$ we have \begin{equation}
    \begin{split}&
        P(X^n\in R_n|H_0)+P(X^n\not \in R_n|H_1)\\&\ge \frac{1}{2}e^{-nKL(p_{\theta_n}^n+\frac{z_n}{\sqrt{n}},p_{\theta_n}^n)}\\&\overset{}{\ge} \frac{1}{2}e^{-\frac{1}{2}\sqrt{z_n^{\top}\mathcal{I}_n(\tilde \theta_n)z_n}}.
    \end{split}
\end{equation} Therefore using (i) we obtain that $\liminf_n  P(X^n\in R_n|H_0)+P(X^n\not \in R_n|H_1)>0$ and note there is a contradiction. Hence we have successfully showed that the desired result holds.
\end{proof}

\subsection{Proof of \Cref{thm_center}}
\begin{proof}
This is a direct consequence of using \cref{thm8} for the random functions $$\tilde g_n:x_{1:n}\rightarrow g_n\Big(x_1+\mathbb{E}(X^n_1)-\bar{X^n},\dots,x_n+\mathbb{E}(X^n_1)-\bar{X^n}\Big);$$ which can be approximated by the following smooth functions:$$\tilde f_n:x_{1:n}\rightarrow f_n\Big(x_1+\mathbb{E}(X^n_1)-\bar{X^n},\dots,x_n+\mathbb{E}(X^n_1)-\bar{X^n}\Big).$$ 
\end{proof}

\subsection{Proof of \Cref{thm_4}}\label{app:thm_4}
\begin{proof}

By the triangle inequality we know that:
\begin{align}
\label{yp}\ba{ g_n( Z^n)-\E\bb{g_n(Y^n)\mid X^n}}
\le~&  \ba{ \E\bb{g_n(Y^n)- g_n(\tilde{Y}^n)\mid X^n} }+  \ba{ g_n(\tilde Y^n)-\E\bb{g_n(\tilde Y^n)\mid X^n}}.
\end{align}
\noindent 
\noindent We bound successively each terms of the right-hand side. Firstly using \cref{thm1} we remark that for all $\delta>0$ we have $$P\bb{\ba{\E\bb{g_n(\tilde Y^n)\mid X^n}-\E\bb{g_n(Z^n)\mid X^n}}\ge \delta}\rightarrow 0.$$Hence by combining this with the definition of $t_{b,n}^{\beta/2}$ we know that

\begin{align}
\limsup_{\delta\downarrow 0}P\big(\ba{ g_n(Z^n)-\E\bb{g_n(\tilde Y^n)\mid X^n}}\ge t_{{\rm b},n}^{\beta/2}(X^n)+\delta \big)\le \beta/2.
\end{align} 

\noindent Let $N^n\sim N(0,\Sigma_n)$ be a Gaussian vector with variance-covariance $\Sigma_n$.
Using \cite{chernozhukov2017central} we know that there is a constant $C$ that does not depend on $n$ such that the following holds:
\begin{align}&\label{eqn:nn}
 \sup_t \Big|P\bb{\sqrt{n}\max_k\big|\bar{X^n_k}-\mathbb{E}(X_{1,k}^n)\big|\ge t}-P(\max_k|N^{n,k}|\ge t)\Big|\le \frac{C\log(d_n)^{7/6}\max(\|\max_j|X^n_{1,j}|\|_{L_4}^4,1)}{n^{1/6}}
\end{align}
Combining this with hypothesis $(H_3)$ we obtain that:
\begin{align}&
P\bp{\big|  \E\bb{g_n(Y^n+\bar{X^n}-\mathbb{E}(X^n_1))-g_n(Y^n)\mid X^n} \big|\ge t_{g,n}^{\beta/2}}
\\&\le P\bp{\|\sqrt{n}[\bar{X^n}-\mathbb{E}(X^n_1)]\|_{\infty}^{\alpha}\ge \frac{t_{g,n}^{\beta/2}}{C_n}}
\\&\le\frac{C\log(d_n)^{7/6}\max(\|\max_j|X^n_{1,j}|\|_{L_4}^4,1)}{n^{1/6}}+   P\bp{\max_k|N_{k}^n|\ge \big[\frac{t_{g,n}^{\beta/2}}{C_n}\big]^{\frac{1}{\alpha}}}
\\&\le \beta/2+o(1)
\end{align}
Therefore  for all $\delta>0$ we have
\begin{align}&
P\bp{\big|g_n(Z^n)-\E(g_n(Y^n))\big|\le t_{{\rm b},n}^{\beta/2}(X^n)+ t_{g,n}^{\beta/2}+\delta}
\\&\le P\bp{\big|g_n(Z^n)-\E(g_n(\tilde Y^n))\big|\le t_{{\rm b},n}^{\beta/2}(X^n)+\delta}
\\&~+ P\bp{\big|\E(g_n(\tilde Y^n)-g_n(Y^n)\mid X^n)\big|\le  t_{g,n}^{\beta/2}}
\\&\le \beta+o(1)
\end{align}
\end{proof}

\subsection{Proof of \Cref{ex_3_6}}
\begin{proof}
Define the function: $g_n:x_{1:n}\rightarrow \bp{\frac{1}{\sqrt{n}}\sum_{i\le n}x_i}^2$. It is straightforward to check that $(g_n)$ and $(X^n)$ satisfy conditions $(H_0)$ and $(H_1)$. Moreover for all $x$ we note that: 
\begin{align}
    \ba{\mathbb{E}\bp{g_n(Y^n+\frac{x}{\sqrt{n}})-g_n(Y^n)}}&=  x^2.
\end{align}Therefore $(g_n)$ also satisfy conditions $(H_3)$ with $C_n=1$ and $\alpha=2$. The result is a direct consequence of \Cref{thm_4}.
\end{proof}
\subsection{Proof of \Cref{cafe_2}}\label{app:cafe_2}
\begin{proof}

~Firstly we prove that with high-probability $\bar{X}^n-\mathbb{E}(X_1^n)$ is in $B_{d_n}(\gamma_n)$. Indeed, using \Cref{casser} and Chebystchev inequality we know that there is a constant $C$ such that 
\begin{align}
 P\bb{\max_k\ba{\bar{X^n_k}-\mathbb{E}(X_{1,k}^n)}\ge \gamma_n }
 &\le  \frac{\E\bp{\max_k\ba{\bar{X^n_k}-\mathbb{E}(X_{1,k}^n)}^2}}{\gamma^2_n}
 \\&\le \frac{C\log(d_n)^2\bn{\max_{k\le k_n}\ba{X^n_{1,k}}}_{L_2}^2}{\gamma_n^2n}\rightarrow 0.
\end{align}
Therefore with high-probability $\tilde\mu:=\bar{X}^n-\mathbb{E}(X_1^n)$ is in $B_{d_n}(\gamma)$. Moreover, we note that: \begin{align}&P\bp{\ba{g_n( Z^n-\tilde \mu)-\mathbb{E}(g_n( Z^n-\tilde \mu)\mid X^n)}\ge t^{\alpha}_{*}(X^n)~\mid X^n}\le \alpha.
\end{align}
We remark that $Z^n-\tilde \mu=\tilde Z^n$ therefore according to \Cref{thm_center} for all $\delta>0$ we have:
\begin{align}& {P\bp{\ba{g_n(Y^n)-\mathbb{E}(g_n(Y^n)}\ge t^{\alpha}_{*}(X^n)+\delta}\mid X^n}\le \alpha+o_n(1).
\end{align}
Which implies that asymptotically the confidence intervals $\big[t^{\beta}_{*}(X^n)-\delta, ~t^{\alpha}_{*}(X^n)+\delta\big]$ has an asymptotic coverage of at least $1-\beta$.
\end{proof}
\subsection{Proof of \Cref{cafe}}\label{app:cafe}
\begin{proof}
Denote $$\epsilon_n:=\sup_{\nu\in \mathcal{P}_n}\E_{X_1^n\in \mathcal{P}_n}\bp{\ba{ f_{n}(\tilde Z^n)-g_n(\tilde Z^n) }} + \E\bp{ \ba{f_{n}({Y}^n)-g_n({Y}^n) }}+\sup_{\nu\in \mathcal{P}_n}\max\bp{R_{n,1}^{c,\nu}, R_{n,2}^{c,\nu}, R_{n,3}^{c,\nu}}.$$
Choose $\nu\in \mathcal{Q}_n$, and let $X^n\sim \nu$ if $t^{\alpha}_{n}(X^n)$ is chosen such that:
\begin{equation}
{P\bp{ \ba{ g_n(Z^{n}+\E(X^n_1)-\bar{X^n} )-\E\bb{g_n(Z^{n}-\E(X^n_1)+\bar{X^n} )\mid X^n} }\ge t_{n}^{\alpha} (X^n)\mid X^n}} \le \alpha.
\end{equation} then according to \Cref{thm_center} and \Cref{prop1} for all $\delta>0$ we have
\begin{equation}
P\bp{ \ba{ g_n(X^{n})-\E\bb{g_n(X^n)} }\ge t^{\alpha}_{n}+\delta}\le \alpha+O(\epsilon_n).
\end{equation} But as  $\mathbb{E}(X_1^n)\in A_n$ by  definition of $t_{n}^{\alpha}(X^n)$ we know that: $t_{n}^{\alpha}(X^n)\le t_{\alpha}^{*,n}$. This directly implies that 
\begin{equation}
{{P_{X^n,Y^n\overset{i.i.d}{\sim} \nu}\bp{ \ba{ g_n(Y^{n})-\E\bb{g_n(Y^n)} }\ge t_{\alpha}^{
*,n} (X^n)+\delta}}} \le \alpha+O(\epsilon_n).
\end{equation}
As this holds for all $\nu \in A_n$ we get the desired result.
\end{proof}

\section{Proofs from \Cref{sec:illustrative}}
\subsection{Proof of \cref{ex1}}\label{pr:ex1}
\begin{proof} We will use \Cref{thm1}.
We note that the condition $(H_0)$ holds automatically as the functions  $(g_n)$ are three-times differentiable. Therefore to we only need to verify that $(H_1)$  holds. By the chain rule we have: $$\partial_i g_n( Z^{n, i, \bar{X}^n})=\frac{p}{\sqrt{n}}\big(\frac{1}{\sqrt{n}}\sum_{j\ne i}^{n}X_j+\frac{1}{\sqrt{n}}\bar{X^n}\big)^{p-1}; $$and $$ \partial_i^2 g_n( Z^{n,i,\bar{X^n}} )=\frac{p(p-1)\mathbb{I}(p\ge 2)}{n}\big(\frac{1}{\sqrt{n}}\sum_{i=1}^{n}X_i+\frac{1}{\sqrt{n}}\bar{X^n}\big)^{p-2};$$ $$\partial_i^3 g_n( Z^{n,i,x} )=\frac{p(p-1)(p-2)\mathbb{I}(p\ge 3)}{\sqrt{n}^3}\big(\frac{1}{\sqrt{n}}\sum_{i=1}^{n}X_i+\frac{1}{\sqrt{n}}x\big)^{p-3}.$$ 
Moreover, according to the  Rosenthal inequality for martingales \cite{hitczenko1990best}, there is a constant $C$ such that 
\begin{equation}
    \begin{split}
        &\bn{\frac{1}{\sqrt{n}}\sum_{i=1}^{n}X_i}_{12p}\le C\|X_1\|_{12p}
    \end{split}
\end{equation} This implies that: $R_{n,1}=O(n^{-1/2}),\quad R_{n,2}=O(n),\quad \&\quad R_{n,3}=O(n^{-3/2}).$
 \Cref{thm1} concludes the proof.
\end{proof}

\subsection{Proof of \Cref{stress}}\label{app:stress} 
\begin{proof}
We define $g_n(X):= \sqrt{n}\prod_{l=1}^n (1+\frac{X_l}{{n}}).$ Condition $(H_0)$ holds automatically as $g_n$ is smooth. To obtain the desired result we only need to verify that $(H_1)$  holds. 
By the chain rule we have: $$\partial_i g_n( Z^{n,i,\bar{X^n}})=\frac{1}{\sqrt{n}}\prod_{l\ne i}(1+\frac{ Z^n_l}{{n}})\times(1+\frac{\bar{X^n}}{n}); \quad \rm{and}\quad \partial_i^2 g_n( Z^{n,i} )=0.$$ Therefore we note that for condition $(H_1)$ to hold we only need to upper bound $\partial_i g_n(Z^{n,i,\bar{X^n}})$. Using the fact that the observations $(X_i)$ are assumed to be bounded we know that there is a constant $C$ such that 
\begin{equation}
    \begin{split}
        &\big\|\prod_{l\ne i}\Big(1+\frac{ Z^n_l}{{n}}\Big)\big\|_{L_{12}}^{12}
        \\&\overset{(a)}{\le} \mathbb{E}\Big[ \prod_{l\ne i}\mathbb{E}\Big((1+\frac{ Z^n_l}{{n}})^{12}\Big|X\Big)\Big]
        \\& \overset{(b)} {\le} \prod_{l\ne i}(1+\frac{\max(C,1)}{n})
        \sim e^{\max(C,1)}
    \end{split}
\end{equation} where to get (a) we exploited the conditional independence of the observations $(Z^n_i)$ and to obtain (b) we used the fact that the observations $(X_i)$ are assumed to be bounded. This implies that $(H_1)$ holds. 
\end{proof}
\subsection{Proof of \Cref{ex32}}\label{pr:ex32}
\begin{proof} We want to used \Cref{thm1} and \Cref{cor1}. Let $(\beta_n)\in \mathbb{R}^{\mathbb{N}}$ be a sequence satisfying: (i) $\beta_n\rightarrow \infty$ and (ii) $\beta_n=o(n^{1/4})$

\noindent We define $(f_n:\mathbb{R}^n\rightarrow\mathbb{R})$ as the following sequence of functions: $$f_n(x_{1:n}):=\frac{1}{\beta_n}\log\bp{1+e^{\beta_n\frac{1}{\sqrt{n}}\sum_{i\le n}x_i}}.$$ We note that the functions $(f_n)$ are three times differentiable functions and that the following holds:
$$\bn{f_n(Z^n)-g_n(Z^n)}\le \frac{\log(2)}{\beta_n}\quad \rm{and}\quad \bn{f_n(Z^n)-g_n(Z^n)}\le \frac{\log(2)}{\beta_n}.$$ Therefore the functions $(f_n)$ and $(g_n)$ satisfy conditions $(H_0)$. Moreover by the chain rule we have:
$$\partial_i f_n(x_{1:n})=\frac{1}{\sqrt{n}},\quad \partial^2_i f_n(x_{1:n})=\frac{\beta_n}{{n}},\quad \partial^3_i f_n(x_{1:n})=\frac{\beta_n^2}{{n}^{3/2}}. $$
This implies that condition $(H_1)$ also holds as $R_{n,1}=O(n^{-1/2}),\quad R_{n,2}=o(n^{-1/2}),\quad \&\quad R_{n,3}=0(n^{-1}).$
 \Cref{cor1} and \Cref{thm1} conclude the proof.

\end{proof}
\subsection{Proof of \Cref{roryl}}\label{app:roryl}
\begin{proof}
For ease of notations  for all element $\z\in\mathbb{Z}^2$ we denote by $(\z_1,\z_2)$ its coordinates. Moreover we remark that the i.i.d random variables $(X^n_{i,j})$ are indexed by $\dbracket{n}^2$ rather than $\dbracket{n^2}.$ As there is a one-to-one mapping between those two sets, we note that the random variables $(X^n_{i,j})$ could be indexed by $\dbracket{n^2}.$ Therefore \Cref{thm1} applies. However for ease of notations we keep the original indexing.

We note that $(g_n)$ is three times differentiable so $(H_0)$ holds. Moreover by the chain rule we have:
\begin{equation}
    \partial_{\z}g_n(ZZ^{n,\z,\overline{X^n}}):=\frac{1}{n\sqrt{n}}\frac{\sum_{m\in \{-1,1\}^n}m_{\z_1}m_{\z_2}e^{\frac{1}{\sqrt{n}}m^{\top}ZZ^{n,\z,\overline{X^n}}m}}{\sum_{m\in \{-1,1\}^n}e^{\frac{1}{\sqrt{n}}m^{\top}ZZ^{n,\z,\overline{X^n}}m}}.
\end{equation}
This implies that \begin{equation}
   \big\| \partial_{\z}g_n(ZZ^{n,\z,\overline{X^n}})\big\|_{L_{12}}\le \frac{1}{n\sqrt{n}}.
\end{equation}
Using the chain rule we have:
\begin{equation}
    \partial^2_{\z}g_n(ZZ^{n,\z,\overline{X^n}}):=\frac{1}{n^2}\frac{\sum_{m\in \{-1,1\}^n}(m_{\z_1}m_{\z_2})^2e^{\frac{1}{\sqrt{n}}m^{\top}ZZ^{n,\z,\overline{X^n}}m}}{\sum_{m\in \{-1,1\}^n}e^{\frac{1}{\sqrt{n}}m^{\top}ZZ^{n,\z,\overline{X^n}}m}}
   -\frac{1}{n^2}\Big[\frac{\sum_{m\in \{-1,1\}^n}m_{\z_1}m_{\z_2}e^{\frac{1}{\sqrt{n}}m^{\top}ZZ^{n,\z,\overline{X^n}}m}}{\sum_{m\in \{-1,1\}^n}e^{\frac{1}{\sqrt{n}}m^{\top}ZZ^{n,\z,\overline{X^n}}m}}\Big]^2 .
\end{equation}
This implies that \begin{equation}
   \big\| \partial^2_{\z}g_n(ZZ^{n,\z,\overline{X^n}})\big\|_{L_{12}}\le \frac{1}{n^2}.
\end{equation}
Finally by a last use of the chain rule we have
\begin{equation}\begin{split}
    \partial^3_{\z}g_n(ZZ^{n,\z,x}):=&\frac{1}{n^{5/2}}\frac{\sum_{m\in \{-1,1\}^n}m_{\z_1}m_{\z_2}e^{\frac{1}{\sqrt{n}}m^{\top}ZZ^{n,\z,\overline{X^n}}m}}{\sum_{m\in \{-1,1\}^n}e^{\frac{1}{\sqrt{n}}m^{\top}ZZ^{n,\z,\overline{X^n}}m}}\Big[1-
\Big[\frac{\sum_{m\in \{-1,1\}^n}m_{\z_1}m_{\z_2}e^{\frac{1}{\sqrt{n}}m^{\top}ZZ^{n,\z,\overline{X^n}}m}}{\sum_{m\in \{-1,1\}^n}e^{\frac{1}{\sqrt{n}}m^{\top}ZZ^{n,\z,\overline{X^n}}m}}\Big]^2\Big]
   \end{split}
\end{equation}
This implies that \begin{equation}
\bn{ \sup_{x\in [Z_{\z}^n,\tilde X^n]\cup  [\tilde Y_{\z}^n,\tilde X^n]}\ba{\partial^3_{\z}g_n(ZZ^{n,\z,x})}}_{L_{12}}\le \frac{1}{n^{5/2}}.
\end{equation}
This implies that the desired result holds. The same can be proved for the centered bootstrap $(\tilde Z^n_{\z})$ and $(Y_{\z}^n)$.
\end{proof}

\section{Proofs from \Cref{sec:bands}}

\subsection{Proof of \Cref{seule}}\label{app:seule}
\begin{proof}
We note that we can suppose without loss of generality that $\E\bp{X_1^n}=0$. Let $(\beta_n)$ be a sequence of  reals satisfying: $(i)$ $\beta_n\rightarrow \infty$ and (ii) $\beta_n=o\big(\frac{n^{1/8}}{\sqrt{\log(p_n)}}\big)$.

\noindent To get the desired result we use \Cref{jardin}. In this goal we define $(f_n)$ to be the following sequence of function: $$ f_n(x_{1:n}):= \frac{\log\Big(\sum_{l\le p_n}e^{\beta_n\log(p_n)\frac{1}{\sqrt{n}}\sum_{i\le n} x'_{i,l}}\Big)}{\beta_n\log(p_n)}.$$We remark that the functions $(f_n)$ are three times differentiable, and that they satisfy: \begin{equation}
    \begin{split}\label{fr}&
        \Big\|f_n(\tilde Z^n)-g_n( \tilde Z^n)\Big\|_{L_1}\le\frac{1}{\beta_n}
        \\&\Big\|f_n( Y^n)-g_n( Y^n)\Big\|_{L_1}\le\frac{1}{\beta_n}.
    \end{split}
\end{equation}

\noindent We prove that the conditions of \Cref{jardin} hold. Using \cref{fr} we know that hypothesis $(H_0)$ is holding; and we only need to prove that  $(H_1^*)$ also holds. For ease of notations we write: $$\omega_k(x_{1:n}):=\frac{e^{\beta\log(p_n)\frac{1}{\sqrt{n}}\sum_{i\le n} x
       _{i,k}}}{\sum_{l\le p_n}e^{\beta\log(p_n)\frac{1}{\sqrt{n}}\sum_{i\le n} x_{i,l}}}.$$ By the chain rule we remark that  for all $k_1\le p_n$ and all $x_{1:n}\in \RR^{p_n}$  we have 
  \begin{equation}
    \begin{split}
       \Big|\partial_{i,k_1}f_n(x_{1:n})\Big|
       &\le \frac{1}{\sqrt{n}}\omega_{k_1}(x'_{1:n})
       \end{split}\end{equation}
As $\sum_{k\le p_n}\omega_k(x_{1:n})=1$ we obtain for all $i\le n$ that 
\begin{equation}
    \begin{split}
       \Big\|\big\|\partial_{i}f_n(\tilde Z^{n, i, \bar{X}^n})\big\|_{v,p_n}\Big\|_{L_{12}}
       &\le \frac{1}{\sqrt{n}}\Big\|\sum_{j\le p_n}\omega_j({ \tilde Z^{n, i, \bar{X}^n}_{1:n}})\Big\|_{L_{12}}
       \le  \frac{1}{\sqrt{n}}.\end{split}\end{equation}
       This directly implies that $R_{n,1}^*=O(\frac{1}{n^{1/6}})$. 
       Moreover by using the chain rule we obtain that:
       \begin{equation}\begin{split}&
       \Big| \partial_{i,k_1,k_2}^2f_n(x_{1:n})\Big|
       \le \frac{\beta_n\log(p_n)}{n}\omega_{k_1}(x_{1:n})\omega_{k_2}(x_{1:n})+\mathbb{I}(k_1=k_2)\frac{\beta_n\log(p_n)}{n}\omega_{k_1}(x_{1:n}).
    \end{split}
\end{equation}

\noindent Therefore  as $\sum_{k\le p_n}\omega_{k}(x_{1:n})=1$ we have for all $i\le n$
\begin{equation}
    \begin{split}
       \Big\|\big\|\partial_{i}^2f_n(\tilde Z^{n, i, \bar{X}^n})\big\|_{m,p_n}\Big\|_{L_{12}}
       \le  \frac{2\beta_n \log(p_n)}{n}.\end{split}\end{equation}
       this implies that $R_{n,2}^*=O(\beta_n\log(p_n)/\sqrt{n})$.
       Finally, by another application of the chain rule we have:
       \begin{equation}\begin{split}&\label{landing}
       \Big| \partial_{i,k_{1:3}}^3f_n(x_{1:n})\Big|
       \\&\le2 \frac{\beta_n^2\log(p_n)^2}{n^{\frac{3}{2}}}\omega_{k_1}(x_{1:n})\omega_{k_2}(x_{1:n})\omega_{k_3}(x_{1:n})+[\mathbb{I}(k_2=k_3)+\mathbb{I}(k_1=k_3)\\&+\mathbb{I}(k_1=k_2)] \frac{\beta_n^2\log(p_n)^2}{n^{\frac{3}{2}}}\omega_{k_1}(x_{1:n})\omega_{k_2}(x_{1:n})+\mathbb{I}(k_1=k_2=k_3) \frac{\beta_n^2\log(p_n)^2}{n^{\frac{3}{2}}}\omega_{k_1}(x_{1:n}).
    \end{split}
\end{equation}

Therefore we obtain for all $i\le n$ that:
\begin{equation}\begin{split}&
       \Big\| \max_{x\in [\bar{X}^n,Z^{n}_i]\bigcup [0,\tilde Y^n_i]}\big\|\partial_{i}^3f_n(Z^{n,i,x}_{1:n})\big\|_{t,p_n}\Big\|_{L_{12}}
       \le \frac{6\beta^2_n\log(p_n)^2}{{n^{\frac{3}{2}}}}
    \end{split}
\end{equation}This implies that $R_{3,n}^*=O(\frac{\beta_n^3\log(p_n)^2}{n^{1/2}})$
Finally by assumption we know that $\|\sup_{l\le p_n}X^n_{1,l}\|_{L_{12}}<\infty$, hence the assumption $(H_1^*)$ holds as well as the desired result.

\end{proof}

\section{Proofs from \Cref{sec:minmax}}

\subsection{Proof of \Cref{prop3}}\label{app:prop3}
\begin{proof} 
To get the desired results we use \Cref{jardin}. Firstly we remark that the observations $X_1^n$ take value in $\RR^{p_n\times p_n}$. As there is a one to one mapping between $\RR^{p_n\times p_n} $ and $\RR^{p_n^2}$ we could take $X_1^n$ to take value in $\RR^{p_n^2}$. For ease of notations we do not do so, but note that the conditions of \Cref{jardin} can alternatively be expressed for observations taking value in $\RR^{p_n\times p_n}$.

\noindent We denote $g_n(x_{1:n}):=\min_{i\le p_n}\max_{j\le p_n}\frac{1}{\sqrt{n}}\sum_{i\le n}X^n_{l,i,j}$; 
and let $(\beta_n)$ be a sequence of positive reals satisfying (i) $\beta_n\rightarrow \infty$ and (ii) $\beta_n=o(\log(p_n)/n^{-1/4})$. ~We set $(f_n)$ to be the following sequence of functions: $$f_n(x_{1:n}):=\frac{1}{\beta_n log(p_n)}\log\Big(\sum_{i\le p_n}\big(\sum_{j\le p_n}e^{\frac{\beta_n\log(p_n)}{\sqrt{n}}\sum_{l\le n}x_{l,i,j}}\big)^{-1}\Big).$$ We remark that they are three times differentiable and that following holds:
\begin{equation}
    \begin{split}
        &\big\|g_n(Z^n)-f_n(Z^n)\|_{L_1}\le \frac{2}{\beta_n}
        \\&\big\|g_n(\tilde Y^n)-f_n(\tilde Y^n)\|_{L_1}\le \frac{2}{\beta_n}.
    \end{split}
\end{equation}
Therefore hypothesis $(H_0)$ holds and we only need to check that hypothesis $(H_1^*)$ also holds to be able to use \Cref{jardin}. For ease of notations we write $\hat{R}_{i,j}:=\frac{1}{\sqrt{n}}\sum_{l\le n}X_{l,i,j}$ and denote $$\omega_{i,j}(x_{1:n}):=\frac{e^{\beta_n\log(p_n)\hat{R}_{i,j}}}{\sum_{k\le p_n}e^{\beta_n\log(p_n)\hat{R}_{i,k}}};\qquad \overline{\omega}_{i}(x_{1:n}):=\frac{\big[\sum_{k\le p_n}e^{\beta_n\log(p_n)\hat{R}_{i,k}}\big]^{-1}}{\sum_{l\le p_n}\big[\sum_{k\le p_n}e^{\beta_n\log(p_n)\hat{R}_{i,k}}\big]^{-1}}.$$
We remark that the following holds: \begin{equation}\label{eq:s}\sum_{k,j\le p_n}\omega_{k,j}(x_{1:n})\overline{\omega}_k(x_{1:n})=1,\quad \sum_{j\le p_n}\omega_{k,j}(x_{1:n})=1\end{equation}By the chain rule, we remark that for all $k,j\le p_n$ we have:
$$\partial_{i,k,j}f_n(x_{1:n})=\frac{1}{\sqrt{n}}\omega_{k,j}(x_{1:n})\overline{\omega}_k(x_{1:n}).$$
~Therefore as $\sum_{k,j\le p_n}\omega_{k,j}(x_{1:n})\overline{\omega}_k(x_{1:n})=1$, this implies that:
$$\max_{l\le n} \Big\|\|\partial_{l,i,j}f_n(Z^{n, i, \bar{X}^n})\|_{v,p_n^2}\Big\|_{L_{12}}\le \frac{1}{\sqrt{n}}.
   $$ This implies that $R_{n,1}^*=O(\frac{1}{n^{1/6}}).$
Moreover by another application of the chain rule we have for all $i_1,j_1,i_2,j_2\le p_n$ and all $l\le  n$:
\begin{equation}
    \begin{split}&
        \partial_{l,i_1,j_1,i_2,j_2}^2f_n(x_{1:n})\\&=\frac{\beta_n\log(p_n)\mathbb{I}(i_1=i_2) }{\sqrt{n}} \partial_{l,i_1,j_1}f_n(x_{1:n})\Big[\mathbb{I}(j_1=j_2)-2 \omega_{i_2,j_2}(x_{1:n})\Big]
        \\&-\partial_{l,i_1,j_1}f_n(x_{1:n})\frac{\log(p_n)\beta_n }{\sqrt{n}}\omega_{i_2,j_2}(x_{1:n})\overline{\omega}_{i_2}(x_{1:n}).
    \end{split}
\end{equation}
Therefore we have by combining this with \cref{eq:s} and the fact that $\sum_{k,j\le p_n}|\partial_{l,k,j}f_n(x_{1:n})|\le 1/\sqrt{n}$  we have
$$\max_{l\le n}\Big\|\big\|  \partial_{l,i_1,j_1,i_2,j_2}^2f_n(Z^{n, i, \bar{X}^n})\big\|_{m,p_n}\Big\|\le \frac{4\beta_n\log(p_n)}{n}.
     $$
Therefore we have $R_{2,n}^*=O(\frac{\log(p_n)\beta_n}{\sqrt{n}})$. Finally, by a last application of  the chain rule  for all $i_{1:3},j_{1:3}\le p_n$ we have:
\begin{equation}
    \begin{split}&
        \partial_{l,i_1,j_1,i_2,j_2,i_3,j_3}^3f_n(x_{1:n})\\&=\frac{\beta_n\log(p_n)\mathbb{I}(i_1=i_2) }{\sqrt{n}} \partial_{l,i_1,j_1,i_3,j_3}^2f_n(x_{1:n})\Big[\mathbb{I}(j_1=j_2)-2 \omega_{i_2,j_2}(x_{1:n})\Big]
        \\&-\partial_{l,i_1,j_1,i_3,j_3}^2f_n(x_{1:n})\frac{\log(p_n)\beta_n }{\sqrt{n}}\omega_{i_2,j_2}(x_{1:n})\overline{\omega_{i_2}}(x_{1:n})
        \\&-2\frac{\log(p_n)^2\beta_n^2\mathbb{I}(i_1=i_2=i_3) }{n} \partial_{l,i_1,j_1}f_n(x_{1:n})\omega_{i_2,j_2}(x_{1:n}) \Big[\mathbb{I}(j_2=j_3)-\omega_{i_3,j_3}(x_{1:n})\Big]
        \\&-\partial_{l,i_1,j_1}   \mathbb{I}(i_2=i_3)f_n(x_{1:n})\frac{\log(p_n)^2\beta^2  }{n}\omega_{i_2,j_2}(x_{1:n})\overline{\omega_{i_2}}(x_{1:n})
     \times\Big[\mathbb{I}(j_2=j_3)-2\omega_{i_3,j_3}(x_{1:n})\Big]
     \\&+\partial_{l,i_1,j_1}f_n(x_{1:n}){\omega_{i_2,j_2}}(x_{1:n})\overline{\omega_{i_2}}(x_{1:n})
     {\omega_{i_3,j_3}}(x_{1:n})\overline{\omega_{i_3}}(x_{1:n}).
    \end{split}
\end{equation}
Therefore by combining this with \cref{eq:s}  and the fact that $\sum_{i_{1:2}\le p_n,j_{1:2}\le p_n}\ba{ \partial_{l,i_1,j_1,i_2,j_2}^2f_n(Z^{n, i, \bar{X}^n})}\le \frac{4\beta_n\log(p_n)}{n}$ we have
\begin{equation}
    \begin{split}&
   \max_{l\le n}   \bn{\max_{z\in [\bar{X}^n,\tilde Z_1^n]\bigcup [\bar{X}^n,Y_1^n] }  \bn{\partial_{l,i_1,j_1,i_2,j_2,i_3,j_3}^3f_n(Z^{i,z})}_{t,p_n}}_{L_{12}}
\le \frac{10\log(p_n)^2\beta_n^2}{n^{3/2}}.
    \end{split}
\end{equation}This implies that $R_{3,n}^*=O(\frac{\log(p_n)^2\beta_n^2}{\sqrt{n}})$.
Therefore if we choose $\beta_n=\frac{n^{1/6}}{{\log(p_n)}^{2/3}}$ then 
by using \Cref{jardin} guarantees that:
\begin{equation}
    \begin{split}&
      \Big\|d_{\mcF}(g_n(Z^n),g_n(\tilde Y^n)|X^n)\Big\|_{L_1}=O\bp{\frac{{\log(p_n)}^{2/3}}{n^{1/6}}}
    \end{split}
\end{equation}
Moreover by definition of $\tilde Y^n$ we have $\E\bp{Z_{1,i,j}^n\mid X^n}=\E\bp{\tilde Y_{1,i,j}^n\mid X^n}$ therefore $\theta_{Z^n}=\theta_{\tilde Y^n}$. By combining this with \cref{lem:trans_df} we conclude that:
\begin{equation}
    \begin{split}&
      \Big\|d_{\mcF}(g_n(Z^n)- \sqrt{n}\theta_{Z^n},~g_n(\tilde Y^n)-\sqrt{n}\theta_{\tilde Y^n}|X^n)\Big\|_{L_1}=O\bp{\frac{{\log(p_n)}^{2/3}}{n^{1/6}}}
    \end{split}
\end{equation}
Moreover by following the exact same line of reasoning we can prove that:
\begin{equation}
    \begin{split}&
      \bn{d_{\mcF}\bp{g_n(\tilde Z^n),g_n( Y^n)\mid X^n}}_{L_1}=O\bp{\frac{{\log(p_n)}^{2/3}}{n^{1/6}}}
      \\&\bn{d_{\mcF}\bp{g_n(\tilde Z^n)-\sqrt{n}\theta_{\tilde Z^n},~g_n( Y^n)-\sqrt{n}\theta_0\mid X^n}}=O\bp{\frac{{\log(p_n)}^{2/3}}{n^{1/6}}}
    \end{split}
\end{equation}
Therefore according to \Cref{prop1} if $(t_{n,\alpha})$ is a sequence chosen such that: $$P\bp{\ba{g_n(\tilde Z^n)-\sqrt{n}\theta_{\tilde Z^n}}\ge t_{n,\alpha}\mid  X^n}=\alpha$$ then for all sequence $(\epsilon_n)$ satisfying $\epsilon_n\downarrow 0$ and $\bp{\frac{\log(p_n)^2}{\sqrt{n}}}^{1/3}=o(\epsilon_n)$ we have $$\limsup_{n\rightarrow 0}P\bp{\ba{g_n(X_n)-\sqrt{n}\theta_0}\ge t_{n,\alpha}+\epsilon_n\mid  X^n}\le\alpha.$$ 
\end{proof}

\begin{lemma}
\label{seperate}

Let $(p_n)$ be a sequence of integers satisfying $\log(p_n)=o(n^{1/4})$. Define $(X^n_i)$ to be a triangular array of sequences of i.i.d random variables taking value in $\RR^{p_n\times p_n}$. Suppose that $\|\sup_{l_1,l_2\le p_n}|X_{1,l_1,l_2}^n|\|_{L_{12}}<\infty$ and that $\frac{\log(p_n)}{\sqrt{n}}=o\Big[\inf_{(i,j)\ne (l,k)}\big|\E\big[X^{n}_{1,l,k}\big]-\E\big[X^{n}_{1,i,j}\big]\big|\Big]$. We denote: $\hat \theta_n(x_{1:n})=\frac{1}{n}\min_{i\le p_n}\max_{j\le p_n}\sum_{l\le n}X_{l,i,j}^n$.
~Then the following holds  $$\Big\|d_{\mcF}\bp{\sqrt{n}\Big[\hat \theta_n( Z^n)-\theta_{Z^n}\Big],~\sqrt{n}\Big[\hat \theta_n( Y^n)-\theta_{Y^n}\Big]\mid X^n}\Big\|_{L_1}=o(1);$$where  for a process $(Y'_{1:n})$ we have written, by abuse of notations, 
$\theta_{Y'}:=\min_{i\le p_{n}}\max_{j\le p_{n}} ~\E\big[Y^{'}_{l,i,j}|X^n\big].$     

\end{lemma}
\begin{proof}
\noindent Firstly, we remark that \Cref{prop3} imply that:
$$\Big\|d_{\mcF}\bp{\sqrt{n}\Big[\hat \theta_n( Z^n)-\theta_{Z^n}\Big],~\sqrt{n}\Big[\hat \theta_n(\tilde Y^n)-\theta_{\tilde Y^n}\Big]\mid X^n}\Big\|_{L_1}=o(1).$$ Therefore to get the desired result we only need to prove that:
$$\Big\|d_{\mcF}\bp{\sqrt{n}\Big[\hat \theta_n( Y^n)-\theta_{Y^n}\Big],~\sqrt{n}\Big[\hat \theta_n(\tilde Y^n)-\theta_{\tilde Y^n}\Big]\mid X^n}\Big\|_{L_1}=o(1).$$
We denote $I,J\le p_n$ the unique pair of indices satisfying $\E\bp{X^n_{1,I,J}}=\theta_0$; and define: $$\epsilon_n:=\frac{\log(p_n)\bn{\max_{i,j}\ba{X^n_{1,i,j}}}_{L_2}}{\sqrt{n}\inf_{(i,j)\ne (l,k)}\big|\E\big[X^{n}_{1,l,k}\big]-\E\big[X^{n}_{1,i,j}\big]\big|}.$$
Using \Cref{casser} we know that there is a constant $C$ such that 
\begin{equation}
    \begin{split}&
      P\bp{\max_{i,j\le p_n} \ba{\frac{1}{{n}}\sum_{l\le n}Y^n_{l,i,j}-\E(Y^n_{l,i,j})}\ge \frac{\inf_{(i,j)\ne (l,k)}\big|\E\big[X^{n}_{1,l,k}\big]-\E\big[X^{n}_{1,i,j}\big]\big|}{4}}
     \\&=  P\bp{\max_{i,j\le p_n} \ba{\frac{1}{{n}}\sum_{l\le n}Y^n_{l,i,j}-\E(Y^n_{l,i,j})}\ge \frac{\log(p_n)\bn{\max_{i,j}\ba{X^n_{1,i,j}}}_{L_2}}{4\epsilon_n\sqrt{n}}}
       \\&\overset{}{\le} {C\epsilon_n^2}\rightarrow 0.
    \end{split}
\end{equation} 
 We remark that if $\hat{\theta}_n(Y^n)$ is distinct from  $\frac{1}{{n}}\sum_{l\le  p_n}Y^n_{l,I,J}$ this means that there are indexes $i,j\le p_n$ such that  $\ba{\frac{1}{{n}}\sum_{l\le n}Y^n_{l,i,j}-\E(Y^n_{l,i,j})}$ is bigger than $ \frac{\inf_{(i,j)\ne (l,k)}\big|\E\big[X^{n}_{1,l,k}\big]-\E\big[X^{n}_{1,i,j}\big]\big|}{2}$. Therefore we have:\begin{equation}
    \begin{split}
      & P\bp{\hat{\theta}_n(Y^n)\ne\frac{1}{{n}}\sum_{l\le  p_n}Y^n_{l,I,J}}
       \\&\le P\bp{\max_{i,j\le p_n} \ba{\frac{1}{{n}}\sum_{l\le n}Y^n_{l,i,j}-\E(Y^n_{l,i,j})}\ge \frac{\inf_{(i,j)\ne (l,k)}\big|\E\big[X^{n}_{1,l,k}\big]-\E\big[X^{n}_{1,i,j}\big]\big|}{2}}
     \\& \le {C\epsilon_n^2}\to 0.
    \end{split}
\end{equation}Similarly we have
\begin{equation}
    \begin{split}
      & P\bp{\hat\theta_n(\tilde Y^n)=\frac{1}{{n}}\sum_{l\le  p_n}\tilde Y^n_{l,I,J}}
       \\&\le P\bp{\max_{i,j\le p_n} \ba{\frac{1}{{n}}\sum_{l\le n}X^n_{l,i,j}-\E(X^n_{l,i,j})}\ge \frac{\inf_{(i,j)\ne (l,k)}\big|\E\big[X^{n}_{1,l,k}\big]-\E\big[X^{n}_{1,i,j}\big]\big|}{4}}\\&+P\bp{\max_{i,j\le p_n} \ba{\frac{1}{{n}}\sum_{l\le n}Y^n_{l,i,j}-\E(Y^n_{l,i,j})}\ge \frac{\inf_{(i,j)\ne (l,k)}\big|\E\big[X^{n}_{1,l,k}\big]-\E\big[X^{n}_{1,i,j}\big]\big|}{4}}
     \\& \le {2C\epsilon_n^2}\to 0.
    \end{split}
\end{equation}
Therefore with high-probability we have: $$\hat\theta_n(Y^n)=\hat\theta_n(\tilde Y^n)-\theta_{\tilde Y^n}+\theta_{Y^n}.$$
Using the dominated convergence theorem we thus obtain that
$$\Big\|d_{\mcF}\Big[~\sqrt{n}\big[{\hat\theta_n(\tilde Y^n)}-\theta_{\tilde Z^n}\big],\sqrt{n}\big[{g_n(Y^n)}-\theta_{Y^n}\big]\mid X^n\Big]\Big\|_{L_1}\rightarrow 0.$$

\end{proof}

\begin{lemma}
\label{seperate_2}

Let $(p_n)$ be a sequence of integers satisfying $\log(p_n)=o(n^{1/7})$. Define $(X^n_i)$ to be a triangular array of i.i.d random variables taking value in $\RR^{p_n\times p_n}$. Suppose that $\sup_n\|\sup_{l_1,l_2\le p_n}|X_{1,l_1,l_2}^n|\|_{L_{12}}<\infty$. We denote $ t_{g,n}^{\beta/4}$ a constant such that: $$P\bp{\max_{i,j\le p_n}\big|N^n_{i,j}|\ge t_{g,n}^{\beta/4}}\le \beta/4 ,$$  where $(N^n)$ is a sequence of gaussian matrices satisfying $\rm{Cov}(N^n_{i,j},N^n_{k,l})=\rm{Cov}(X^n_{i,j},X^n_{k,l})$ and write $\hat \theta_n(x_{1:n})=\frac{1}{n}\min_{i\le p_n}\max_{j\le p_n}\sum_{l\le n}X_{l,i,j}^n$. Choose $t_{\beta/4}(X^n)$ such that the following holds:
$$P\bp{\sqrt{n}\ba{\hat \theta_n( Z^n)-\mathbb{E}(\hat{\theta}_n(Z^n)\mid X^n)}\ge t_{\beta/4}(X^n)\mid X^n}\le \beta/4.$$
The following holds  $$\lim_{\delta \downarrow 0}\limsup_{n\rightarrow \infty} P\bp{\sqrt{n}\ba{\hat \theta_n( X^n)-\theta_{X^n}}\ge t_{\beta/4}(X^n)+3t_{g,n}^{\beta/4}+\delta}\le \beta;$$where  for a process $(Y'_{1:n})$ we have written, by abuse of notations, 
$\theta_{Y'}:=\min_{i\le p_{n}}\max_{j\le p_{n}} ~\E\big[Y^{'}_{l,i,j}|X^n\big].$     %

\end{lemma}
\begin{proof} We want to use \Cref{thm_4} to prove this result.  By the triangle inequality we know that:
\begin{align}
    \sqrt{n}\ba{\hat\theta_n(Y^n)-\theta_{Y^n}}
    \le  \sqrt{n}\ba{\hat\theta_n(Y^n)-\mathbb{E}(\hat\theta_n({Y^n}))}+\sqrt{n}\ba{\mathbb{E}(\hat\theta_n({Y^n}))-\theta_{Y^n}}.
\end{align}
This implies that:
\begin{align}&
    P\bp{\sqrt{n}\ba{\hat\theta_n(Y^n)-\theta_{Y^n}}\le t_{\beta/2}(X^n)+3t_{g,n}^{\beta/4}+\delta }
    \\&\le P\bp{ \sqrt{n}\ba{\hat\theta_n(Y^n)-\mathbb{E}(\hat\theta_n({Y^n}))}\le t_{\beta/4}(X^n)+2t_{g,n}^{\beta/4}+\delta/2 }\\&+P\bp{\sqrt{n}\ba{\mathbb{E}(\hat\theta_n({Y^n}))-\theta_{Y^n}}\le t_{g,n}^{\beta/4}+\delta/2}.
\end{align}
We bound the first term using \Cref{thm_4} and use arguments that are inspired by the proof of \Cref{thm_4} to bound the second.

\noindent Firstly, we remark that in \Cref{prop3} we proved that $\sqrt{n}\hat{\theta}_n$ satisfies \Cref{ass:approx}. Moreover we note that for all $x\in \mathbb{R}^{p_n\times p_n}$ we have
$$\sqrt{n}\ba{\mathbb{E}\bp{\hat\theta_n(Y^n+\frac{x}{\sqrt{n}})-\hat\theta_n(Y^n)}}\le 2\max_{i,j\le p_n}|x_{i,j}| .$$ Therefore according to \Cref{thm_4} we know that:
\begin{align}\label{sdcp_1}\lim_{\delta\downarrow 0}\limsup_{n\rightarrow \infty}P\bp{\sqrt{n}\ba{\hat \theta_n( Y^n)-\mathbb{E}(\hat{\theta}_n(Y^n))}\ge t_{\beta/4}(Y^n)+2 t_{g,n}^{\beta/4}+\delta}\le \beta/2.\end{align}
Moreover we note that: $$\sqrt{n}\ba{\mathbb{E}(\hat\theta_n(Y^n))-\theta_{X^n}}\le \sup_{i,j}\ba{\frac{1}{\sqrt{n}}\sum_{l\le n}X_{l,i,j}^n-\mathbb{E}(X_{1,i,j}^n)}.$$
Using \cite{chernozhukov2017central} we know that there is a constant $C$ that does not depend on $n$ such that the following holds:
\begin{align}&\label{sdcp_2}
 \sup_t \Big|P\bb{\sup_{i,j}\ba{\frac{1}{\sqrt{n}}\sum_{l\le n}X_{l,i,j}^n-\mathbb{E}(X_{1,i,j}^n)}\ge t}-P(\max_{i,j}|N^{n}_{i,j}|\ge t)\Big|\\&\le \frac{C\log(p_n)^{7/6}\max(\|\max_{i,j\le p_n}|X^n_{1,i,j}|\|_{L_4}^4,1)}{n^{1/6}}\to 0.
\end{align}
Therefore by combining \cref{sdcp_1} and \cref{sdcp_2} we obtain that:\begin{align}\label{sdcp_1}\lim_{\delta\downarrow 0}\limsup_{n\rightarrow \infty}P\bp{\sqrt{n}\ba{\hat \theta_n( Y^n)-\theta_{X^n}}\ge t_{\beta/4}(Y^n)+3 t_{g,n}^{\beta/4}+\delta}\le \beta.\end{align}
\end{proof}

\subsection{Proof of \Cref{prop4}}\label{app:prop4}
We separately prove each statement. Firstly we examine the case where both  $(H_0^{\rm{minmax}})$ and $(H_1^{\rm{minmax}})$ are assumed to hold.
\begin{proof} 
The first step of the proof is to approximate $\sqrt{n}\hat{\theta}_n(\cdot)$ by a discrete version and use \Cref{seperate} to prove the desired results. In this goal, choose $(\epsilon_n)$  and $(\delta_n)$ to be two  sequences satisfying: $$\epsilon_n=\frac{1}{n^{1/8}\sqrt{\log(nC_n)d'_nC_nc_n}};\qquad \delta_n=\frac{1}{(C_nn)^{3/2}}.$$

\noindent We denote $\Omega^{\mathcal{P}}_n:=\{\omega_i^*,~i\le p_n\}$ a maximal $\epsilon_n$-packing of $B_{d'_n}^1$, it is well known that $p_n\le \Big(\frac{2}{ \epsilon_n}+1\Big)^{d'_n}$. Moreover we remark that for all $\theta\in B_{d'_n}^1$ there is an $\omega^*\in\Omega^{\mathcal{P}}_n $ such that $\bn{\omega^*-\omega}\le \epsilon_n$. Indeed if this would not be the case then $\{\omega\}\bigcup \Omega^{\mathcal{P}}_n$ would also be a packing set of $B_{d'_n}^1$ contradicting the fact that we assumed it to be maximal. Similarly we can find a $\delta_n$-covering net of $B_{d'_n}^1$:  $\Omega^{\delta_n}:=\{\omega'_i,~i\le q_n\}$  with $q_n\le \bp{\frac{2}{\delta_n}+1}^{d'_n}$.

\noindent By exploiting conditions $(H_0^{\rm{minmax}})$  for all $\omega_1,\omega_2\in B_{d'_n}^1$ we have
\begin{equation}\label{gold35}\sup_{x\in \RR^{d_n}}\inf_{\omega_1^*,\omega_2^*\in \Omega^{\mathcal{P}}_n}\|f_n(x,\omega_1,\omega_2)-f_n(x,\omega^*_1,\omega^*_2)\big\|_{L_{\infty}}< C_n \epsilon_n;
\end{equation}
and  by using $(H_1^{\rm{minmax}})$ we remark that 
\begin{equation}\label{gold2}\inf_{\substack{(i,j)\ne (k,l)}}\Big|\E\Big[f_n(Y_1^n,\omega^*_k,\omega^*_l)\Big]-\E\Big[f_n(Y_1^n,\omega^*_i,\omega^*_j)\Big]\Big|>c_n\epsilon_n.
\end{equation}
We define $g_n(x_{1:n}):=\min_{i\le p_n}\max_{j\le p_n}\frac{1}{\sqrt{n}}\sum_{l\le n}f_n(x_l,\omega^*_i,\omega^*_j)$ and remark that it approximates $\sqrt{n}\hat{\theta}_n(\cdot).$ Indeed by combining \cref{gold35} and \Cref{casser} we know that there is a constant $K$ such that :\begin{align}
    \label{gold3}&\bn{g_n(\xi^n)-\sqrt{n}\hat{\theta}_n(\xi^n)}_{L_{1}}
    \\&\le \bn{\sup_{\omega_1,\omega_2\in B_{d'_n}^1}\sup_{\omega^*_1,\omega^*_2\in \Omega^{\mathcal{P}}_n}\ba{\frac{1}{\sqrt{n}}\sum_{i\le n}  f_n(\xi^n_{i},\omega_1,\omega_2)-f_n(\xi^n_{i},\omega^*_1,\omega^*_2)}}_{L_{1}}
     \\&\le  C_n\delta_n\sqrt{n}+\bn{\sup_{\omega'_1,\omega'_2\in \Omega^{\delta_n}}\sup_{\omega^*_1,\omega^*_2\in \Omega^{\mathcal{P}}_n}\ba{\frac{1}{\sqrt{n}}\sum_{i\le n}  f_n(\xi^n_{i},\omega'
     _1,\omega'_2)-f_n(\xi^n_{i},\omega^*_1,\omega^*_2)}}_{L_{1}}
     \\&\le C_n^{-1/6} n^{-1/6}+K\log(q_n)C_n\epsilon_n
        \\&\le C_n^{-1/6} n^{-1/6}+K\log(\frac{2}{(C_nn)^{3/2}}+1)d'_n C_n\epsilon_n\rightarrow 0.
\end{align} 
Therefore if we establish that the bootstrap method is consistent for $(g_n)$ we will have demonstrated that it is for $\sqrt{n}\hat\theta_n(X_{1:n}^n)$.
For ease of notations we denote $$\theta^{dis}_{Z^n}:=\min_{i\le p_n}\max_{j\le p_n}\E\big[f_n(Z^n_1, \omega^*_i,\omega^*_j)|\xi^n\big],\qquad \theta^{dis}_{\xi^n}:=\min_{i\le p_n}\max_{j\le p_n}\E\big[f_n(\xi^n_1, \omega^*_i,\omega^*_j)\big].$$
We remark that as $\log(p_n)=o(n^{1/4})$ 
and as \cref{gold2} guarantees that $$\log(p_n)/\sqrt{n}=o(\min_{(i,j)\ne (k,l)}\ba{\E\bp{f_n(\xi_1^n,\omega_i,\omega_j)}-\E\bp{f_n(\xi_1^n,\omega_l,\omega_k)}})$$ then \Cref{seperate} guarantees that $$\bn{d_{\mcF}\Big[g_n(\xi^n)-\sqrt{n}\theta^{dis}_{\xi^n},~{g_n(Z^n)-\sqrt{n}\theta^{dis}_{Z^n}}\Big|\xi^n\Big]}_{L_1}\rightarrow 0.$$
Hence by a last application of \cref{gold3}  we get the desired result.
\end{proof}

\vspace{4mm}

We now examine the case where we only assume that $(H_0^{\rm{minmax}})$ holds.
\begin{proof} 
The first step is to approximate $\sqrt{n}\hat{\theta}_n(\cdot)$ by a discrete version and use \Cref{seperate_2} to prove the desired results. In this goal, choose $(\epsilon_n)$  and $(\delta_n)$ to be two  sequences satisfying: $$\epsilon_n=\frac{1}{n^{1/8}\sqrt{\log(nC_n)d'_nC_n}};\qquad \delta_n=\frac{1}{(C_nn)^{3/2}}.$$

\noindent Similarly as previously, we denote $\Omega^{\mathcal{P}}_n:=\{\omega_i^*,~i\le p_n\}$ a maximal $\epsilon_n$-packing of $B_{d'_n}^1$ with $p_n\le \Big(\frac{2}{ \epsilon_n}+1\Big)^{d'_n}$. Similarly we can find a $\delta_n$-covering net of $B_{d'_n}^1$:  $\Omega^{\delta_n}:=\{\omega'_i,~i\le q_n\}$  with $q_n\le \bp{\frac{2}{\delta_n}+1}^{d'_n}$.

As above, we remark that for all $\theta\in B_{d'_n}^1$ there is an $\omega^*\in\Omega^{\mathcal{P}}_n $ such that $\bn{\omega^*-\omega}\le \epsilon_n$. 
\noindent By exploiting conditions $(H_0^{\rm{minmax}})$ we obtain for all $\omega_1,\omega_2\in B_{d'_n}^1$ that:
\begin{equation}\label{gold352}\sup_{x\in \RR^{d_n}}\inf_{\omega_1^*,\omega_2^*\in \Omega^{\mathcal{P}}_n}\|f_n(x,\omega_1,\omega_2)-f_n(x,\omega^*_1,\omega^*_2)\big\|_{L_{\infty}}<  C_n \epsilon_n.
\end{equation}

We define $g_n(x_{1:n}):=\min_{i\le p_n}\max_{j\le p_n}\frac{1}{\sqrt{n}}\sum_{l\le n}f_n(x_l,\omega^*_i,\omega^*_j)$ and remark that it approximates $\sqrt{n}\hat{\theta}_n(\cdot).$ Indeed by combining \cref{gold352} and \Cref{casser} we know that there is a constant $K$ such that :\begin{align}
    \label{gold32}&\bn{g_n(\xi^n)-\sqrt{n}\hat{\theta}_n(\xi^n)}_{L_{1}}
    \\&\le \bn{\sup_{\omega_1,\omega_2\in B_{d'_n}^1}\sup_{\omega^*_1,\omega^*_2\in \Omega^{\mathcal{P}}_n}\ba{\frac{1}{\sqrt{n}}\sum_{i\le n}  f_n(\xi^n_{i},\omega_1,\omega_2)-f_n(\xi^n_{i},\omega^*_1,\omega^*_2)}}_{L_{1}}
     \\&\le  C_n\delta_n\sqrt{n}+\bn{\sup_{\omega'_1,\omega'_2\in \Omega^{\delta_n}}\sup_{\omega^*_1,\omega^*_2\in \Omega^{\mathcal{P}}_n}\ba{\frac{1}{\sqrt{n}}\sum_{i\le n}  f_n(\xi^n_{i},\omega'
     _1,\omega'_2)-f_n(\xi^n_{i},\omega^*_1,\omega^*_2)}}_{L_{1}}
     \\&\le C_n^{-1/6} n^{-1/6}+K\log(q_n)C_n\epsilon_n
        \\&\le C_n^{-1/6} n^{-1/6}+K\log(\frac{2}{(C_nn)^{3/2}}+1)d'_n C_n\epsilon_n\rightarrow 0.
\end{align} 
Therefore if we establish the desired result for $(g_n)$ it will also hold for $\sqrt{n}\hat\theta_n(X_{1:n}^n)$.
 We write $X_1^n:=(f_n(\xi_1^n,\omega_i,\omega_i))_{i,j\le p_n}$ and denote $ t_{g,n}^{\beta/4}$ a constant such that: $$P\bp{\max_{i,j\le p_n}\big|N^n_{i,j}|\ge t_{g,n}^{\beta/4}}\le \beta/4 ,$$  where $N^n$ is a gaussian matrix with $\rm{Cov}\bp{N_{i,j}^n,N_{k,l}^n}=\rm{Cov}\bp{X_{i,j}^n,X_{k,l}^n}$. Choose $t_{\beta/4}(X^n)$ such that the following holds:
$$P\bp{\ba{g_n(Z^n)-\E(g_n(Z^n)\mid X^n)}\ge t_{\beta/4}(X^n)\mid X^n}\le \beta/4.$$ 
We remark that $\log(p_n)=o(n^{1/7})$ 
and that $g_n$ satisfies all of the conditions of \Cref{seperate_2}. Therefore according to \Cref{seperate_2} we know that: $$\lim_{\delta\downarrow 0}\limsup_{n\rightarrow \infty}P\bp{\ba{g_n(Z^n)-\E(g_n(Z^n)\mid X^n)}\ge t_{\beta/4}(X^n)+3t_{g,n}^{\beta/4}+\delta\mid X^n}\le \beta.$$ We note that as $\|g_n(X^n)-\sqrt{n}\theta_n(X^n)\|_{L_1}$ we therefore have that $$\lim_{\delta\downarrow 0}\limsup_{n\rightarrow \infty}P\bp{\ba{g_n(Y^n)-\E(g_n(Y^n))}\ge t_{\beta/4}(X^n)+3t_{g,n}^{\beta/4}+\delta\mid X^n}\le \beta.$$Finally we note that $t^{*,\beta/4}_{g,n}\ge t^{\beta/4}_{g,n}$ which implies the desired result.

\end{proof}

\section{Proofs from \Cref{sec:p_value}}

\subsection{Proof of \Cref{nulle}}\label{app:nulle}
\begin{proof}For simplicity we write $\Theta_n:=\{\theta_k,~k\le p_n\}$ and for all $\theta\in \Theta_n$ we denote $$g_{\theta}(x_{1:2},y_{1:2}):=K_{\theta}(x_1,x_2)+K_{\theta}(y_1,y_2)-K_{\theta}(x_1,y_2)-K_{\theta}(y_1,x_2);$$  $$D_n:=4\max\bp{\big\|\sup_{k\le p_n}K_{\theta_k}(X_{1,1}^{M},X_{1,1}^{M})\big\|_{L_{120}}, \sup_{k\le p_n}\sum_{l\le l_n}\lambda_{l,k},~1}.$$
Using Mercer's theorem we note that there are orthonormal eigenfunctions $(\psi_{l,k})$ and  positive eigenvalues $(\lambda_{l,k})$ such that: $$K_{\theta_k}(\cdot,\cdot):=\sum_{l}\lambda_{l,k}\psi_{l,k}(\cdot)\psi_{l,k}(\cdot).$$

\vspace{2mm}

\noindent As we assumed that $\max_{k}\sum_{l}\lambda_{l,k}<\infty$ then there is a sequence $(l_n)$ such that: $$\max_{k}\sum_{l\ge l_n}\lambda_{l,k}=o\bp{\min(\lambda_n,1)^2(\beta_nnp_nD_n^3)^{-1}}.$$ We show that we can suppose without loss of generality that the kernels $(K_{\theta_k})$ are of rank $l_n$ or less. In this goal, we denote $K_{\theta_k}^*(\cdot,\cdot)$ the following kernels $K_{\theta_k}^*(\cdot,\cdot):=\sum_{l\le l_n}\lambda_{l,k}\psi_{l,k}(\cdot)\psi_{l,k}(\cdot)$ and shorthand $$H_{i,j}^{*\theta_k}(X^M):=K^*_{\theta_k}(X^{M}_{j,1},X^{M}_{i,1})+K^*_{\theta_k}(X^{M}_{j,2},X^{M}_{i,2})-K^*_{\theta_k}(X^{M}_{j,1},X^{M}_{i,2})-K^*_{\theta_k}(X^{M}_{j,2},X^{M}_{i,1});$$  $$p^*_{\theta}(X^{M}):=\frac{\frac{1}{n^2}\sum_{i,j\le n}H_{i,j}^{*\theta}(X^M)}{\frac{4}{|n|^3}\sum_{i\le n}\big[\sum_{j\le n} H_{i,j}^{*,\theta}(X^M)]^2-\frac{4}{n^4}\Big[\sum_{i,j\le n} H_{i,j}^{*,\theta}(X^M)\Big]^2+\lambda_n}.$$ 

$$ \omega^*_k(x_{1:n}):=\frac{e^{\beta_np^*_{\theta_k}(X^M)}}{\sum_{k'\le p_n}e^{\beta_np^*_{\theta_k'}(X^M)}}$$
\vspace{2mm}

\noindent We prove that: $$\bn{ nT_n(X^M)-\sum_{k\le p_n} \frac{1}{n}\sum_{i,j\le n}  H_{i,j}^{*\theta_k}(X^M)\omega^*_k(X^M)}_{L_1}\to 0$$ and that $$\bn{ nT_n(Z^M)-\sum_{k\le p_n} \frac{1}{n}\sum_{i,j\le n}  H_{i,j}^{*\theta_k}(Z^M)\omega^*_k(Z^M)}_{L_1}\to 0.$$ In this goal, we remark that:
\begin{equation}
    \begin{split}\label{dm_0}
\sup_{i,j\le n} \bn{\sup_{k\le p_n}\ba{ \bp{ H_{i,j}^{*\theta_k}(X^M)-H^{\theta_k}_{i,j}}}}_{L_4}
&\le\sup_{i,j\le n}\sum_{k\le p_n} \bn{{ \bp{ H_{i,j}^{*\theta_k}(X^M)-H^{\theta_k}_{i,j}}}}_{L_4}
\\&\overset{(a)}{\le}4\sup_{i,j\le n}\bn{{ \sup_{k\le p_n}\ba{ K^{*}_{\theta_k}(X^M_{i,1},X^M_{j,2})-K_{\theta_k}(X^M_{i,1},X^M_{j,2})}}}_{L_4}
\\&\le 4\sup_{i,j\le n}\bn{{ \sup_{k\le p_n}\sum_{l\ge l_n}\lambda_{l,k}\ba{ \psi_{l,k}(X_{i,1}^M)\psi(X_{j,2}^M)}}}_{L_4}
\\&\overset{}{\le}4p_n\sup_{k\le p_n}\sum_{l\ge l_n} \lambda_{l,k} \sup_{i\le n}\sup_{k\le p_n, l\in \mathbb{N}}\bn{{ { \psi_{l,k}(X_{i,1}^M)^2}}}_{L_4}
\\&\overset{}{\le}4p_nD_n\sup_{k\le p_n}\sum_{l\ge l_n} \lambda_{l,k} 
    \end{split}
\end{equation} where to get (a) we used the fact that $X_{i,1}^M$ has the same distribution than $X_{i,2}^M$.

\vspace{2mm}
\noindent  We remark that the function $\sigma_n:(x_1,\dots,x_{p_n})\rightarrow \sum_{k\le p_n}H_{i,j}^{*\theta_k}(X^M) \frac{e^{\beta_nx_k}}{\sum_{k'\le p_n} e^{\beta_nx_{k'}}}$ is Lipchitz in the max norm: $\ba{\sigma_n(x)-\sigma_n(y)}\le \beta_n \sup_{k\le p_n}\ba{H_{i,j}^{*\theta_k}(X^M)} \max_{k\le p_n}\ba{x_k-y_k}.$ Therefore coupling this with  the triangle inequality we get that
\begin{equation}
    \begin{split}&~\label{dm_1}
     \bn{ nT_n(X^M)-\sum_{k\le p_n} \frac{1}{n}\sum_{i,j\le n}  H_{i,j}^{*\theta_k}(X^M)\omega^*_k(X^M)}_{L_1}
     \\\le& \bn{\sum_{\substack{k\le p_n\\i,j\le n}}\frac{1}{n} \bp{ H_{i,j}^{*\theta_k}(X^M)-H^{\theta_k}_{i,j}}\omega_k(X^M)}_{L_1}+ \bn{\frac{1}{n}\sum_{\substack{k\le p_n\\i,j\le n}}  H_{i,j}^{*\theta_k}(X^M)\bp{\omega_k(X^M)-\omega^*_k(X^M)}}_{L_1}
     \\\overset{(a)}{\le}&n \bn{\sup_{k\le p_n}\ba{\frac{1}{n^2}\sum_{i,j\le n} \bp{ H_{i,j}^{*\theta_k}(X^M)-H^{\theta_k}_{i,j}}}}_{L_1}+2n\beta_n\sup_{i,j\le n}\bn{\sup_{k\le p_n}\ba{H_{i,j}^{*\theta_k}(X^M)}}_{L_2}
      \bn{\sup_{k\le p_n}\ba{p_{\theta_k}^*(X^n)-p_{\theta_k}(X^n)}}_{L_2}.
    \end{split}
\end{equation}

\noindent The first term of \cref{dm_1} is bounded by:
$$n \bn{\sup_{k\le p_n}\ba{\frac{1}{n^2}\sum_{i,j\le n} \bp{ H_{i,j}^{*\theta_k}(X^M)-H^{\theta_k}_{i,j}}}}_{L_1}\le np_n\sup_{k\le p_n}\sum_{l\ge l_n}\lambda_{l,k}\rightarrow 0.$$
Moreover we can bound the second term of \cref{dm_1} using the triangular inequality, \cref{dm_2} and the fact that  $$\bp{\frac{4}{|n|^3}\sum_{i\le n}\big[\sum_{j\le n} H_{i,j}^{*,\theta}]^2-\frac{4}{n^4}\Big[\sum_{i,j\le n} H_{i,j}^{*,\theta}\Big]^2+\lambda_n}^{-1}\ge \lambda_n,\qquad \bn{\sup_{k\le p_n}\ba{H_{i,j}^{*\theta_k}(X^M)}}_{L_4}\le D_n.$$ Indeed we have,
\begin{equation}
    \begin{split}\label{dm_2}&~
   2n\beta_n\sup_{i,j\le n}\bn{\sup_{k\le p_n}\ba{H_{i,j}^{*\theta_k}(X^M)}}_{L_2}
      \bn{\sup_{k\le p_n}\ba{p_{\theta_k}^*(X^n)-p_{\theta_k}(X^n)}}_{L_2}
      \\\overset{}{\le}& \frac{2n}{\lambda_n}\beta_nD_n
      \bn{\sup_{k\le p_n}\ba{\frac{1}{n^2}\sum_{i,j\le n} \bp{ H_{i,j}^{*\theta_k}(X^M)-H^{\theta_k}_{i,j}}}}_{L_2}
     +\frac{2n\beta_n}{\lambda_n^2}D_n^2\sup_{k\le p_n}
      \Big\|\frac{4}{|n|^3}\sum_{i\le n}\Big(\big[\sum_{j\le n} H_{i,j}^{*\theta_k}(X^M)]^2\\&-\big[\sum_{j\le n} H_{i,j}^{\theta_k}]^2\Big)-\frac{4}{n^4}\bp{\Big[\sum_{i,j\le n} H_{i,j}^{*\theta_k}(X^M)\Big]^2-\Big[\sum_{i,j\le n} H_{i,j}^{\theta_k}\Big]^2}\Big\|_{L_2}
      \\\le & 4np_n\sup_{k\le p_n}\sum_{l\ge l_n}\lambda_{l,k}\big[1+\frac{2D_n\beta_n}{\lambda_n}+\frac{16D_n^3}{\lambda_n^2}\big]\rightarrow 0.
    \end{split}
\end{equation}

\noindent This implies directly that: $ \bn{nT_n(X^M)-\sum_{k\le p_n} \frac{1}{n}\sum_{i,j\le n}  H_{i,j}^{*\theta_k}(X^M)\omega^*_k(X^M)}_{L_1}\rightarrow 0$.
Following the exact same road map we can show that:$ \bn{nT_n(Z^M)-\sum_{k\le p_n} \frac{1}{n}\sum_{i,j\le n}  H_{i,j}^{*\theta_k}(Z^M)\omega^*_k(Z^M_{1:n})}_{L_1}\rightarrow 0$.

\noindent Therefore we have $$ \bn{ d_{\mcF}\bp{nT_n(Z^M),\sum_{k\le p_n} \frac{1}{n}\sum_{i,j\le n}  H_{i,j}^{*\theta_k}(Z^M)\omega^*_k(Z^M)\mid X^n}}_{L_1}\rightarrow 0;$$
$$ \bn{ d_{\mcF}\bp{nT_n(Y^M),\sum_{k\le p_n} \frac{1}{n}\sum_{i,j\le n}  H_{i,j}^{*\theta_k}(Y^M)\omega^*_k(Y^M)\mid X^n}}_{L_1}\rightarrow 0.$$ Hence using \cref{lem:df-is-metric} we know that to establish the desired result it is sufficient to study the asymptotics of $ \frac{1}{n}\sum_{i,j\le n}  H_{i,j}^{*\theta_k}(Z^M)\omega^*_k(Z^M)$. Hence we can suppose without loss of generality that all the Kernels $(K_{\theta_i})$ have ranks of $l_n$ or less, and do so

\vspace{2mm}

\noindent As the distribution of $X_{i,1}^M$ is the same than the one of $X_{i,2}^M$ they also have the same mean embedding. Moreover we note that for all $c\in\mathbb{R}$ we have :\begin{equation}
    \begin{split}
       H_{i,j}^{\theta_k}=&\sum_{l\le l_n}\lambda_{l,k}\big[\psi_{l,k}(X^M_{i,1})-\psi_{l,k}(X^M_{i,2})\big]\big[\psi_{l,k}(X_{j,1}^n)-\psi_{l,k}(X_{j,2}^n)\big] 
       \\&=\sum_{l\le l_n}\lambda_{l,k}\big[\psi_{l,k}(X^M_{i,1})+c-\big(\psi_{l,k}(X^M_{i,2})+c\big)\big]\big[\big(\psi_{l,k}(X_{j,1}^n)+c\big)-\big(\psi_{l,k}(X_{j,2}^n)+c\big)\big] .
    \end{split}
\end{equation}Therefore as the test statistics $\hat{T_n}$ depends only on $(H_{i,j}^{\theta_k})$ we can suppose without loss of generality that $$\E(\psi_{l,k}(X^{M}_{1,1}))= \E(\psi_{l,k}(X^{M}_{1,2}))=0,\quad \forall\le l_n,~ k\le p_n.$$ 
 We define $X_i^{*,n}:=(X_{i,l,k}^{*})_{l,k}$ to be random variables, defined as \begin{equation}\begin{split}X_{i,l,k}^{*}:=&\psi_{l,k}(X^{n}_{i,1})-\psi_{l,k}(X^{n}_{i,2}).\end{split}\end{equation} 
We define the process $X^{*,n}:=\big(X_i^{*,n}\big)$ and note that the observations $(X^{*,n}_i)$ take value in $M_{p_n\times l_n}(\RR)$. We remark that we could have taken the observations $(X^{*,n}_i)$ to take value in $\RR^{p_nl_n}$ as there is a one to one mapping from $M_{p_n\times l_n}(\RR)$ to $\RR^{p_nl_n}$. However for ease of notations we keep them as defined. 
~Let $( Z^{*}_i)$  and $(Y^{*}_i)$ be defined in the following way:

$$Z^*_i:=\big(\psi_{l,k}(Z^M_{i,1})- \psi_{l,k}(Z^M_{i,2})\big) \quad {\rm and}\quad Y^*_i:=\big(\psi_{l,k}(Y^M_{i,1})-\psi_{l,k}(Y^M_{i,2})\big).$$ We note that they form 
respectively a bootstrap sample and  an independent copy  of $(X^{*}_i)$. Moreover we remark that $Z_{i,1}^M|X^n\overset{d}{=}Z_{i,2}^M|X^n$ therefore we can also assume with out loss of generality that $\E(Z^*_{i,l,k}|X^n)=0$. We choose $(g_n)$ to be the following sequence of functions:
$$g_n(x_{1:n}):= \sum_{k}\omega_{k}(x_{1:n})\sum_{l}\lambda_{l,k}\Big(\frac{1}{\sqrt{n}}\sum_{j\le n}x_{j,l,k}\Big)^2$$
where we have set $\omega_k(x_{1:n})= \frac{e^{\beta_nh_{n,k}(x_{1:n})}}{\sum_{k}e^{\beta_n h_{n,k}(x_{1:n})}}$ and  $$h_{n,k}(x_{1:n}):=\frac{\sum_{l}\lambda_{l,k}\Big(\frac{1}{n}\sum_{j\le n}x_{j,l,k}\Big)^2}{\frac{4}{n^3}\sum_l\lambda_{l,k} (\sum_ix_{i,l,k})^3-\frac{4}{n^4}(\sum_l\lambda_{l,k}(\sum_ix_{i,l,k})^2)^2+\lambda_n}.$$

\vspace{2mm}

\noindent We remark that we have $$g_n(Z^{*}_{1:n})=n T_n( Z^M),\quad g_n(\tilde Y^{*}_{1:n})=n T_n( Y^M),\quad g_n(X^{*}_{1:n})=n T_n( X^M).$$ It is therefore enough to study $g_n$ and $(X^{*}_{{1:n}})$.
For ease of notations we write: \begin{equation}\label{cvd}\widehat{\sigma_{n,k}}(x_{1:n}):=\frac{4}{n^3}\sum_l\lambda_{l,k} (\sum_ix_{i,l,k})^3-\frac{4}{n^4}(\sum_l\lambda_{l,k}(\sum_ix_{i,l,j})^2)^2+\lambda_n;\end{equation}
and \begin{equation}\label{cvd2}\overline{x_{l,k_1}}:=\frac{1}{{n}}\sum_{j\le n}x_{j,l,k_1},\quad\hat{H}_{n,k}(x_{1:n}):=\sum_{l}\lambda_{l,k}\Big(\frac{1}{\sqrt{n}}\sum_{j\le n}x_{j,l,k}\Big)^2.\end{equation}
Moreover we note that we have \begin{equation}
    \begin{split}\label{eq:ker:1}
  \big\| \sup_{k\le p_n}\sum_{l\le l_n}\lambda_{l,k}[X^{*}_{1,l,k}]^2\|_{L_{120}}&\overset{(a)}{\le} 2  \big\| \sup_{k\le p_n}\sum_{l\le l_n}\lambda_{l,k}\Big([\psi_{l,k}(X^M_{i,1})]^2+[\psi_{l,k}(X^M_{i,2})]^2\Big)\|_{L_{120}}
             \\&\overset{(b)}{\le}4\big\|\sup_{k\le p_n}K_{\theta_k}(X_{1,1}^{M},X_{1,1}^{M})\big\|_{L_{120}}
    \end{split}
\end{equation}where (a) and (b) come from the Cauchy-Schwarz inequality. Similarly we have:\begin{equation}
    \begin{split}\label{eq:ker:2}
  \bn{ \sup_{k\le p_n}\sum_{l\le l_n}\lambda_{l,k}\ba{X^{*}_{1,l,k}}}_{L_{120}}
  &\overset{(a)}{\le}  \sqrt{\sup_{k\le p_n}\sum_{l\le l_n}\lambda_{l,k}}\sqrt{ \bn{ \sup_{k\le p_n}\sum_{l\le l_n}\lambda_{l,k}\ba{X^{*}_{1,l,k}}^2}_{L_{120}}}
 \\&\le 2\sqrt{\sup_{k\le p_n}\sum_{l\le l_n}\lambda_{l,k}}\sqrt{\bn{\sup_{k\le p_n}K_{\theta_k}(X_{1,1}^{M},X_{1,1}^{M})}_{L_{120}}}
    \end{split}
\end{equation}where (a) comes from the Cauchy-Schwarz inequality. Therefore using \Cref{casser} we note that there is a constant $C$ such that:

\begin{align}&
\Big\|\sup_{k\le p_n}\sup_{z\in [\bar{X}^n, Z_1^{*}]\cup[\bar{X}^n,\tilde Y_1^*]} \frac{1}{\sqrt{n}}\hat{H}_{n,k}(Z^{*,i,z})\Big\|_{L_{120}}
\\&\overset{(a)}{\le}\frac{2}{n^{3/2}}\big\| \sup_{k\le p_n}\sum_{l\le l_n}\lambda_{l,k}[X^{*}_{1,l,k}]^2\|_{L_{120}}
+2\Big\|\sup_{k\le p_n} \frac{1}{\sqrt{n}}\sum_{l\le l_n}\lambda_{l,k}(\frac{1}{\sqrt{n}}\sum_{l\le n}Z_l^{*,n,i})^2\Big\|_{L_{120}}
\\&\overset{(b)}{\le} C\log(p_n) \|\sup_{k\le p_n}K_{\theta_k}(X_{1,1}^{M},X_{1,1}^{M})\big\|_{L_{120}}.
\end{align}
where (a) is a consequence of the triangle inequality and (b) of \Cref{casser} and \cref{eq:ker:1}. Similarly using \Cref{casser} and \cref{eq:ker:2} we can establish that there is a constant $C'$ such that:
\begin{align}&\Big\|\sup_{k\le p_n}\sup_{z\in [\bar{X}^n, Z_1^{*}]\cup [\bar{X}^n,\tilde Y_1^*]} \sum_{l\le l_n} \lambda_{l,k}\frac{1}{\sqrt{n}}\Big|z_{l,k}+\sum_{i\ge 2} Z_{i,l,k}^{*}\Big|\Big\|_{L_{120}}\\&\le C'\log(p_n) \sqrt{\sup_{k\le p_n}\sum_{l\le l_n}\lambda_{l,k}}\sqrt{\bn{\sup_{k\le p_n}K_{\theta_k}(X_{1,1}^{M},X_{1,1}^{M})}_{L_{120}}}\end{align}

\vspace{2mm}

\noindent To prove the desired result we use \Cref{jardin}. We remark that the functions $(h_{n,k})$ and $(g_n)$ are three times differentiable. This implies that $(H_0)$ hold. We check that $(H_1^*)$ also holds. 
\noindent In this goal we first check that the partial derivatives of $(h_{n,k})$ are bounded. For ease of notations for a function $f_n$ we shorthand: $$\partial_{i,k,l}f_n(x_{1:n}):=\partial_{x_{i,k,l}}f_n(x_{1:n}),\quad \partial_{i,k_{1:2},l_{1:2}}^2f_n(x_{1:n}):=\partial_{x_{i,k_1,l_1},x_{i,k_2,l_2}}^2f_n(x_{1:n})$$ $$\partial_{i,k_{1:3},l_{1:3}}^2f_n(x_{1:n}):=\partial_{x_{i,k_1,l_1},x_{i,k_2,l_2}x_{i,k_3,l_3}}^3f_n(x_{1:n}).$$

\noindent In this goal, using the chain rule we note that 
for all $k\le p_n$ and all $l\le l_n$ we have:
\begin{align}
 \partial_{{i,k,l}}\hat{\sigma}_{n,k}(x_{1:n}):=
\frac{\lambda_{l,k}}{n}\bp{12\overline{x_{l,k}}^2 
           -16\overline{x_{l,k}} \frac{\hat{H}_{n,k}(x_{1:n})}{n}}
           \end{align}
           and:
\begin{equation}
    \begin{split}
        \partial_{{i,k,l}}h_{n,k}(x_{1:n})=&
        \frac{2}{{n}}\frac{\lambda_{l,k}\overline{x_{l,k}}}{\widehat{\sigma_{n,k}}(x_{1:n})}
    -\frac{\hat{H}_{n,k}(x_{1:n}) \partial_{{i,k,l}}\widehat{\sigma}_{n,k}(x_{1:n})}{\widehat{\sigma_{n,k}}(x_{1:n})^2}.
    \end{split}
\end{equation}
Therefore we obtain that there is a constant $K$ such that:
\begin{equation}
    \begin{split}
    {n}  \bn{\sup_{k\le p_n}\sum_{l\le l_n}  \ba{ \partial_{{i,k,l}}h_{n,k}(Z^{n, i, \bar{X}^n})}}_{L_{12}}
      \le 
        K\Big[\frac{D_n}{\lambda_n}+ \frac{D_n^2+D_n^3}{\lambda_n^2}\Big]
    \end{split}
\end{equation}
Using once again the chain rule we have that:
\begin{equation}
    \begin{split}
        \partial_{{i,k,l}}\omega_{k'}(x_{1:n})=& \beta_n  \partial_{{i,k,l}}h_{n,k}(x_{1:n})\omega_{k'}(x_{1:n})\bp{\mathbb{I}(k'=k)-\omega_{k}(x_{1:n})}
    \end{split}
\end{equation}
 as well as:
\begin{equation}
    \begin{split}
        \partial_{{i,k,l}}g_n(x_{1:n})=& \frac{2}{\sqrt{n}}\lambda_{l,k}(\frac{1}{\sqrt{n}}\sum_{j\le n}x_{j,l,k}) \omega_k(x_{1:n})+\sum_{k'\le p_n}\hat{H}_{n,k'}(x_{1:n}) \partial_{{i,k,l}}\omega_{k'}(x_{1:n}).
    \end{split}
\end{equation}
Therefore, using \cref{eq:ker:1} and \cref{eq:ker:2} and as $\sum_{k\le p_n}\omega_k(x_{1:n})=1$ 
we obtain that there is a constant $K_2$ that do not depend on $n$ such that 

\begin{equation}
    \begin{split}&
  \sup_{i\le n}  \Big\| \big\| \partial_{{i,k,l}}g_n( Z^{*,i,\bar{X}^n}_{1:n})\big\|_{v,p_n\times l_n} \Big\|_{L_{12}}
  \\&\le  \beta_n\sup_{i\le n}  \Big\| \sup_{k\le p_n}\big|\frac{1}{\sqrt{n}}\hat{H}_{n,k}(Z^{*,i,\bar{X^n}})\big|\Big\|_{L_{24}}\Big\| \sup_{k\le p_n}\sum_{l\le l_n}\sqrt{n}\big| \partial_{{i,k,l}}h_{n,k}(Z^{*,i,\bar{X^n}})\big| \Big\|_{L_{24}}
  \\&\quad+ \sup_{i\le n}  \Big\| \sup_{k\le p_n}\sum_{l\le l_n}\big|\lambda_{l,k}(\frac{1}{\sqrt{n}}\sum_{j\le n}Z^{*,i,\bar{X^n}}_{j,l,k})\big| \Big\|_{L_{12}}
    \\&\le \frac{K_2\beta_n\log(p_n)D_n^4}{\sqrt{n}\min(\lambda_n^2,1)} .
    \end{split}
\end{equation}
This implies that $R^*_{n,1}:=O(\frac{\beta_n\log(p_n)D_n^4}{n^{1/6}\min(\lambda_n^2,1)}).$
Moreover all $k_1,l_1,k_2,l_2$ using the chain rule we have: $\partial_{{i,k_{1:2},l_{1:2}}}^2h_{n,k_1}(x_{1:n})=0$ if $k_1$ is distinct from $k_2$. Moreover if $k_1=k_2$ we have:
\begin{align}
 \partial_{i,k_{1:2},l_{1:2}}\hat{\sigma}_{n,k}(x_{1:n}):=
\frac{\lambda_{k_1,l_1}}{n^2}\bp{24\overline{x_{l_1,k_1}}\mathbb{I}(l_1=l_2)
           -16 \frac{\hat{H}_{n,k_1}(x_{1:n})}{n}\mathbb{I}(l_1=l_2)
           -16\lambda_{k_2,l_2}\overline{x_{l_1,k_1}}\overline{x_{l_2,k_1}}}
           \end{align}
and by the Chain rule we have:

\begin{equation}
    \begin{split}
        \partial_{i,k_{1:2},l_{1:2}}^2h_{n,k_1}(x_{1:n})
        &=\frac{2\lambda_{k_1,l_1}}{n\widehat{\sigma_{n,k_1}}(x_{1:n})}\Big[\frac{\mathbb{I}(l_1=l_2)}{n}+\frac{\overline{x_{l_2,k_1}}\partial_{{i,k_2,l_2}}\widehat{\sigma_{n,k_1}}(x_{1:n})}{\widehat{\sigma_{n,k_1}}(x_{1:n})}\Big]
 \\&-\frac{\partial_{{i,k_2,l_2}}h_{n,k_2}(x_{1:n})}{n}\frac{\partial_{{i,k_1,l_1}}\widehat{\sigma_{n,k_1}}(x_{1:n})}{\widehat{\sigma_{n,k_1}}(x_{1:n})}
   -\frac{\hat{H}_{n,k_1}(x_{1:n})}{n}\frac{\partial_{i,k_{1:2},l_{1:2}}^2\widehat{\sigma_{n,k_1}}(x_{1:n})}{\widehat{\sigma_{n,k_1}}(x_{1:n})}
    \\&+ 2\frac{\hat{H}_{n,k_1}(x_{1:n})}{n}\frac{\partial_{{i,k_1,l_1}}\widehat{\sigma_{n,k_1}}(x_{1:n})\partial_{{i,k_2,l_2}}\widehat{\sigma_{n,k_1}}(x_{1:n})}{\widehat{\sigma_{n,k_1}}(x_{1:n})^3}.
  \end{split}
\end{equation}

\noindent Therefore there is a constant $K_3$ such that:
\begin{equation}
    \begin{split}&
    \Big\| \sup_{k_{1:2}\le p_n} \sum_{l_{1:2}\le l_n}\ba{ \partial_{i,k_{1:2},l_{1:2}}^2h_{n,k_1}(Z^{*, i, \bar{X}^n}_{1:n}) }\Big\|_{L_{12}}
    \le\frac{ K_3D_n^5}{\max(\lambda_n^3,1)n^2} .
    \end{split}\end{equation}
    By another application of the chain rule we have:
    \begin{equation}
    \begin{split}
        \partial_{i,k_{1:2},l_{1:2}}\omega_{k'}(x_{1:n})=& \beta_n \bp{\mathbb{I}(k'=k_1)-\omega_{k_1}(x_{1:n})}
        \\&\times \bp{\partial_{i,k_{1:2},l_{1:2}}^2h_{n,k_1}(x_{1:n})\omega_{k'}(x_{1:n})+ \partial_{{i,k_1,l_1}}h_{n,k_1}(x_{1:n})\partial_{{i,k_2,l_2}}\omega_{k'}(x_{1:n})}
        \\-&    \beta_n \partial_{{i,k_1,l_1}}h_{n,k_1}(x_{1:n})\omega_{k'}(x_{1:n})\partial_{{i,k_2,l_2}}\omega_{k_1}(x_{1:n})
    \end{split}
\end{equation}
by the chain rule we also have that:
\begin{equation}
    \begin{split}
        \partial_{i,k_{1:2},l_{1:2}}g_n(x_{1:n})=& \frac{2}{{n}}\lambda_{k_1,l_1}\mathbb{I}(k_1=k_2) \omega_{k_1}(x_{1:n})+\frac{2}{\sqrt{n}}\lambda_{k_1,l_1}(\frac{1}{\sqrt{n}}\sum_{j\le n}x_{j,l,k_1}) \partial_{{i,k_2,l_2}}\omega_{k_1}(x_{1:n})
        \\&+\frac{2}{\sqrt{n}}\lambda_{k_2,l_2}(\frac{1}{\sqrt{n}}\sum_{j\le n}x_{j,k_2,l_2})\partial_{{i,k_1,l_1}}\omega_{k_2}(x_{1:n})
        \\&+\sum_{k'\le p_n}\hat{H}_{n,k'}(x_{1:n})\partial_{i,k_{1:2},l_{1:2}}^2\omega_{k'}(x_{1:n})
    \end{split}
\end{equation}
Therefore we know that there are constants $K_3$ such that:
\begin{equation}
    \begin{split}&
    \Big\| \big\|\partial^2_{x_{1}}g_n( Z^{*,i,\bar{X^n}}_{1:n})\big\|_{m,p_n\times l_n} \Big\|_{L_{12}}
 \le \frac{K_3D_n^7}{\min(\lambda_n^4,1)n}\beta_n^2.
    \end{split}\end{equation}
This implies that $R_{n,2}^*=O(\frac{\beta_n^2D_n^7}{\min(\lambda_n^4,1)\sqrt{n}})$.    Finally using a similar line of reasoning we note that Moreover all $k_1,l_1,k_2,l_2,l_3,k_3$ using the chain rule we have: $\partial_{{i,k_{1:3},l_{1:3}}}^3h_{n,k_1}(x_{1:n})=0$ if $k_1$ is distinct from $k_2$ or from $k_3$. Moreover if $k_1=k_2=k_3$ we have:
\begin{align}
 \partial_{i,l_{1:3},k_{1:3}}^3\hat{\sigma}_{n,k}(x_{1:n})=
\frac{\lambda_{k_1,l_1}}{n^3}\Big(&24\mathbb{I}(l_1=l_2=l_3)-32\lambda_{k_1,l_3}\overline{x_{l_3,k_1}}\mathbb{I}(l_1=l_2)
           \\&-16\lambda_{k_2,l_2}(\overline{x_{l_1,k_1}}\mathbb{I}(l_3=l_2)+\overline{x_{l_2,k_1}}\mathbb{I}(l_1=l_3) )\Big)
           \end{align}
and by the Chain rule we have:

\begin{equation}
    \begin{split}&
        \partial_{i,k_{1:3},l_{1:3}}^3h_{n,k_1}(x_{1:n})
        \\=&-\frac{\partial_{l_3,k_3}\widehat{\sigma_{n,k_1}}(x_{1:n})}{n\widehat{\sigma_{n,k_1}}(x_{1:n})^2}\Big(\frac{\overline{x_{l_1,k_1}}\partial_{{i,k_2,l_2}}\widehat{\sigma_{n,k_1}}(x_{1:n})}{\widehat{\sigma_{n,k_1}}(x_{1:n})}\big[2\lambda_{k_1,l_1}+3\frac{\hat{H}_{n,k_1}(x_{1:n})\partial_{{i,k_1,l_1}}\widehat{\sigma_{n,k_1}}(x_{1:n})}{\widehat{\sigma_{n,k_1}}(x_{1:n})}\Big]\\-&\partial_{{i,k_2,l_2}}\hat{H}_{n,k_2}(x_{1:n}) \partial_{{i,k_1,l_1}}\widehat{\sigma_{n,k_1}(x_{1:n})}
   -h_{n,k_1}(x_{1:n})\partial_{i,k_{1:2},l_{1:2}}^2\widehat{\sigma_{n,k_1}}(x_{1:n})+2\lambda_{k_1,l_1}\frac{\mathbb{I}(l_1=l_2)}{n}\Big)
   \\+&\frac{2\lambda_{k_1,l_1}}{n\widehat{\sigma_{n,k_1}}(x_{1:n})^2}\bp{\frac{\partial_{{i,k_2,l_2}}\widehat{\sigma_{n,k_1}}(x_{1:n})\mathbb{I}(l_2=l_3)}{n}+\overline{x_{l_3,k_1}}\partial^2_{{i,k_{2:3},l_{2:3}}}\widehat{\sigma_{n,k_1}}(x_{1:n})}\\-&\frac{\partial_{{i,k_{2:3},l_{2:3}}}h_{n,k_2}(x_{1:n})}{n}\frac{\partial_{{i,k_1,l_1}}\widehat{\sigma_{n,k_1}}(x_{1:n})}{\widehat{\sigma_{n,k_1}}(x_{1:n})}
-\frac{\partial^2_{{i,k_1,l_1,k_3,l_3}}\widehat{\sigma_{n,k_1}}(x_{1:n})}{n\widehat{\sigma_{n,k_1}}(x_{1:n})}\Big(\partial_{{i,k_2,l_2}}h_{n,k_2}(x_{1:n})\\&-{\hat{H}_{n,k_1}(x_{1:n})}\frac{\partial_{{i,k_2,l_2}}\widehat{\sigma_{n,k_1}}(x_{1:n})}{\widehat{\sigma_{n,k_1}}(x_{1:n})}\Big)
  -\frac{2\lambda_{l_3,k_1}\bar{X}_{l_3,k_1}}{n\widehat{\sigma_{n,k_1}}(x_{1:n})}\bp{\partial_{i,k_{1:2},l_{1:2}}^2\widehat{\sigma_{n,k_1}}(x_{1:n})-\frac{\widehat{\sigma_{n,k_1}}(x_{1:n})\partial_{{i,k_2,l_2}}\widehat{\sigma_{n,k_1}}(x_{1:n})}{\widehat{\sigma_{n,k_1}}(x_{1:n})}}
    \\& -\frac{\hat{H}_{n,k_1}(x_{1:n})}{n\widehat{\sigma_{n,k_1}}(x_{1:n})}\bp{\partial_{i,k_{1:3},l_{1:3}}^3\widehat{\sigma_{n,k_1}}(x_{1:n})-\frac{\partial_{{i,k_1,l_1}}\widehat{\sigma_{n,k_1}}(x_{1:n})\partial_{{i,k_{2:3},l_{2:3}}}^2\widehat{\sigma_{n,k_1}}(x_{1:n})}{\widehat{\sigma_{n,k_1}}(x_{1:n})}}
   \end{split}
\end{equation}

\noindent Therefore there is a constant $K_3$ such that:
\begin{equation}
    \begin{split}&
    \Big\| \sup_{k_{1:3}\le p_n}\sup_{z\in [\bar{X}^n,Z_i^*]\cup[\bar{X}^n,\tilde Y^*_i]} \sum_{l_{1:3}\le l_n} \ba{\partial_{i,k_{1:3},l_{1:3}}^3h_{n,k_1}(Z^{*, i, z}_{1:n}) }\Big\|_{L_{12}}
    \le\frac{ K_3D_n^7}{\max(\lambda_n^4,1)n^3} .
    \end{split}\end{equation}
    By another application of the chain rule we have:
    \begin{equation}
    \begin{split}
        \partial_{i,k_{1:3},l_{1:3}}\omega_{k'}(x_{1:n})=& \beta_n\bp{\mathbb{I}(k'=k_1)-\omega_{k_1}(x_{1:n})}
        \\&\times \Big({\partial_{i,k_{1:3},l_{1:3}}^3h_{n,k_1}(x_{1:n})\omega_{k'}(x_{1:n})+\partial_{i,k_{1:2},l_{1:2}}^2h_{n,k_1}(x_{1:n})\partial_{i,k_3,l_3}\omega_{k'}(x_{1:n})}%
      \\& \quad +{ \partial_{{i,k_1,l_1}}h_{n,k_1}(x_{1:n})\partial^2_{{i,k_{2:3},l_{2:3}}}\omega_{k'}(x_{1:n})+\partial^2_{{i,k_1,l_1},i,k_3,l_3}h_{n,k_1}(x_{1:n})\partial_{{i,k_2,l_2}}\omega_{k'}(x_{1:n})
    }\Big)
     \\-&\beta_n\partial_{i,k_3,l_3}\omega_{k_1}(x_{1:n})\\&\times \bp{ \partial_{i,k_{1:2},l_{1:2}}^2h_{n,k_1}(x_{1:n})\omega_{k'}(x_{1:n})- \partial_{{i,k_1,l_1}}h_{n,k_1}(x_{1:n})\partial_{{i,k_{2},l_{2}}}\omega_{k'}(x_{1:n})} 
        \\-&    \beta_n\omega_{k'}(x_{1:n})\\&\times\Big( \partial^2_{{i,k_1,l_1},i,k_3,l_3}h_{n,k_1}(x_{1:n})\partial_{{i,k_2,l_2}}\omega_{k_1}(x_{1:n})+ \partial_{{i,k_1,l_1}}h_{n,k_1}(x_{1:n})\partial^2_{{i,k_{2:3},l_{2:3}}}\omega_{k_1}(x_{1:n})\Big)
    \end{split}
\end{equation}
as well as:
\begin{equation}
    \begin{split}&
        \partial_{i,k_{1:3},l_{1:3}}^3g_n(x_{1:n})\\=& \frac{2}{{n}}\lambda_{k_1,l_1}\mathbb{I}(k_1=k_2=k_3) \partial_{i,k_3,l_3}\omega_{k_1}(x_{1:n})+\frac{2}{\sqrt{n}}\lambda_{k_1,l_1}(\frac{1}{\sqrt{n}}\sum_{j\le n}x_{j,l,k_1}) \partial_{{i,k_{2:3},l_{2:3}}}\omega_{k_1}(x_{1:n})
        \\&+\frac{2\mathbb{I}(k_1=k_3)}{n}\lambda_{k_1,l_1} \partial_{{i,k_{2},l_{2}}}\omega_{k_1}(x_{1:n})+\frac{2\mathbb{I}(k_2=k_3)}{{n}}\lambda_{k_2,l_2}\partial_{{i,k_1,l_1}}\omega_{k_2}(x_{1:n})
        \\&+\frac{2}{\sqrt{n}}\lambda_{k_2,l_2}(\frac{1}{\sqrt{n}}\sum_{j\le n}x_{j,k_2,l_2})\partial_{{i,k_1,l_1},i,k_3,l_3}^2\omega_{k_2}(x_{1:n})\\&+\frac{2}{\sqrt{n}}\lambda_{k_3,l_3} (\frac{1}{\sqrt{n}}\sum_{j\le n}x_{j,k_3,l})\partial_{i,k_{1:2},l_{1:2}}^2\omega_{k_3}(x_{1:n})
    +\sum_{k'\le p_n}\hat{H}_{n,k'}(x_{1:n})\partial_{i,k_{1:3},l_{1:3}}^3\omega_{k'}(x_{1:n})
    \end{split}
\end{equation}
Therefore we know that there are constants $K_4$ such that:
\begin{equation}
    \begin{split}&
  \max_i  \Big\|\max_{z\in [\bar{X^n},Z^*_{i}]\cup[\bar{X^n},\tilde Y^*_i]} \big\|\partial^2_{x_{1}}g_n( Z^{*,i,z}_{1:n})\big\|_{t,p_n\times l_n} \Big\|_{L_{12}}
 \le \frac{K_4D_n^{10}}{\min(\lambda_n^6,1)n^2}\beta_n^3
    \end{split}\end{equation}

\noindent    Therefore we have $R_{n,3}^*=O(\frac{D_n^{10}\beta_n^3}{\min(\lambda_n^6,1)n})$.

\noindent    This implies that both $(H_0)$ and $(H_1^*)$ hold and \Cref{jardin} guarantees that $$\Big\|d_{\mcF}\Big(n\hat{T}_n( Z^M),g_n( Y^M)\mid X^n\Big)\Big\|_{L_1}\rightarrow 0.$$ 
\end{proof}


\section{Proof of \Cref{common_stacking_msbien}}\label{app:common_stacking_msbien}

We divide the proof of the two statements of \Cref{common_stacking_msbien} in two. Firstly we prove that the bootstrap method is consistent if $\beta_n=o(\sqrt{n-m_n})$.
\begin{proof}
Let $(Y^n_i)$ be an independent copy of $(X^n_i)$.
We shorthand $$\omega_n^p(x_{1:n-m_n}):=
\frac{e^{-\beta_n\mcR^k_n(x_{1:n})}}{\sum_{k'\le p_n}e^{-\beta_n\mcR^{k'}_n(x_{1:n})}},\qquad \hat\theta^p_n:=\hat\theta^p_n(X_{1:m_n}^n);$$
$$\hat\Theta_n:=\hat\Theta_n(X_{1:n}^n),\qquad \hat\Theta^Y_n:=\hat\Theta_n(X_{1:m_n}^nY^{n}_{m_n+1:n}). $$
The first step of the proof is to realize that as conditionally on $\hat\Theta_n$ and  $\hat\Theta_n^Y$  the observations $\bp{\mathcal{L}_n\bp{X^n_{i+m_n},\hat\Theta^Y_n(X^n_{i+m_n})}-\mathcal{L}_n\bp{X^n_{i+m_n},\hat\Theta_n(X^n_{i+m_n})}}_{i\ge 0}$ are independent and identically distributed we have
\begin{equation}
\begin{split}&\label{fin_1}
   \bn{ \sqrt{{\rm var}\Big[\frac{1}{\sqrt{n-m_n}}\sum_{i\le n- m_n}\mathcal{L}_n\bp{X_{i+m_n}^n,~\hat\Theta_n^Y(X^n_{i+m_n})}-\mathcal{L}_n\bp{X^n_{i+m_n},\hat\Theta_n(X^n_{i+m_n})} \Big|\hat\Theta_n^Y,\hat\Theta_n \Big]}}_{L_1}
    \\&=\bn{  \sqrt{{\rm var}\Big[\mathcal{L}_n\bp{X_{m_n+1}^n,~\hat\Theta_n^Y(X^n_{m_n+1})}-\mathcal{L}_n\Big(X^n_{m_n+1},\hat\Theta_n(X^n_{m_n+1})\Big)\Big|\hat\Theta_n,\hat\Theta_n^Y\Big]}}_{L_1}
    \end{split}\end{equation}
    Moreover by exploiting Taylor expansions we know that \begin{equation}
\begin{split}&
\bn{  \sqrt{{\rm var}\Big[\mathcal{L}_n\bp{X_{m_n+1}^n,~\hat\Theta_n^Y(X^n_{m_n+1})}-\mathcal{L}_n\Big(X^n_{m_n+1},\hat\Theta_n(X^n_{m_n+1})\Big)\Big|\hat\Theta_n,\hat\Theta_n^Y\Big]}}_{L_1}
    \\&\le \sum_{l\le d'_n}\Big\|\sup_{\theta \in \Omega(\{\hat\theta^p_n,~p\le p_n\})}\ba{\partial_{2,l}\mathcal{L}_n(X^n_{m_n+1},\theta(X^n_{m_n+1}))}\Big\|_{L_{\infty}} \times\Big\|\hat\Theta_{n,l}(X^n_{m_n+1}) -\hat\Theta_{n,l}^Y(X^n_{m_n+1})\Big\|_{L_2}
    \\&\le L_nd'_n\sup_{l\le d'_n}\Big\|\hat\Theta_{n,l}(X^n_{m_n+1}) -\hat\Theta_{n,l}^Y(X^n_{m_n+1})\Big\|_{L_2}
    \end{split}\end{equation} We define $(X_{i}^{n,j})$ as the following interpolating process: $X_{i}^{n,j}:=\begin{cases}X_i^n~{\rm if}~i\le j\\Y_j^{n}~{\rm otherwise}\end{cases}$ and shorthand $\hat\Theta_{n,l}^j:=\hat\Theta_{n,l}(X_{1:n}^{n,j})$. By the Effron-Stein inequality we have:
    \begin{equation}
\begin{split}\label{fin_2}
    \bn{\hat\Theta_{n,l}(X^n_{m_n+1}) -\hat\Theta_{n,l}^Y(X^n_{m_n+1})}^2_{L_2}
    &\overset{}{\le} \sum_{j\ge m_n+1}\bn{\hat\Theta_{n,l}^j(X^n_{m_n+1})-\hat\Theta_{n,l}^{j-1}(X^n_{m_n+1})}^2_{L_2}        \\&\le (n-m_n)\bn{\hat\Theta_{n,l}^{n}(X^n_{m_n+1})-\hat\Theta_{n,l}^{n-1}(X^n_{m_n+1})}_{L_2}^2
    \end{split}\end{equation}
Moreover using the triangle inequality  and the inequality $(a+b)^2\le 2(a^2+b^2)$ we can upper-bound the right-hand side as
       \begin{equation}
\begin{split}&\label{fin_3}
\bn{\hat\Theta_{n,l}^n(X^n_{m_n+1})-\hat\Theta_{n,l}^{n-1}(X^n_{m_n+1})}^2_{L_2} 
\\&= \bn{\sum_{p\le p_n}\hat\theta^p_{n,l}\bp{\omega_n^p(X^n_{m_n+1:n})-\omega_n^p(X^n_{m_n+1:n-1}Y^n_n)}}^2_{L_2} 
   \\\le&  2\bn{\sum_{p\le p_n}\hat\theta^p_{n,l}\omega_n^p(X^n_{m_n+1:n})\Big[e^{\frac{\beta_n}{n-m_n}[\mathcal{L}_n(X_n^n,\hat\theta_n^p)-\mathcal{L}_n(Y_n^n,\hat\theta_n^p)]}-1\Big]}_{L_2}^2
        \\&~+2\bn{\sum_{p\le p_n}\hat\theta^p_{n,l}\omega_n^p(X^n_{m_n+1:n})\sum_{p'\le p_n}\omega_n^{p'}(X^n_{m_n+1:n})\Big[e^{\frac{\beta_n}{n-m_n}[\mathcal{L}_n(X_n^n,\hat\theta_n^{p'})-\mathcal{L}_n(Y_n^n,\hat\theta_n^{p'})]}-1\Big]}_{L_2}^2
        \\&\le4\frac{\beta_n^2}{(n-m_n)} \|\sup_{p\le p_n}|\hat{\theta}_{n,l}^p|\|^2_{L_{\infty}}\bn{\sup_{p\le p_n}\ba{\mathcal{L}_n(X_1^n,\hat\theta_n^p)}}^2_{L_{\infty}}e^{\frac{2\beta_n}{n-m_n}\bn{\sup_{p\le p_n}\mathcal{L}_n(X_1^n,\hat\theta_n^p)}_{L_{\infty}}}
        \\&\le 4\frac{\beta_n^2}{(n-m_n)} T^2_{n}L^2_n e^{\frac{2\beta_n}{n-m_n}L_n}
    \end{split}\end{equation}
    
    By combining \cref{fin_1}
 and \cref{fin_2}  we therefore obtain that:
 \begin{equation}
\begin{split}&
 \bn{ \sqrt{{\rm var}\Big[\frac{1}{\sqrt{n-m_n}}\sum_{i\le n- m_n}\mathcal{L}_n\bp{X_{i+m_n}^n,~\hat\Theta_n^Y(X^n_{i+m_n})}-\mathcal{L}_n\bp{X^n_{i+m_n},\hat\Theta_n(X^n_{i+m_n})} \Big|\hat\Theta_n^Y,\hat\Theta_n \Big]}}_{L_1}
    \\&\longrightarrow 0.
    \end{split}\end{equation}
Hence by using the fact that $d_{\mcF}$ satisfies the triangle inequality (see \Cref{lem:df-is-metric}) we observe that it is enough to prove that the distribution of\\ $\frac{1}{\sqrt{n}}\sum_{i\le n}\mathcal{L}_n\bp{X^n_{i+m_n},\hat\Theta^Y_n(X^n_{i+m_n})}$ can be correctly approximated by the bootstrap method. We note that the process $\bp{\mathcal{L}_n\bp{X^n_{i+m_n},\hat\Theta^Y_n(X^n_{i+m_n})}}$ is exchangeable and that if we define the function: $g_n(x_1,\dots,x_n)=\frac{1}{\sqrt{n}}\sum_{i\le n}x_i$ then: $$g_n\bp{\bp{\mathcal{L}_n\bp{X^n_{i+m_n},\hat\Theta^Y_n(X^n_{i+m_n})}}_{1:n}}:=\frac{1}{\sqrt{n}}\sum_{i\le n}\mathcal{L}_n\bp{X^n_{i+m_n},\hat\Theta^Y_n(X^n_{i+m_n})}.$$We remark that $(g_n)$ satisfy all the conditions of \Cref{thm8} which implies that the desired result holds. 

\end{proof}
\section{Proof of  \Cref{nulle_mt2}}For simplicity we write
\begin{equation}\begin{split}&T_{l,n}:=\bn{\sup_{p\le p_n}\ba{\hat\theta^p_{n,l}(X_{1:m_n}^n)(Y_1^{n,1})}}_{L_{\infty}}\lor 1,\\&L^{*}_{l_{1:i},n}:=\big\|\max_{\substack{p\le p_n}}\big|\partial^i_{2,l_{1:i}}\mathcal{L}_n(X_n^n,\hat\theta^p_{n}(X^n_{1:m_n})(X_n^n))\big|\big\|_{L_{\infty}}\lor 1.\end{split}\end{equation}
\begin{proof}
For ease of notations, we shorthand:
$$ \hat\Theta_n^{Z^{n,2}}:=\hat\Theta_n(X^n_{1:m_n}Z^{n,2}_{1:n-m_n}),\qquad \hat\theta^p_n:=\hat\theta^p_n(X_{1:m_n}^n).$$ We first notice using \cref{lem:df-is-metric} and \cref{lem:trans_df} that:
 \begin{equation}
     \begin{split}\label{eq:sm_last} &\Big\|d_{\mcF}\Big(\mcR^{{\rm s}}_{\hat\Theta_n^{Z^{n,2}}}(Z_{1:n-m_n}^{n,1})-\E\bp{\mcR^{{\rm s}}_{\hat\Theta_n^{Z^{n,2}}}(Z_{1:n-m_n}^{n,1})\mid\hat\Theta_n^{Z^{n,2}},X^n},\\&\qquad \quad\mcR^{{\rm s}}_{\hat\Theta_n^{'}}(Y_{1:n-m_n}^{n,1})-\E\bp{\mcR^{{\rm s}}_{\hat\Theta_n^{'}}(Y_{1:n-m_n}^{n,1})\mid\hat\Theta_n^{'}}\mid X^n\Big)\Big\|_{L_1}
     \\&\le \Big\|d_{\mcF}\Big(\mcR^{{\rm s}}_{\hat\Theta_n^{Z^{n,2}}}(Z_{1:n-m_n}^{n,1})-\mathbb{E}\big(\mcR^{{\rm s}}_{\hat\Theta_n^{Z^{n,2}}}(Z_{1:n-m_n}^{n,1})\mid\hat\Theta_n^{Z^{n,2}},X^n\big) \\&~\qquad\qquad ,\mcR^{{\rm s}}_{\hat\Theta_n^{Z^{n,2}}}(Y_{1:n-m_n}^{n,1})-\mathbb{E}\big(\mcR^{{\rm s}}_{\hat\Theta_n^{Z^{n,2}}}(Y_{1:n-m_n}^{n,1})\mid\hat\Theta_n^{Z^{n,2}}\big)\Big)\Big\|_{L_1}
     \\&+\Big\|d_{\mcF}\Big(\mcR^{{\rm s}}_{\hat\Theta_n^{Z^{n,2}}}(Y_{1:n-m_n}^{n,1})-\E\bp{\mcR^{{\rm s}}_{\hat\Theta_n^{Z^{n,2}}}(Y_{1:n-m_n}^{n,1})\mid\hat\Theta_n^{Z^{n,2}}},\\&\qquad\qquad  \mcR^{{\rm s}}_{\hat\Theta_n^{'}}(Y_{1:n-m_n}^{n,1})-\E\bp{\mcR^{{\rm s}}_{\hat\Theta_n^{'}}(Y_{1:n-m_n}^{n,1})\mid\hat\Theta_n^{'}}\mid X^n\Big)\Big\|_{L_1}
     \end{split}
 \end{equation} 
 We upper-bound each term seperately. In this goal, we define as $(X'_i)$ the process defined as $X'_i=\mathcal{L}_n(X_{i+m_n}^n,\hat\Theta^{Z^{n,2}}_n).$ We remark that the sequence $(X'_i)$ is an exchangeable sequence and we define: $$Z'_i:=\mathcal{L}_n(Z^{n,1}_i,\hat\Theta_n^{Z,n2}),\quad Y'_i:=\mathcal{L}_n(Y^{n,1}_i,\hat\Theta_n^{Z,n2}). $$ We note that  $(Z'_i)$ and $(Y'_i)$ respectively form a bootstrap sample of $(X'_i)_{i\le n-m_n}$ and a copy of  $(X'_i)$ that is conditionally on $\hat \Theta_n^{Z^{n,2}}$ independent.
~We define $F_{1,n}$ to be the following functions: 
 $$F_{1,n}(x_{1:n-m_n}):=\frac{1}{\sqrt{n-m_n}}\sum_{i\le n-m_n}x_i.$$

\noindent Moreover we denote $(X^{*}_i)$ the following random vectors $X^{*}_i=\Big(\mathcal{L}_n(X_{i+m_n}^{n},\hat\theta^p_n\Big)_{p\le p_n}.$ We also define:$$Z^{*}_i=\Big(\mathcal{L}_n(Z_{i+m_n}^{n,2},\hat\theta^p_n\Big)_{p\le p_n},\quad Z^{*}_i=\Big(\mathcal{L}_n(Y_{i+m_n}^{n,2},\hat\theta^p_n\Big)_{p\le p_n}.$$ We remark that $(Z^{*}_i)$ and $(Y^{*}_i)$ are respectively bootstrap samples and (conditionally) independent copy of $(X^{*}_i)$. 

\noindent We shorthand $$ R_n^p:=\E(\mathcal{L}_n(X^n_p,\hat\theta^p_n)|\hat\theta^p_n);\qquad\bar{X'_p}:=\frac{1}{n-m_n}\sum_{i\le n-m_n}X'_{i,p}.$$ 

\noindent We define the following weight functions: $\omega_n^p:\times_{i=1}^{n-m_n}\RR\rightarrow \RR$ as: $$\omega_n^p(x_{1:n-m_n}):=
\frac{e^{ \frac{-\beta_n}{n-m_n}\sum_{i\le n-m_n}x_{i,p}}}{\sum_{p'\le p_n}e^{\frac{-\beta_n}{n-m_n}\sum_{i\le n-m_n}x_{i,p'}}}.$$ We define
$F_{2,n}:\times_{l=1}^n\RR^{p_n}\rightarrow\RR$ as the following random function: \begin{equation}\begin{split}F_{2,n}(x_{1:n-m_n}):=\frac{1}{\sqrt{n-m_n}}\sum_{i\le n-m_n}&\mathcal{L}_n\bp{Y^{n,1}_i,\sum_{p\le p_n}\hat\theta^p_n(Y_{i}^{n,1})\omega_n^p(x_{1:n-m_n})}\\&-\E\bp{\mathcal{L}_n\bp{Y^{n,1}_i,\sum_{p\le p_n}\hat\theta^p_n(Y_{i}^{n,1})\omega_n^p(x_{1:n-m_n})}\mid X^n}.
\end{split}\end{equation}We remark that, conditionally on $X^n$ and $Y^{n,1}$, the functions $(F_{2,n})$ are independent from $(Z^*_i)$ and $(Y^*_i)$. 
Using \cref{eq:sm_last} we note that if  
the following hold then the desired result also holds:
\begin{equation}
\begin{split}
\bn{d_{\mcF}\bp{F_{2,n}(Z^{*}_{1:n-m_n}),F_{2,n}(\tilde Y^{*}_{1:n-m_n}) |X^n}}_{L_1}\rightarrow 0;
    \end{split}\end{equation}
  \begin{equation}
\begin{split}
   \bn{d_{\mcF}\bp{F_{1,n}(Z'_{1:n-m_n}),F_{1,n}(\tilde Y'_{1:n-m_n}) \mid X^n}}_{L_1}\rightarrow 0.
    \end{split}\end{equation}
   To prove that those hold we use \Cref{jardin} extended to exchangeable sequences and random functions. We notice that the functions $(F_{2,n})$ and $(F_{1,n})$ are three times differentiable (random) functions. Therefore for both functions hypothesis $(H_0)$ holds and to get the desired results we only need to check that hypothesis $(H_1^*)$ also holds. For simplicity for all $i\le 3$ we write  \begin{equation}\begin{split}\hat{R}_{j}^{n,l_{1:i}}(x_{1:n}):=\frac{1}{\sqrt{n-m_n}}\Big(& \partial_{2,l_{1:i}}^i\mathcal{L}_n\Big({Y^{n,1}_j,\sum_{p\le p_n}\hat\theta^p_n(Y_j^{n,1})\omega_n^p(x_{1:n-m_n})}\Big)\\&-\E\Big({\partial_{2,l_{1:i}}^i\mathcal{L}_n\big({Y^{n,1}_j,\sum_{p\le p_n}\hat\theta^p_n(Y_j^{n,1})\omega_n^p(x_{1:n-m_n})}\big)\mid X^n}\Big)\Big).\end{split}\end{equation}  We define the following random functions: $f_{l_{1:3},x}(z):=\partial_{2,l_{1:3}}^3\mathcal{L}_n\bp{z,\sum_{p\le p_n}\hat{\theta}_n^p(z)\omega_p( xZ^{*}_{2:n-m_n}))};$ 
and remark that the set of functions $\{f_{l_{1:3},x},~l_{1:3}\le d'_n, x\in [0, Z^{*}_1]\bigcup [0,  Y^{*}_1]\}$  has an epsilon covering number proportional to $\big(\sum_{l\le n}L^*_{l,n}T_{l,n}\beta_n/((n-m_n)\epsilon)\big)^{p_n}p_n$. This implies that there is a constant $K$ such that  we have:
   \begin{equation}\begin{split}\label{train}\max_{i\le n}\bn{\sup_{x\in [\bar{X}^{n,*}, Z^{*}_1]\bigcup [\bar{X}^{n,*}, \tilde Y^{*}_1]}\ba{\sum_{j\le m_n}\hat{R}_{j}^{n,l_{1:3}}(Z^{*,i,x}_{1:n})}}_{L_{72}}\le {K\log(p_n)}L^*_{n,l_{1:3}}.\end{split}\end{equation}
   Moreover we note that for all $k\le 2$ there is a constant $C$ such that
      \begin{equation}\begin{split}\label{train_233}\max_{i\le n}\bn{\ba{\sum_{j\le m_n}\hat{R}_{j}^{n,l_{1:k}}(Z^{*,i,\overline{Z^*}}_{1:n})}}_{L_{72}}\le {C}L^*_{n,l_{1:k}}.\end{split}\end{equation}
   For ease of notations we write: $\overline{\theta^p_{n,l}}(x_{1:n-m_n},y):=\hat{\theta}_{n,l}^p(y)-
\sum_{p'\le p_n}\hat{\theta}_{n,l}^{p'}(y)\omega_{n}^{p'}(x_{1:n-m_n}).$
   We remark that  for all $p\le p_n$ we have
    \begin{equation}
\begin{split}&
    \partial_{i,p}F_{2,n}(x_{1:n-m_n})\\=&-\frac{\beta_n}{n-m_n}
\omega^p_n(x_{1:n-m_n})\sum_{l\le d'_n}\sum_{j\le m_n}\hat{R}_{l}^{n,j}(x_{1:n})\Big[\hat{\theta}_{n,l}^p(Y_j^{n,1})-
\sum_{p'\le p_n}\sum_{j\le m_n}\hat{\theta}_{n,l}^{p'}(Y_j^{n,1})\omega_{n}^{p'}(x_{1:n-m_n})\Big]
\\=&-\frac{\beta_n}{n-m_n}
\omega^p_n(x_{1:n-m_n})\sum_{l\le d'_n}\sum_{j\le m_n}\hat{R}_{l}^{n,j}(x_{1:n})\overline{\theta^p_{n,l}}(x_{1:n-m_n},Y_j^{n,1})
\end{split}\end{equation}
Therefore as $\sum_{p\le p_n}\omega_n^p(x_{1:n})=1$, using \cref{train_233} we obtain that there is a constant $C$ that does not depend on $n$ such that   \begin{equation}
\begin{split}&
  \max_{i\le n} \bn{ \bn{\partial_{i}F_{2,n}(Z^{*,i,\overline{X^*}}_{1:n-m_n})}_{v,p_n}}_{L_{12}}
    \le \frac{C\beta_n}{n-m_n}\sum_{l\le d'_n}L^*_{l,n}
\Big\|\max_{p\le p_n}\big|\hat{\theta}_{n,l}^p(Y_1^{n,1})\big|\Big\|_{L_{12}}.
\end{split}\end{equation}
Moreover for all $p\le p_n$ we have $\partial_{i}F_{1,n}(x_{1:n-m_n})=\frac{1}{\sqrt{n-m_n}}.$
Therefore we obtain that\begin{equation}\begin{split}&
  \max_i \Big\| \partial_{i}F_{1,n}(Z^{',i,\overline{X'}}_{1:n-m_n})\Big\|_{L_{12}}
    \le \frac{1}{\sqrt{n-m_n}}.
\end{split}\end{equation} In addition, we remark that we have $\partial_{i}^2F_{1,n}(x_{1:n-m_n})=0$. Therefore the condition $(H_1^*)$ holds for $F_{1,n}$. For ease of notations we write $\omega^{p_1,p_2}_n(x_{1:n-m_n}):=\mathbb{I}(p_1=p_2)-\omega_{n}^{p_2}(x_{1:n-m_n}).$ We have
 \begin{equation}
\begin{split}&
    \partial_{i,p_1,p_2}^2F_{2,n}(x_{1:n-m_n})
    \\&=\frac{\beta_n^2\omega_n^{p_1}(x_{1:n-m_n})}{(n-m_n)^2}
\omega^{p_1,p_2}_n(x_{1:n-m_n})\sum_{l\le d'_n}\sum_{j\le m_n}\hat{R}_{l}^{n,j}(x_{1:n})\overline{\theta^{p_1}_{n,l}}(x_{1:n},Y_j^{n,1})
\\&+\frac{\beta_n^2\omega_n^{p_1}(x_{1:n-m_n})\omega_n^{p_2}(x_{1:n-m_n})}{(n-m_n)^2}
\sum_{l_1,l_2\le d'_n}\sum_{j\le m_n}\hat{R}_{l_1,l_2}^{n,j}(x_{1:n})\overline{\theta^{p_1}_{n,l_1}}(x_{1:n},Y_j^{n,1})\overline{\theta^{p_2}_{n,l_2}}(x_{1:n},Y_j^{n,1})
\\&-\frac{\beta_n^2}{(n-m_n)^2}\omega_{n}^{p_1}(x_{1:n-m_n})\omega_{n}^{p_2}(x_{1:n-m_n})\sum_{l\le d'_n}\sum_{j\le m_n}\hat{R}^{n,j}_l(x_{1:n})\overline{\theta^{p_2}_{n,l}}(x_{1:n},Y_j^{n,1}).
\end{split}\end{equation}
This combined with \cref{train_233} implies that there is a constant $C_2$ such that:
 \begin{equation}
\begin{split}&
 \max_i \bn{ \bn{\partial_{i}^2F_{2,n}(Z^{*,i,\overline{X^*}}_{1:n-m_n})}_{m,p_n}}_{L_{12}}
\le \frac{C_2\beta_n^2}{(n-m_n)^2}
\Big[\sum_{l\le d'_n}L^*_{l,n}T_{l,n}+\sum_{l\le d'_n}L^*_{l_1,l_2,n}T_{l_1,n}T_{l_2,n}\Big]
\end{split}\end{equation}

\vspace{3mm}

\noindent Finally using the chain rule for all $p_{1},p_2,p_3\le p_n$ we have 

\vspace{1.5mm}

 \begin{equation}
\begin{split}&
    \partial_{i,p_{1:3}}^3F_{2,n}(x_{1:n-m_n})
    \\&=\frac{\beta_n^3\sum_{l\le d'_n}\sum_{j\le m_n}\hat{R}_{l}^{n,j}(x_{1:n})\overline{\theta^{p_1}_{n,l_1}}(x_{1:n},Y_j^{n,1})\omega_n^{p_1}(x_{1:n-m_n})}{(n-m_n)^3} \Big[\omega_n^{p_2}(x_{1:n-m_n})\omega^{p_2,p_3}_n(x_{1:n-m_n})\\&-\omega_{n}^{p'}(x_{1:n-m_n}) \omega^{p_1,p_3}_n(x_{1:n-m_n})\Big]
-\frac{\beta_n^3\omega_n^{p_1}(x_{1:n-m_n})}{(n-m_n)^3}
\sum_{l_1,l_2\le d'_n}\sum_{j\le m_n}\hat{R}_{l_1,l_2}^{n,j}(x_{1:n})\overline{\theta^{p_1}_{n,l_1}}(x_{1:n},Y_j^{n,1})
 \\&\quad \times\overline{\theta^{p_3}_{n,l_3}}(x_{1:n},Y_j^{n,1})\Big[\omega_n^{p_3}(x_{1:n-m_n})\omega_n^{p_1,p_2}(x_{1:n-m_n})+\omega_n^{p_3}(x_{1:n-m_n})\omega_n^{p_1,p_2}(x_{1:n-m_n})\Big]\\&-\frac{\beta_n^3\prod_{i=1}^3\omega_n^{p_i}(x_{1:n-m_n})}{(n-m_n)^3}
\Big[\sum_{l_{1:3}\le d'_n}\sum_{j\le p_n}\hat{R}_{l_1,l_2,l_3}^{n,j}(x_{1:n})\prod_{k=1}^3\overline{\theta^{p_k}_{n,l_k}}(x_{1:n},Y_j^{n,1})-\sum_{l_{1:2}\le d'_n}\hat{R}_{l_1,l_2}^{n,j}(x_{1:n},Y_j^{n,1})\\&\quad \times\overline{\theta^{p_2}_{n,l_1}}(x_{1:n},Y_j^{n,1})\overline{\theta^{p_3}_{n,l_2}}(x_{1:n},Y_j^{n,1})\Big]
+\frac{\beta_n^3\omega_n^{p_1}(x_{1:n-m_n})\omega^{p_2}_n(x_{1:n})\big[\omega_{n}^{p_1,p_3}(x_{1:n-m_n})+\omega_{n}^{p_2,p_3}(x_{1:n-m_n})\big]}{(n-m_n)^3} 
\\&\quad \times \sum_{l\le d'_n}\sum_{j\le m_n}\hat{R}_{l}^{n,j}(x_{1:n})\overline{\theta^{p_2}_{n,l_1}}(x_{1:n},Y_j^{n,1})
\end{split}\end{equation}
Therefore using \cref{train} we establish that that there is a constant $C_3<\infty$ such that 
\begin{equation}
\begin{split}&
  \max_{i\le n} \bn{\max_{x\in [\bar{X}^n, Z^{*}_1]\bigcup [\bar{X}^n, \tilde Y^{*}_1]} \bn{\partial_{i}^3F_{2,n}(Z^{*,i,x}_{1:n-m_n})}_{t,p_n}}_{L_{12}}
\\&\le \frac{C_3\beta_n^3\log(p_n)}{(n-m_n)^3}
\sum_{l_1\le d'_n}T_{l_1,n}\Big[L^*_{l_1,n}+\sum_{l_{2}\le d'_n}L^*_{l_{1:2},n}T_{l_2,n}+\sum_{l_{2:3}\le d'_n}L^*_{l_{1:3},n}T_{l_2,n}T_{l_3,n}\Big]
\end{split}\end{equation}
This implies that hypothesis $(H_1^*)$ holds for $F_{2,n}$; which means that 
\begin{equation}
\begin{split}
\bn{d_{\mcF}\bp{F_{2,n}(Z^{*}_{1:n-m_n}),F_{2,n}(\tilde Y^{*}_{1:n-m_n}) |X^n}}_{L_1}\rightarrow 0;
    \end{split}\end{equation}We note that $F_{2,n}(\tilde Y^{*}_{1:n-m_n}) :=\mathcal{R}^s_{\hat\Theta'_n}(Y_{1:n-m_n}^{n,1}).$
   \\\noindent Moreover, we remark that we have proved that the condition $(H_1^*)$ also holds for $F_{1,n}$. This implies that we have
  \begin{equation}
\begin{split}
   \bn{d_{\mcF}\bp{F_{1,n}(Z'_{1:n-m_n}),F_{1,n}(\tilde Y'_{1:n-m_n}) \mid X^n}}_{L_1}\rightarrow 0.
    \end{split}\end{equation}
    As $F_{1,n}$ is linear this gives us the desired result.
%

\end{proof}

\section{Proofs from \Cref{app:preliminaries}}

\subsection{Proof of \Cref{casser}}\label{app:casser}
\begin{proof}
The first step is to bound $\E\Big[\max_{k\le p_n}\frac{1}{\sqrt{n}}\sum_{i\le n}\tilde X^n_{i,k}\Big]$. 
Define $(f_n:\RR^d\rightarrow\RR)$ to be the following sequence of functions:
$$f_n(x_{1:n}):= \frac{\log\Big(\sum_{l\le p_n}e^{\log(p_n)\frac{1}{\sqrt{n}}\sum_{i\le n} x_{i,l}}\Big)}{\log(p_n)}.$$We observe that
 $\|f_n(\tilde X^n)-\max_{k\le p_n}\frac{1}{\sqrt{n}}\sum_{i\le n}\tilde X^n_{i,k}\|_{L_1}\le 1$; and that the functions $(f_n)$ are infinitely differentiable.
We denote,  for all $z\in \RR$, $(\tilde X^{i,n}_{l})$ and  $(\tilde X^{i,z}_{l})$ the processes respecting $$\tilde X^{i,n}_l:=\begin{cases}\tilde X^n_l~{\rm if}~i\le l\\ 0~{\rm otherwise}.\end{cases}\qquad \tilde X^{i,z}_l:=\begin{cases}\tilde X^{i,n}_l~{\rm if}~i\ne l\\ z~{\rm otherwise}\end{cases}.$$ 
We first suppose that $\|\max_{k\le p_n}\tilde X^n_{1,k}\|_{L_p}\le 1$.
Using the Taylor expansion we have: 

\vspace{2mm}

 \begin{equation}
    \begin{split}\label{idaho}
      \Big| \E \Big[f_n(\tilde X^n)\Big]\Big|&\le \sum_{i\le n}  \Big| \E \Big[f_n(\tilde X^{i,n}_{1:n})-f_n(\tilde X^{i-1,n}_{1:n})\Big]\Big|
      \\ &\le \sum_{i\le n}\Big|\E\Big[~(\tilde X^n_i)^{\top}~\partial_if_n(\tilde X^{i,0}_{1:n}) \Big]\Big|
      + \sum_{i\le n}\Big|\E\Big[(\tilde X^n_i)^{\top}\partial^2_if_n(\tilde X^{i,0}_{1:n})\tilde X^n_i\Big]\Big|
       \\&+ \sum_{i\le n}\E\Big[\sup_{z\in [0,\tilde X^n_i]}\Big<\Big|\partial^3_if_n(\tilde X^{i,z}_{1:n})\Big|,\Big|\tilde X^{\otimes 3}_i\Big|\Big>\Big]
       \\&\overset{(a)}{\le} n \Big|\sum_{d_1,d_2\le p_n}\E\Big[\partial^2_{i,d_1,d_2}F^{n}_{\beta}(\tilde X^{i,0}_{1:n})\Big]\E(\tilde X^n_{1,d_1}\tilde X^n_{1,d_2})\Big|
      \\&+ n\sum_{d_1,d_2,d_2\le n}\E\Big[\sup_{z\in [0,\tilde X^n_i]}\Big|\partial^3_{i,d_{1:3}}F^{n}_{\beta}(\tilde X^{i,z}_{1:n})\Big|~\Big|\tilde X^n_{1,d_1}\tilde X^n_{1,d_2}\tilde X^n_{1,d_3}\Big|\Big]
    \end{split}
\end{equation} where to get (a) we used the fact that as $\sigma(\tilde X^{i,0}_{1:n})=\sigma(\tilde X_0^n,\dots, \tilde X_{i-1}^n)$ we have $\E(\tilde X^n_i|\tilde X^{i,0}_{1:n})=0$. The next step is to upper-hand the right-hand side of \Cref{idaho}. For ease of notations we write: $$\omega_k(x_{1:n}):=\frac{e^{\log(p_n)\frac{1}{\sqrt{n}}\sum_{i\le n} x_{i,k}}}{\sum_{l\le p_n}e^{\log(p_n)\frac{1}{\sqrt{n}}\sum_{i\le n} x_{i,l}}}.$$

\vspace{1.5mm}

\noindent For all $k_1,k_2\le p_n$ we obtain by the chain rule that
       \begin{equation}\begin{split}&
       \Big| \partial_{i,k_1,k_2}^2f_n(x_{1:n})\Big|
       \le \begin{cases}\frac{\log(p_n)}{n}\omega_{k_1}(x_{1:n})\omega_{k_2}(x_{1:n})\qquad{\rm if}~k_1\ne k_2\\\frac{\log(p_n)}{n}\Big[ \omega_{k_1}^2(x_{1:n})+\omega_{k_1}(x_{1:n})\Big]\qquad{\rm if}~k_1=k_2.\end{cases}
    \end{split}
\end{equation}

\vspace{2mm}

\noindent As $\sum_{k\le p_n}\omega_k(x_{1:n})=1$, this implies that
\begin{equation}
    \begin{split}&
       n \Big|\sum_{d_1,d_2\le p_n}\E\Big[\partial^2_{i,d_1,d_2}F_{\beta}(\tilde X^{i,0})\Big]\E(\tilde X^n_{1,d_1}\tilde X^n_{1,d_2})\Big|
       \le  {2 \log(p_n)}.\end{split}\end{equation}
We now bound the second term of the right hand side of \cref{wed_dr}. For all  integers $k_1,k_2,k_3\le p_n$, by the chain rule, we obtain that:
       \begin{equation}\begin{split}\label{bj_1}&
       \ba{ \partial_{i,k_{1:3}}^3F_{\beta}(x'_{1:n})}
       \\&\le\quad \frac{\log(p_n)^{3/2}}{n^{3/2}} \omega_{k_1}(x_{1:n})\omega_{k_2}(x_{1:n})\omega_{k_3}(x_{1:n})\bp{1+\mathbb{I}(k_1=k_3)+\mathbb{I}(k_2=k_3)}
       \\&+\frac{\log(p_n)^{3/2}}{n^{3/2}} \omega_{k_1}(x_{1:n})\omega_{k_2}(x_{1:n})\bp{\mathbb{I}(k_1=k_3)+\mathbb{I}(k_2=k_3)}\\&+\frac{\log(p_n)^{3/2}}{n^{3/2}} \omega_{k_1}(x_{1:n})\bp{\omega_{k_3}(x_{1:n})\mathbb{I}(k_1=k_2)
     +\mathbb{I}(k_1=k_2=k_3)}
    \end{split}
\end{equation}
Therefore we obtain that
\begin{equation}\begin{split}&
       n\sum_{d_1,d_2,d_3\le n}\E\bp{\sup_{z\in [0,\tilde X^n_i]}\ba{\partial^3_{i,d_{1:3}}F^{\beta,n}(\tilde X^{i,z})}~\ba{\tilde X^n_{1,d_1}\tilde X^n_{1,d_2}\tilde X^n_{1,d_3}}}
       \le \frac{7\beta^2\log(p_n)^2}{\sqrt{n}}.
    \end{split}
\end{equation}
Hence using \cref{idaho} we establish that:
\begin{equation}\begin{split}
        \ba{ \E\bp{\max_{k\le p_n}\frac{1}{\sqrt{n}}\sum_{i\le n}\tilde X^n_{i,k}}}
&\le 1+ \Big[{\log(p_n)} +\frac{7\log(p_n)^2}{6\sqrt{n}}\Big] . \end{split}
\end{equation}
By potentially renormalizing $\|\max_{k\le p_n}\ba{\tilde X_{i,k}^n}\|_{L_p}$ we obtain therefore than in general we have: \begin{equation}\begin{split}
        &\ba{ \E\bp{\max_{k\le p_n}\frac{1}{\sqrt{n}}\sum_{i\le n}\tilde X^n_{i,k}}}
\\&\le \bn{\max_{k\le p_n}|X_{1,k}|}_{L_p}\bp{1+ {\log(p_n)} +\frac{6\log(p_n)^2}{\sqrt{n}}}. \end{split}
\end{equation}
Lastly, according to the  Rosenthal inequality for martingales \cite{hitczenko1990best}, there are constants $(C_p)$ that do not depend on $d$ or $(\tilde X^n_i)$ such that
\begin{equation}\begin{split}&
    \bn{\max_{k\le p_n}\frac{1}{\sqrt{n}}\sum_{i\le n}\tilde X^n_{i,k}-\E\bp{\max_{k\le p_n}\frac{1}{\sqrt{n}}\sum_{i\le n}\tilde X^n_{i,k}}}_{L_p}
    \le 2 C_p\bn{\sup_{d_1\le p_n}\ba{\tilde X^n_{i,d_1}}}_{L_p}.
    \end{split}
\end{equation}
Therefore we get the desired results using the triangular inequality.

\noindent Finally, let $(\mathbb{F}_i)$ designates the filtration $\mathbb{F}_i:=\sigma\Big( X^n_1,\dots, X^n_i\Big)$ then the following is an array of martingale differences $\bp{\E(g_{k,n}( X^n)|\mathbb{F}_i)-\E(g_{k,n}( X^n)|\mathbb{F}_{i-1})}$. The last point of \Cref{casser} follows directly from this observation.
\end{proof}

\end{document}